\documentclass[a4paper,11pt]{article}
\usepackage[T1]{fontenc}
\usepackage[utf8]{inputenc}
\usepackage{amssymb}
\usepackage{dsfont}
\usepackage{tikz-cd}
\usepackage{faktor}
\usepackage{amsthm}
\usepackage{amsmath}
\usepackage{enumerate}
\usepackage{amsthm}
\usepackage{array}
\usepackage{stmaryrd}
\usepackage{fullpage}
\usepackage{url}
\usepackage[backref]{hyperref}
\usepackage{mathtools}
\usepackage{cleveref}
\usepackage{mathrsfs}

\newcommand{\Z}{\mathbb{Z}}
\newcommand{\Q}{\mathbb{Q}}
\newcommand{\R}{\mathbb{R}}
\newcommand{\N}{\mathbb{N}}

\newcommand{\abs}[1]{\left\lvert #1\right\rvert}

\newcommand{\inv}{^{-1}}
\newcommand{\Lzero}{\mathrm{L}^{0}}
\newcommand{\Sinf}{\mathfrak{S}_{\infty}}
\newcommand{\Aut}{\mathrm{Aut}}
\newcommand{\Autf}{\mathrm{Aut}_f}
\newcommand{\MAlg}{\mathrm{MAlg}}
\newcommand{\MAlgf}{\mathrm{MAlg}_f}
\newcommand{\Leb}{\mathrm{Leb}}
\newcommand{\id}{\mathrm{id}}
\DeclareMathOperator{\supp}{\mathrm{supp}}

\title{Polish full groups preserving an infinite measure}
\author{Fabien \textsc{Hoareau}}

\newtheorem{thm}{Theorem}[section]
\newtheorem{cor}[thm]{Corollary}
\newtheorem{lem}[thm]{Lemma}
\newtheorem{prop}[thm]{Proposition}

\newtheorem{thmi}{Theorem}				% pour avoir des lettres 
\theoremstyle{definition}

\newtheorem{claim}{Claim}[subsection]
\newtheorem*{claim2}{Claim}

\newtheorem{defi}[thm]{Definition}
\newtheorem{ex}[thm]{Example}
\newtheorem{rem}[thm]{Remark}
\newtheorem{nota}[thm]{Notation}

\newtheorem*{question*}{Question}
\newtheorem*{ack}{Acknowledgements}

\begin{document}

\maketitle

\begin{abstract}
   We extend some results of Carderi and Le Maître on full groups in the probability context to the infinite measure one: there exists at most one Polish group topology (refining the weak topology and coarser than the uniform topology) on an ergodic full group, and the orbit full group of a locally compact group acting in a Borel manner can be endowed with a Polish group topology. Moreover, orbit full groups are complete invariants for orbit equivalence. 
   We then generalize a result from Le Maître on non-Polishability of the group of finitely supported bijections: the finitely supported elements of an ergodic full group carries a Polish group topology if and only if the associated full group comes from a countable equivalence relation. We finish with algebraic and topological results on ergodic (orbit) full groups concerning normal subgroups, contractibility and genericity of aperiodic elements.
\end{abstract}

\tableofcontents

\section{Introduction}

In his pioneering work (\cite{Dye1959}, \cite{Dye1963}), Dye introduced the full groups as subgroups of the group $\Aut(X,\mu)$ of probability measure-preserving (pmp) bijections of a standard probabilty space $(X,\mu)$, where two such bijections are identified if they coincide on a conull set. This definition extends to the broader setup of non-singular Borel bijections of $(X,\mu)$, \textit{i.e.} bijections $T: X \to X$ such that $T_\ast \mu$ is equivalent to $\mu$. We denote by $\Aut(X,[\mu])$ the group they form.

\begin{defi}
	A \textbf{full group} $\mathbb{G}$ is a subgroup of $\Aut(X,[\mu])$ which is stable under \textbf{cutting and pasting}, where a bijection $T$ is obtained by cutting and pasting $(T_n)$ if there exists a countable partition $(A_n)$ of the space such that $T_{\restriction A_n} = T_{n \restriction A_n}$ for any $n \in \N$.
\end{defi}

At this level of generality, countably generated full groups were first studied by Krieger in \cite{Krieger1969}.
The study of these groups is enriched when we see them as Polish groups, and as such we recall the following definitions.

\begin{defi}
The \textbf{uniform topology} on $\Aut(X,[\mu])$ is the group topology induced by the \textbf{uniform metric} $d_\mu$ defined by $d_\mu(S,T) = \mu(\left\{ x \in X \mid T(x) \neq S(x) \right\})$. We also define the \textbf{bi-uniform metric} $\widetilde{d}$ by $\widetilde{d}(S,T) = d_\mu(S,T) + d_\mu(S\inv, T\inv)$.
\end{defi}

The group $\Aut(X,[\mu])$ is complete for the bi-uniform metric. Moreover it can be shown that a full group is separable for the uniform topology if and only if it is the full group of a countable equivalence relation (see \Cref{defi: eqrel fg}), in which case it is a Polish group.

In the more specific context of pmp bijections, Carderi and Le Maître took a step out of the countable world and developped in \cite{CarderiLM2016} a framework providing many new Polish full groups.
Our first order of business is to generalize this construction to the setup of bijections preserving a $\sigma$-finite infinite measure. The group they form is denoted by $\Aut(X,\lambda)$, where $(X,\lambda)$ is a standard $\sigma$-finite space. We give the definition of orbit full groups in this context.

\begin{defi}
	Let $G$ be a Polish group acting on a standard $\sigma$-finite space $(X, \lambda)$ in a Borel manner. The \textbf{orbit full group} $[\mathcal{R}_G]$ is defined by \[[\mathcal{R}_G] = \left\{ T \in \Aut(X,\lambda) \mid \forall x \in X : T(x) \in G \cdot x \right\}.\]
\end{defi}

If $(X,\lambda)$ is a standard probability space instead, we recover the definition from \cite{CarderiLM2016}.
However, from a dynamical point of view, the world of $\sigma$-finite infinite measures is very different from the world of finite measures. Indeed, the phenomenon of dissipativity appears, examplified by the map $x \mapsto x+1$ on $(\R, \Leb)$. The space is cleaved by the \textit{Hopf decomposition} into the conservative (or recurrent) and dissipative parts, and even for well-understood systems such as odometers, Eigen, Hajian and Halverson proved in \cite{EigenHajianHalverson1998} that recurrence behaves differently: odometers can fail to be multiply recurrent. Furthermore, the group $\Aut(X,\lambda)$ is not a simple group, as opposed to its probability counterpart $\Aut(X,\mu)$ (see \cite{Fathi1978}). Our first result is the following 
(for the definition of the \textit{topology of orbital convergence in measure}, see \Cref{section: Polish structures on orbit full groups}).

\begin{thmi}
	[see Thm.\ \ref{Thm: orbit full groups are Polish groups LC}]
	{\label{thmi: orbit fg of G lcsc}}
	let $G$ be a locally compact Polish group acting in a measure-preserving manner on a standard $\sigma$-finite space $(X,\lambda)$. The orbit full group $[\mathcal{R}_G]$ endowed with the topology of orbital convergence in measure is a Polish group.
\end{thmi}

Our result is actually more general than this, and expands on a subtle point specific to infinite measures, namely local finiteness. Recall that a Borel measure on a topological space is \textbf{locally finite} when every point admits a neighbourhood of finite measure. Under this additional condition, we obtain the conclusion of \Cref{thmi: orbit fg of G lcsc} for continuous actions of general Polish groups. In fact this conclusion may fail to hold when the measure is not locally finite and the acting group is not locally compact (see \Cref{ex: action of Sinf}).

\begin{thmi}
	[see Thm.\ \ref{Thm: orbit full groups are Polish groups}, Cor.\ \ref{essfree}]
	{\label{thmi: Polish topology on fg}}
	Let $G$ be a Polish group acting in a continuous manner on a Polish space $(X,\tau_X)$ equipped with an atomless $\sigma$-finite measure $\lambda$ which is locally finite on $\tau_X$. The orbit full group $[\mathcal{R}_G]$, equipped with the topology of orbital convergence in measure, is a Polish group. Moreover, if the action is measure-preserving and essentially free, $G$ topologically embeds in $[\mathcal{R}_G]$.
\end{thmi}

\Cref{thmi: orbit fg of G lcsc} follows from \Cref{thmi: Polish topology on fg} via \cite[Thm.~4.4]{HoareauLM2024}, which allows us to turn any Borel infinite measure-preserving action of a locally compact Polish group into a continuous one, where the measure is locally finite.

We now restrict our study to the much more rigid class of ergodic full groups.
Recall that a subgroup $\mathbb{G}$ of $\Aut(X,\lambda)$ is \textbf{ergodic} if for every $A \subseteq X$ satisfying $\lambda(T(A) \Delta A) = 0$ for every $T \in \mathbb{G}$, we have that $A$ is either null or conull. Moreover, any measure-preserving action of a Polish group $G$ yields a group homomorphism $\pi : G  \to \Aut(X,\lambda)$, and we say that the $G$-action is ergodic when $\pi(G)$ is ergodic as a subgroup of $\Aut(X,\lambda)$.
Our next result is a reconstruction-type theorem in the vein of \cite[Thm.~2]{Dye1963}, building upon Fremlin's generalization to groups with \textit{many involutions} (see \Cref{thmFRE}), which is the appropriate condition to consider and is satisfied by all ergodic full groups.

\begin{thmi}
	[see Thm.\ \ref{Thm: orbit full groups are complete invariants of OE}]
	Let $G$ and $H$ be two locally compact Polish groups acting in a measure-preserving ergodic manner on a standard $\sigma$-finite space $(X,\lambda)$. Then any group isomorphism between the orbit full groups $[\mathcal{R}_G]$ and $[\mathcal{R}_H]$ is the conjugation by a bijection $S : X \to X$ which preserves the measure $\lambda$ up to multiplication by an element of $\R_+^\ast$.
\end{thmi}

Apart from the uniform topology, one can also endow $\Aut(X,\lambda)$ with the weak topology (see \Cref{defi: weak topology}), which makes it a Polish group. In the pmp setup, Carderi and Le Maître proved that the topology of orbital convergence in measure brings the weak and the uniform topologies under the same umbrella. It is moreover the only Polish group topology that an ergodic full group can carry (\cite[Thm.~4.7]{CarderiLM2016}). We establish the following, extending this result to our setup.

\begin{thmi}
	[see Thm.\ \ref{UniquePoltopoonfg} and Thm.\ \ref{Thm: any Polish topology is coarser than the uniform}]
	Let $\mathbb{G}$ be an ergodic full group on a standard $\sigma$-finite space. Then $\mathbb{G}$ carries at most one Polish group topology, and such a topology is necessarily coarser than the uniform topology, but refines the weak topology, seen as topologies inherited from those on $\Aut(X,\lambda)$.
\end{thmi}

Proving that our Polish topology is coarser than the uniform one uses automatic continuity, which has already been well-studied. For instance Sabok has developped a framework giving a proof that $\Aut(X,\mu)$ has automatic continuity in \cite{Sabok2019}, generalized to infinite measure by Le Maître in \cite{LM2022}. An indispensable tool for establishing automatic continuity is the \textit{Steinhaus Property}, introduced in \cite{RosendalSolecki2007}, and used by Kittrell and Tsankov in \cite{KittrellTsankov2010} to show that ergodic full groups of countable equivalence relations have automatic continuity.  We expand on these techniques, and following an unpublished result from Fremlin, we show that any infinite measure-preserving ergodic full group is Steinhaus when endowed with the uniform topology.

Our result relies on a new factorisation of any infinite-measure preserving bijection (see \Cref{factoriseT}), allowing us to decrease the  Steinhaus exponent obtained in Fremlin's proof, from $228$ to $114$. 

Next we focus on a group that has no probability measure-preserving analogue: the group $\Autf(X,\lambda)$ of bijections with supports of finite measure. In \cite{LM2022}, Le Maître proved that this group does not admit any Polish group topology, a result in the spirit of the work of Rosendal (see \cite{Rosendal2005}). In the infinite measure-preserving context, we already know that the full group of a countable equivalence relation is a Polish group when endowed with the uniform topology, which has to coincide with the topology of orbital convergence in measure by uniqueness. In that case, the finitely supported bijections of the full group form a group that can be endowed with a Polish group topology, contrasting with the case of $\Autf(X,\lambda)$. We show that this is an equivalence.

\begin{thmi}
	[see Thm.\ \ref{Thm: characterisation of fg of countable eqrel}]
	Let $\mathbb{G}$ be an ergodic full group on a standard $\sigma$-finite space. The group $\mathbb{G}_f = \mathbb{G} \cap \Autf(X,\lambda)$ can be endowed with a Polish group topology if and only if $\mathbb{G}$ is the full group of a countable equivalence relation.
\end{thmi}

The last part of this work is dedicated to the topological and algebraic properties of ergodic full groups, seen as Polish groups. Contractibility, which has been shown
\begin{itemize}
	\setlength\itemsep{0em}
	\item for $\Aut(X,\mu)$ by Keane in \cite{Keane1970},
	\item for probability measure-preserving orbit full groups by Carderi and Le Maître in \cite{CarderiLM2016},
	\item for $\Aut(X,\lambda)$ and infinite measure-preserving full groups of countable equivalence relations (in fact even for some type $\mathrm{III}$ actions) by Danilenko in \cite{Danilenko1995},
\end{itemize}
is established for the class of infinite measure-preserving ergodic orbit full groups in \Cref{prop: ergodic orbit fg is contractible}. The strategy used in \cite{Keane1970} and then in \cite{CarderiLM2016} cannot be directly used, as induced bijections are only defined for conservative bijections. We follow Danilenko's proof from \cite[Thm.~2.2]{Danilenko1995}, who uses the contractibility of the space of partitions of $X$ into subsets of measure one, to circumvent this issue. 

The question of the existence of normal subgroups is also adressed. Eigen proved in \cite{Eigen1981} that full groups generated by a single ergodic bijection admit a unique non-trivial normal subgroup. We get the following general statement, which in particular uses the decomposition of any bijection into three involutions, a result of Ryzhikov (see \Cref{3invo}).

\begin{thmi}
	[see Thm.\ \ref{Thm: normal subgroup of an ergodic fg}]
	Let $\mathbb{G}$ be an ergodic full group on a standard $\sigma$-finite space. The group $\mathbb{G}_f$ is simple, and is the only non-trivial normal subgroup of $\mathbb{G}$.
\end{thmi}

Finally, Polishness raises the question of generic properties. Krengel and Sachedva have proved already that conservative elements and ergodic elements are generic in $\Aut(X,\lambda)$, when endowed with the weak topology (see \cite{Krengel1967} and \cite{Sachdeva1971}). Carderi and Le Maître, on the other hand proved that aperiodic elements are generic in pmp orbit full groups, when the acting group is uncountable. In particular their proof uses a result from \cite{Tornquist2006}, which we also generalize to infinite measures. Some results of Choksi and Kakutani from \cite{ChoksiKakutani1979} about conjugates of ergodic bijections are also used to prove the following.

\begin{thmi}
	[see Thm.\ \ref{thm: category-density theorem}]
	Let $\mathbb{G} = [\mathcal{R}_G]$ be an ergodic orbit full group on a standard $\sigma$-finite space associated with a Borel action of a Polish group $(G,\tau_G)$. If the topology of orbital convergence in measure is Polish on $\mathbb{G}$, then the following are equivalent:
\begin{enumerate}[(1)]
  \setlength\itemsep{0em}
  \item the aperiodic elements of $\mathbb{G}$ form a dense subset of $\mathbb{G}$;
  \item the ergodic elements of $\mathbb{G}$ form a dense subset of $\mathbb{G}$;
  \item the $\mathbb{G}_f$-conjugacy class of any aperiodic element of $\mathbb{G}$ is dense in $\mathbb{G}$;
  \item for any essentially free measure-preserving action of a countable group $\Gamma \curvearrowright (X,\lambda)$, there is a dense $G_\delta$ of elements of $\mathbb{G}$ inducing a free action of $\Gamma \ast \Z$.
\end{enumerate}
These equivalent conditions moreover imply the following one:
\begin{enumerate}[(5)]
  \setlength\itemsep{0em}
  \item[(5)] for any $\tau_G$-neighbourhood $V$ of $e_G$, the set $\bigcup_{g \in V} \left\{ x \in X \mid g \cdot x \neq x \right\}$ is conull.
\end{enumerate}
Finally, \textit{(5)} implies the above conditions if the space $X$ is endowed with a compatible Polish topology, with regards to which the measure is locally finite and the $G$-action is continuous.
\end{thmi}

In particular, these equivalent conditions allow us to give another characterization of ergodic full groups coming from a countable equivalence relation (see \Cref{cor: dichotomy dense conjugacy class}). 

Given the hypotheses of \Cref{Thm: orbit full groups are Polish groups} and \Cref{thm: category-density theorem}, their reliance on \cite[Thm.~4.4]{HoareauLM2024}, as well as \cite[Prop.~5.10]{HoareauLM2024} and \Cref{ex: action of Sinf}, the following question is very natural.

\begin{question*}
Let $G$ be a Polish group acting on a standard $\sigma$-finite space $(X,\lambda)$ in a measure-preserving manner, and consider the associated orbit full group $[\mathcal{R}_G]$, endowed with its topology $\tau_m^o$ of orbital convergence in measure. Are the following equivalent?
	\begin{itemize}
	\item The group $([\mathcal{R}_G],\tau_m^o)$ is a Polish group.
	\item There exists a continuous Polish model $G \curvearrowright (Y,\eta)$ for the $G$-action on $(X,\lambda)$: \textit{i.e.} an continuous measure-preserving $G$-action on a Polish space $Y$ endowed with a locally finite measure $\eta$, such that the actions are spatially isomorphic.
	\end{itemize}
\end{question*}

\begin{ack}
	I deeply thank my advisor François Le Maître for many invaluable advices and for his support, which led to the writing of this paper.
    I would also like to thank Corentin Correia and Matthieu Joseph for their many comments and for enlightening discussions. I also thank Arthur Troupel for his helpful comments when proofreading earlier drafts of this introduction. I finally thank David Fremlin for allowing me to use his proof of \Cref{494Yi}.
\end{ack}

\section{Prerequisites}{\label{sec: preliminaries}}

\subsection{About Polish and standard Borel spaces}

A \textbf{Polish space} is a separable and completely metrizable space, on which we can always find a bounded compatible metric, and a \textbf{Polish group} is a topological group whose topology is Polish. We recall also that a $G_\delta$ set is a countable intersection of open sets.

\begin{prop}
	[\textup{\cite[Thm.~3.11]{Kechris1995}}]
	{\label{prop: G delta Polish}}
	Let $(X,\tau)$ be a Polish space. Then $Z \subseteq X$ is Polish for the induced topology if and only if $Z$ is $G_{\delta}$ in $X$. 
\end{prop}

\begin{prop}
	[\textup{\cite[Prop.~1.2.3]{BeckerKechris1996}}]
	{\label{prop: quotient Polish group}}
	Let $G$ be a Polish group, and let $H \vartriangleleft G$ be a closed normal subgroup. Then $G / H$ is a Polish group for the quotient topology.
\end{prop}

Polish spaces are the topological spaces that give rise to a widely used measurable structure via their Borel sets. 
A \textbf{standard Borel space} $X$ is an uncountable measurable space whose Borel $\sigma$-algebra $\mathcal{B}(X)$ comes from a Polish topology.
Let us now endow $(X,\mathcal{B}(X))$ with an atomless measure, there are two different main cases:
	\begin{enumerate}\setlength\itemsep{0em}
	\item If $\mu$ is a probability measure defined on $\mathcal{B}(X)$, then $(X,\mathcal{B}(X),\mu)$ is a \textbf{standard probabilty space} (up to a renormalization any finite measure falls in this category).
	\item If $\lambda$ is an infinite $\sigma$-finite measure defined on $\mathcal{B}(X)$, then $(X,\mathcal{B}(X),\lambda)$ is a \textbf{standard $\sigma$-finite space}.
	\end{enumerate}
In the rest of this paper, the $\sigma$-algebra $\mathcal{B}(X)$ will usually be omitted in the notations, thus we will denote by $(X,\lambda)$ a standard $\sigma$-finite space. We recall the following classical result.

\begin{thm}
	[\textup{Lusin-Suslin, see \textit{e.g.}~\cite[Thm.~15.2]{Kechris1995}}]
	{\label{Thm: Lusin Suslin}}
	Let $X$ and $Y$ be two standard Borel spaces, and let $f: X \rightarrow Y$ be an injective Borel map. Then for every Borel subset $A$ of $X$, $f(A)$ is Borel.	
\end{thm}
	
	All standard Borel spaces are isomorphic (see \cite[Thm.~15.6]{Kechris1995}). Moreover, all standard probability spaces are isomorphic (see \cite[Thm.~17.41]{Kechris1995}). Furthermore, since every atomless sigma finite space can be written as a disjoint union of measure one subsets, this implies that all $\sigma$-finite spaces are isomorphic, justifying the terminologies.

We say that a function $f : X \to Y$ between two topological spaces is \textbf{Baire-measurable} if the preimage $f\inv(A)$ of any open subset $A \subseteq Y$ has the \textbf{Baire property}, that is to say that there exists an open subset $U \subseteq X$ such that $f\inv(A) \Delta U$ is meager. By \cite[Thm.~21.6]{Kechris1995}, all $\sigma(\Sigma_1^1)$-measurable functions between Polish spaces are Baire-measurable. We refer to \cite[I.8]{Kechris1995} for more on the Baire property.

\begin{prop}
	[\textup{\cite[Thm.~1.2.6]{BeckerKechris1996}}]
	{\label{prop: analytic morphism}}
	Let $G$ and $H$ be two Polish groups. Then any Baire-measurable homomorphism $\varphi: G \rightarrow H$ is continuous.
\end{prop}

\subsection{Infinite measures and measure algebras}

We give the definition of measure algebras in the context of $\sigma$-finite spaces, and recall a few essential facts.

\begin{defi}
	We call the \textbf{measure algebra} of a standard $\sigma$-finite space $(X,\lambda)$ and denote by $\MAlg(X,\lambda)$ the space of Borel subsets of $X$, where two such subsets are identified if the measure of their symmetric difference is equal to zero. We will however mostly be interested in $\MAlgf(X,\lambda)$, the subspace of $\MAlg(X,\lambda)$ comprised of the elements of finite measure, which we will call the \textbf{finite measure algebra} of $(X,\lambda)$. It is equipped with the metric $d_{X,\lambda}$ defined by $d_{X,\lambda}(A,B) \coloneqq \lambda(A \Delta B)$.
\end{defi}

We have the following well-known result (see \textit{e.g.} \cite[Lem.~2.1]{LM2022}).

\begin{prop}
	{\label{prop: MAlgf is Polish}}
	Let $(X,\lambda)$ be a standard $\sigma$-finite space. Then $(\MAlgf(X,\lambda),d_{X,\lambda})$ is a separable and complete metric space.
\end{prop}

The existence of a \textit{maximal element} in a measure algebra is very useful. Some authors prove the following via a \textit{measurable Zorn's lemma}, but we can also adapt another proof of Le Maître in the finite measure setting (\cite[Appendix~A]{LMthesis}), to our context. This proof works for an abstract measure algebra, with no underlying measured space (see \textit{e.g.} \cite[Ch.~32]{FremlinVol3}).

\begin{prop}
	{\label{prop: maximal element in MAlgf}}
	Consider $(X,\lambda)$ a standard $\sigma$-finite space, and $\MAlgf(X,\lambda)$ the associated finite measure algebra. Let $\mathcal{F} = \left\{ F_i \mid i \in I \right\}$ be a family of elements of $\MAlgf(X,\lambda)$.
	\begin{enumerate}
	\item[(1)] If $\mathcal{F}$ is upward directed and satisfies $M = \sup_{A \in \mathcal{F}} \lambda(A) < \infty$ for any $A$ in $\mathcal{F}$, then $\mathcal{F}$ admits a supremum, which has measure $M$. Moreover, if $\mathcal{F}$ is stable by countable union, then $\sup \mathcal{F}$ is a maximal element. Furthermore, the supremum (resp. maximal element) of the family is obtained as the limit of an increasing sequence of elements of the family.
	\item[(2)] If $\mathcal{F}$ satisfies $\lambda( \cup_{i \in J} F_i) < R < \infty$ for any finite subset $J \subseteq I$, then $\mathcal{F}$ admits a supremum. Furthermore, the supremum of the family is obtained as the limit of an increasing sequence of finite reunions of elements of the family.
	\end{enumerate}
\end{prop}

\begin{proof}
	\textit{(1)}. Let us start by assuming that $\mathcal{F}$ is an upward directed family of elements of $\MAlgf(X,\lambda)$, in other words for any $A, B$ in $\mathcal{F}$ there exists $C$ in $\mathcal{F}$ such that $A \subseteq C$ and $B \subseteq C$.
	
	Fix a sequence $(A_n)$ of elements of $\mathcal{F}$ with $\lambda(A_n)$ converging to $M$. As $\mathcal{F}$ is upward directed, by induction we can find $B_n \in \mathcal{F}$ such that $A_n \subseteq B_n$ and $B_{n-1} \subseteq B_n$, which gives us an increasing sequence in $\mathcal{F}$ satisfying $\lambda(B_n) \to M$. We prove that it is a Cauchy sequence.
	Let $\varepsilon > 0$ and let $N \in \N$ be such that for any $n \geqslant N$ we have $\lambda(B_n) > M - \varepsilon$. Then for any $n > m \geqslant N$, $d_{X,\lambda}(B_n,B_m) = \lambda(B_n \Delta B_m) = \lambda(B_n \setminus B_m) = \lambda(B_n) - \lambda(B_m) < \varepsilon$. Therefore $(B_n)$ is a Cauchy sequence in $(\MAlgf(X,\lambda),d_{X,\lambda})$ which is complete, and its limit is $\sup \mathcal{F}$. In particular $\lambda(\sup \mathcal{F}) = \sup_{A \in \mathcal{F}} \lambda(A)$.
	
	We now assume that $\mathcal{F}$ is stable by countable union (do note that this implies being upward directed). By the previous argument, $\sup \mathcal{F}$ is defined as the limit of an increasing sequence $(B_n)$ of elements of $\mathcal{F}$. This means that we have $\sup \mathcal{F} = \lim B_n = \bigcup_{n \in \N} B_n$, which is in $\mathcal{F}$ by hypothesis.
	
	\textit{(2)} We finally assume that $\mathcal{F}$ is any family of elements of $\MAlgf(X,\lambda)$ satisfying the measure condition for finite unions. Denote by $\mathcal{G}$ the family of finite reunions of elements of $\mathcal{F}$. It is upward directed, and therefore by the previous argument it admits a supremum $\sup \mathcal{G}$, which is also the supremum of $\mathcal{F}$.
\end{proof}

We have the following two classical lemmas about the behaviour of a probability measure in the class of a $\sigma$-finite measure (recall that such a measure always exist).

\begin{lem}
	[\textup{\cite[Lem.~4.2.1]{Cohn2013}}]
	{\label{lem: probability measure equivalent to infinite measure}}
	Let $X$ be a standard Borel space endowed with its natural $\sigma$-algebra. Let $\lambda$ and $\mu$ be two Borel measures defined on $X$, with $\mu$ finite. Then $\mu$ is absolutely continuous with regards to $\lambda$ if and only if for any $\varepsilon > 0$ there exists $\delta >0$ such that any Borel subset $A$ of $X$ satisfying $\lambda(A) < \delta$ also satisfies $\mu(A) < \varepsilon$.
\end{lem} 

	\begin{lem}
	{\label{lem: convergence in proba equivalent to sigma finite on finite measure}}
Fix $\mu$ a probability measure in $[\lambda]$. Let also $(A_n)$ be a sequence of Borel subsets of $X$. We have the following equivalence.
\[
(\mu(A_n) \rightarrow 0) \Longleftrightarrow (\forall \, C \mbox{ Borel subset of X of finite $\lambda$-measure, } \lambda(A_n \cap C) \rightarrow 0) .
\]
	\end{lem}
	
	\begin{proof}
The converse implication is immediate from \Cref{lem: probability measure equivalent to infinite measure}.
For the direct implication, consider the Radon-Nikodym derivative $f$ such that $d\lambda = f d\mu$. For any Borel subset $C$ of $X$ of finite measure, we have
\[
\lambda(A_n \cap C) = \int_{C} \mathds{1}_{A_n}  f(x)d \mu(x)  \longrightarrow 0 
\]
by the Lebesgue dominated convergence theorem.
	\end{proof}

The interplay between topology and measure is particularly relevant in the case of infinite measures. We recall the definition of locally finite measures.

\begin{defi}
    Let $\lambda$ be a Borel measure on a Polish space $X$. We say that $\lambda$ is \textbf{locally finite} (with regards to the topology of $X$) if every $x\in X$ admits an open neighborhood $U$ such that $\lambda(U)<+\infty$.
\end{defi}

\subsection{Support and separators}

This section holds in the broad context of Borel bijections of a standard Borel space $X$, with no measure involved.
The \textbf{support} of a Borel bijection $T : X \to X$ is defined as follows:
\[
\supp T \coloneqq \left\{ x \in X \mid T(x) \neq x \right\}.
\]
Notice that for two Borel bijections $T,S$ of $X$, we have $S(\supp T) = \supp(STS\inv )$.

A separator gives a very useful decomposition of the support, which is more useful than just a non-null Borel subset disjoint from its image by $T$. 

\begin{defi}
	Let $T$ be a Borel bijection of a standard Borel space $X$. A Borel subset $A$ of $\supp T$, such that $\supp T = A \sqcup \left(T(A) \cup T\inv (A)\right)$ is called a \textbf{separator} for $T$.
\end{defi}

The following well-known lemma is crucial, as it implies that separators always exist for Borel bijections. This proof uses a result from \cite{EisenmannGlasner2016} and is given for convenience. It is of note that we can also prove the following for measure-preserving bijections via \Cref{prop: maximal element in MAlgf}.

\begin{lem}
	{\label{lem: better separator}}
	Let $X$ be a standard Borel space, and $T : X \to X$ be a Borel bijection. Let also $C$ be any Borel subset of $X$. Then there exists a Borel subset $A \subseteq C$ such that $C \cap \supp T = A \sqcup (T(A) \cup T\inv(A))$. In particular for $C = X$, this implies that $A$ is a separator for $T$.
\end{lem}

\begin{proof}	 
	Using Lemma 5.1 from \cite{EisenmannGlasner2016}, we write $\supp T = \bigsqcup_{n \in \N} B_n$, where $(B_n)$ is a partition of $\supp T$ such that $T(B_n)$ is disjoint from $B_n$ for any $n \in \N$. We then have $C\cap \supp T = \bigsqcup_{n \in \N} C \cap B_n$, and for all $n \in \mathbb{N}$ we have that $T(C \cap B_n)$ is disjoint from $C \cap B_n$.

We now inductively define a sequence $(A_n)$ of Borel subsets of $C \cap \supp T$ as follows. Fix $A_0 = C \cap B_0$, and let $A_{i+1}$ be defined by
\[
A_{i+1} = (C \cap B_{i+1}) \setminus \left(   \left(    \bigsqcup_{j \leqslant i} T(A_j)  \right) \cup  \left(  \bigsqcup_{j \leqslant i} T\inv (A_j) \right)  \right).
\]
We conclude the proof by noticing that $A \coloneqq \bigsqcup_n A_n$ is suitable. Indeed, for any $i$ we have $A_{i+1} \cap \left( \bigsqcup_{j \leqslant i} T(A_j) \right) = \emptyset$ and $A_{i+1} \cap \left( \bigsqcup_{j \leqslant i} T\inv (A_j) \right) = \emptyset$ so $A$ is disjoint from both $T(A)$ and $T\inv (A)$. We have $C \cap \supp T = A \cup T(A) \cup T\inv (A)$, as almost any $x \in (C \cap \supp T) \setminus A$ is either in $T(A)$ or in $T\inv (A)$.
\end{proof}

\section{Topologies on the group of infinite measure-preserving bijections}

\subsection{The group \texorpdfstring{$\Aut(X,\lambda)$}{Aut(X,lambda)} and the uniform topology}

Let $(X,\lambda)$ be a standard $\sigma$-finite space.
The group $\Aut(X,\lambda)$ plays a very important part in our work. It is defined as the group of measure-preserving bijections of $(X,\lambda)$, \textit{i.e.} bijections $T : X \to X$ satisfying $\lambda(A) = \lambda( T\inv(A)) = \lambda(T(A))$ for any Borel subset $A \subseteq X$, where two such bijections are identified when they coincide on a conull set.

The following group does not have any analogue in the probability measure-preserving context. We also introduce two useful notations.

	\begin{defi}
The group $\Autf(X,\lambda)$ is the normal subgroup of $\Aut(X,\lambda)$ consisting of all $T$ in $\Aut(X,\lambda)$ such that $\lambda(\supp T)$ is finite.
	\end{defi}

\begin{nota}
For a subgroup $\mathbb{G}$ of $\Aut(X,\lambda)$ and 
a Borel subset $A \subseteq X$
, we define
\[
\mathbb{G}_A \coloneqq \left\{  T \in  \mathbb{G} \mid \mbox{supp}\,T \subseteq A  \right\},
\]
as well as
\[
\mathbb{G}_f \coloneqq \left\{ T \in \mathbb{G} \mid \lambda(\supp T) < \infty \right\} = \mathbb{G} \cap \Autf(X,\lambda).
\]
\end{nota}

In the probability context, Halmos defined in \cite{Halmos1956} two metrics on the group of measure-preserving automorphisms and showed that they are comparable, and that they both induce $\tau_u$. The generalization of this topology to the group of non-singular bijections was introduced by \cite{IonescoTulcea1965}, and studied in the infinite measure context for instance in \cite{ChoksiKakutani1979}.

	\begin{defi}
	{\label{def: uniform topo}}	
We define on $\Aut(X,\lambda)$ the \textbf{uniform topology} $\tau_u$ as the group topology induced by the family of pseudometrics
\[
d_{u,C}: S,T \mapsto \lambda \left( \left\{  x \in C \mid S(x) \neq T(x)  \right\} \right) ,
\]
where $C$ ranges among Borel subsets of $X$ of finite measure.
	\end{defi}
	
	\begin{rem}
	{\label{rem: equiv proba measures}}
By \Cref{lem: convergence in proba equivalent to sigma finite on finite measure} this definition is equivalent to the more commonly found one, which is defined as follows. Let $\mu$ be a probability measure in $[\lambda]$ and define the metric $d_\mu$ on $\Aut(X,\lambda)$ by $d_\mu(S,T) \coloneqq \mu(\left\{ x \in X \mid S(x) \neq T(x)\right\})$. Then $\tau_u$ is the topology induced by $d_\mu$. The definition we chose is more suited to our needs.
	\end{rem}

\begin{rem}
{\label{prop: uniform topology}}
From the definition of $d_{u,C}$, if we set
\[
\mathcal{N}_{C,\varepsilon} = \left\{  T \in \Aut(X,\lambda) \mid  
\lambda(C \cap \supp T) \leqslant \varepsilon  \right\},
\]
then the family $\left\{ \mathcal{N}_{C,\varepsilon} \mid C \mbox{ Borel subset of $X$ of finite measure, } \varepsilon > 0 \right\}$ is a base of neighbourhoods of the identity in $(\Aut(X,\lambda),\tau_u)$. One can also consult \cite[Prop.~494Cb]{FremlinVol3}.
\end{rem}

\subsection{Weak topology and topology of convergence in measure}

We now give one of the classical definitions of the weak topology of $\Aut(X,\lambda)$.

\begin{defi}
	{\label{defi: weak topology}}
	The \textbf{weak topology} on $\Aut(X,\lambda)$ is denoted by $\tau_w$ and is defined in the following way: $T_n \rightarrow T$ if and only if for all $A \subseteq X$ Borel subset of finite measure, one has $\lambda(T_n(A) \Delta T(A)) = d_{X,\lambda}(T_n(A),T(A)) \rightarrow 0$.
\end{defi}

The following proposition states that it is a Polish group.

\begin{prop}
	[\textup{\cite[Prop.~2.2]{LM2022}}]
	{\label{prop: Aut is Polish}}
	The group $\Aut(X,\lambda)$ is equal to the group of isometries of $(\MAlgf(X,\lambda),d_{X,\lambda})$ which fix $\emptyset$. As such, it is a Polish group.	
\end{prop}

Equivalently, one can define $\tau_w$ in the following manner: $T_n \rightarrow T$ if and only if for all $A \subseteq X$ Borel subset of finite measure $\lambda(T_n(A) \setminus T(A)) \rightarrow 0$. Moreover, if $\mu$ is a probability measure in $[\lambda]$ and $T_n \to T$ for $\tau_w$, for any $A \subseteq X$ we have $\mu(T_n(A) \Delta T(A)) \to 0$. We actually have the following.

	\begin{lem}
	{\label{lem: continuous action of Aut lambda on Malg mu}}
	Let $(X,\lambda)$ be a standard $\sigma$-finite space, and $\mu$ be a probability measure in $[\lambda]$. Then for any Borel subset $A \subseteq X$ the application $T \in \Aut(X,\lambda) \mapsto \mu( T(A) \Delta A)$ is $\tau_w$-continuous.
	\end{lem}	
	
	\begin{proof}
	Let $(T_n)$ be a sequence of measure-preserving bijections $\tau_w$-converging to the identity in $\Aut(X,\lambda)$.
	By Lemma \ref{lem: convergence in proba equivalent to sigma finite on finite measure} it is enough to prove that for any Borel subset $C$ of finite $\lambda$-measure we have $\lambda(C \cap (T_n(A) \Delta A)) \to 0$.
	We have
	\begin{align*}
	\lambda(C \cap (T_n(A) \Delta A)) & = \lambda((C \cap T_n(A)) \Delta (C \cap A))\\
	& = d_{X,\lambda}(C \cap T_n(A) , C \cap A)\\
	& \leqslant d_{X,\lambda}(C \cap T_n(A) , T_n(C) \cap T_n(A)) +  d_{X,\lambda}(T_n(C) \cap T_n(A) , C \cap A)\\
	& \leqslant d_{X,\lambda}(C,T_n(C)) + d_{X,\lambda}(T_n(C \cap A) , C \cap A),
	\end{align*}
and both terms tend to $0$ by weak convergence of $(T_n)$ to $\id_X$, as $\lambda(C) < + \infty$.
	\end{proof}

The second topology we will be interested in is a natural topology that $\Aut(X,\lambda)$ inherits from the space of measurable functions. 

\begin{defi}
	Let $(Z,\tau_Z)$ be a Polish space. We denote by $\Lzero(X,\lambda,(Z,\tau_Z))$ the space of Lebesgue-measurable maps from $X$ to $Z$, identified up to measure $0$. This notation will often be shortened to $\Lzero(X,\lambda,Z)$ if there is no possible confusion on the topology of the range. Let $d$ be a compatible bounded metric on $(Z,\tau_Z)$. For any Borel subset $C \subseteq X$ of finite measure, we define $\tilde{d}_C$ by
\[
\tilde{d}_C(f,g) \coloneqq \int_C d(f(x),g(x)) d\lambda(x).
\]
The family of pseudometrics $\tilde{d}_C$ induces the \textbf{topology of convergence in measure} denoted by $\tau_m$, where $C$ ranges among the Borel subsets of $X$ of finite measure.
\end{defi}

	\begin{rem}
	{\label{rem: pseudometrics}} 
	We can also define $\tau_m$ as follows.
	Let $\mu$ be a finite measure in $[\lambda]$, let $d$ be a compatible bounded metric on $Z$, define a metric $\tilde{d}$ on $\Lzero(X,\lambda,(Z,\tau_Z))$ by
\begin{align*}
\tilde{d}(f,g) \coloneqq \int_X d(f(x),g(x)) \, d\mu(x),
\end{align*}
then $\tau_m$ is also induced by the metric $\tilde{d}$.
	\end{rem}

It is of note that $\tau_m$ does not depend on the choice of the measure, only on its equivalence class, and also does not depend on the choice of $d$, but only on the topology it induces.
This follows from the following statement.

	\begin{prop}
	[\textup{\cite[Prop.~6]{Moore1976}}]
	{\label{prop: cv in measure equiv}}
	Let $(f_n)$ be a sequence of elements of $\Lzero(X,\lambda,(Z,\tau_Z))$, and $f \in \Lzero(X,\lambda,(Z,\tau_Z))$, where $(Z,\tau_Z)$ is Polish, with $d$ a compatible bounded metric. Then the following are equivalent:
\begin{enumerate}[(a)]\setlength\itemsep{0em}
\item $f_n \longrightarrow_{\tau_m} f$,
\item \label{item: Moore (b)} for any probability measure $\mu$ in $[\lambda]$, for all $\varepsilon >0$, $\mu (\{  x \in X \mid d(f_n(x),f(x)) > \varepsilon  \}) \longrightarrow 0$,
\item every subsequence of $(f_n)$ admits a subsequence $(f_{n_k})$ converging $\lambda$-almost everywhere to $f$.
\end{enumerate}
	\end{prop}

Next is a useful generalization of the classical approximation by simple functions.

	\begin{lem}
	[\textup{\cite[Lem.~4]{KaichouhLM2015}}]
	{\label{prop: functions with finitely many values are dense}}
Let $Z$ be a Polish space, and $(X,\lambda)$ be a standard $\sigma$-finite space. The space of measurable functions from $X$ to $Z$ that take finitely many values  is dense in $\Lzero(X,\lambda,Z)$.
	\end{lem}

Let us now consider the right action of $\Aut(X,\lambda)$ on $\Lzero(X,\lambda,(Z,\tau_Z))$, defined as follows: if $T \in \Aut(X,\lambda)$ and $f \in \Lzero(X,\lambda,(Z,\tau_Z))$, then 
\[
(f \cdot  T)(x) \coloneqq f(T(x))
\]
We will often write $T(x) = Tx$ for convenience. Notice that when $Z = X$, we may see $\Aut(X,\lambda)$ as a subset of $\Lzero(X,\lambda,(X,\tau_X))$, where we identify an measure-preserving bijection $T$ with the corresponding function $f_T: x \mapsto Tx$. 

We then have the following proposition, which will be crucial for the proof of the main result of section \ref{section: Polish Orbit full groups}. The local finiteness hypothesis is what distinguishes this statement from the analogous result in the finite measure context (see \cite[Prop.~2.9]{CarderiLM2016}), and was proved in \cite{Papangelou1968}.

	\begin{prop}{\label{prop: topology on Aut seen as measureable functions}}
Let $(X,\lambda)$ be a standard $\sigma$-finite space and endow $\Aut(X,\lambda)$ with its weak topology $\tau_w$. Consider also $\tau_X$ a Polish topology on $X$, compatible with its Borel structure.
	\begin{enumerate}[(1)]\setlength\itemsep{0em}
	\item The inclusion of $(\Aut(X,\lambda),\tau_w) $ in $ (\Lzero(X,\lambda,(X,\tau_X)),\tau_m)$ is continuous.
	
	\item If $\lambda$ is locally finite on $\tau_X$, then the inclusion $(\Aut(X,\lambda),\tau_w) \hookrightarrow (\Lzero(X,\lambda,(X,\tau_X)),\tau_m)$ is a topological embedding.
	
	\item If $\lambda$ is locally finite on $\tau_X$, then $\Aut(X,\lambda)$ is $G_\delta$ in $(\Lzero(X,\lambda,(X,\tau_X)),\tau_m)$.
	\end{enumerate}
	\end{prop}

	\begin{proof}

Items \textit{(1)} and \textit{(2)} are proved in \cite[Thm.~1]{Papangelou1968}. Note that it can actually be shown that $(\Aut(X,\lambda),\tau_w)$ acts continuously on $(\Lzero(X,\lambda,(Z,\tau_Z)),\tau_m)$ for any Polish space $(Z,\tau_Z)$, which also yields \textit{(1)}, as $\Aut(X,\lambda)$ can then be seen as the orbit of $\id_X$ in $\Lzero(X,\lambda,(X,\tau_X))$. We do not provide details, as we will not need this stronger fact.

\noindent
\textit{(3).} By \Cref{prop: Aut is Polish} $(\Aut(X,\lambda),\tau_w)$ is a Polish group and by \textit{(2)} the weak topology is the induced topology, so by \Cref{prop: G delta Polish} $\Aut(X,\lambda)$ is $G_\delta$ in $(\Lzero(X,\lambda,(X,\tau_X)),\tau_m)$, which is itself Polish (see \cite[Sec.~2.4]{CarderiLM2016}).
\end{proof}

The hypothesis of local finiteness in Proposition \ref{prop: topology on Aut seen as measureable functions} is used to find a measure-controllable set containing the compact subset we approximate, and cannot be dispensed with. We give examples in \Cref{section: necessity of the local finiteness}.

We will finally need the following.

	\begin{prop}
	[\textup{\cite[Prop.~7]{Moore1976}, see also \cite[Prop.~19.6]{Kechris2010}}]
	{\label{prop: measurable functions to a Polish group is Polish}}
Let $G$ be a Polish group. Then $\Lzero(X,\lambda,G)$ is a Polish group for the topology of convergence in measure and the pointwise product.
	\end{prop}

\section{Full groups and exchanging involutions}{\label{sec: fg and involutions}}

Dye's definition of full groups from \cite{Dye1959} easily extends to our infinite measure-preserving setup as follows. In this whole section, $(X, \lambda)$ will be a standard $\sigma$-finite space as usual.

	\begin{defi}
Consider $(T_n)$ a sequence of elements of $\Aut(X,\lambda)$. An element $T$ in $\Aut(X,\lambda)$ is obtained by \textbf{cutting and pasting} $(T_n)$ if there exists a countable partition $(A_n)$ of $X$ such that for all $n$ in $\mathbb{N}$ we have
\[
T_{\restriction A_{n}} = T_{n \restriction A_{n}}.
\]
A subgroup $\mathbb{G}$ of $\Aut(X,\lambda)$ is a \textbf{full group} if it is stable under the operation of cutting and pasting any sequence of elements of $\mathbb{G}$.
	\end{defi}

	\begin{nota}
For any family $(T_i)_{i \in I}$ with $T_i \in \Aut(X,\lambda) $ for any $ i \in I$, we denote by $[(T_i)_{i \in I}]$ the full group generated by $(T_i)_{i \in I}$, that is to say the smallest full group that contains the family. In particular, any $S \in [(T_i)_{i \in I}]$ satisfies $\supp S \subseteq \bigcup_{i \in I} \supp T_i$, up to a null set.
	\end{nota}

	\begin{defi}
We say that a subgroup $\mathbb{G}$ of $\Aut(X,\lambda)$ is \textbf{ergodic} if for every $A \subseteq X$ such that $\lambda(T(A) \Delta A) = 0$ for every $T$ in $\mathbb{G}$, we have that $A$ is either null or conull.	
	\end{defi}	
	
One of the easy consequences of ergodicity is the following, which will be used in \Cref{sec: complete invariants}.

	\begin{rem}
	{\label{rem: invariant f is constant}}
Any real-valued function $f$ which satisfies $f = f \circ T\inv$ for any $T$ in an ergodic subgroup $\mathbb{G}$ of $\Aut(X,\lambda)$ is essentially constant. Indeed for every $q \in \Q$ consider the set $\left\{ x \in X \mid f(x) < q \right\}$. It is $\mathbb{G}$-invariant (under the action by precomposition by the inverse) so it is null or conull, thus $f$ is constant up to a null set (see \textit{e.g.} \cite[Prop.~1.0.9]{Aaronson1997}).
	\end{rem}
	
Ergodicity of a group $\mathbb{G}$ immediately yields ergodicity of any group containing $\mathbb{G}$, but it does not go down to subgroups in general. We do however have the following well known proposition (see \textit{e.g.}~\cite[Prop.~1.6.7]{Aaronson1997}), about the behaviour of ergodicity with regards to weak-density. The proof is immediate from \Cref{lem: continuous action of Aut lambda on Malg mu}.

	\begin{prop}
	{\label{prop:weaker invariance for ergodic subgroups}}
Let $\mathbb{G} \leqslant \Aut(X,\lambda)$ be an ergodic subgroup, and let $\Gamma$ be a $\tau_w$-dense subgroup of $\mathbb{G}$. Then $\Gamma$ is ergodic.
	\end{prop}

\begin{ex}
	{\label{ex: G_f ergodic iff G ergodic}}
	For any ergodic full group $\mathbb{G} \leqslant \Aut(X,\lambda)$, $\mathbb{G}_f$ is $\tau_w$-dense in $\mathbb{G}$ (and in fact dense for $\tau_u$, see \Cref{rem: Autf uniformly dense in Aut}), so $\mathbb{G}$ is ergodic if and only if $\mathbb{G}_f$ is ergodic.
\end{ex}
	
For the uniform topology on full groups, we have the following. Recall that for any probability measure $\mu \in [\lambda]$,  $d_\mu(S,T) = \mu(\left\{ x \in X \mid S(x) \neq T(x) \right\})$. The result is classical even in the broader setting of non-singular full groups, and attributed to Hamachi and Osikawa \cite[Lem.~6]{HamachiOsikawa1981}.

	\begin{prop}
	{\label{closedforunif}}
Any full group $\mathbb{G} \leqslant \Aut(X,\lambda)$ is complete when equipped with the metric $\widetilde{d}: S,T \mapsto d_\mu(S,T) + d_\mu(S\inv ,T\inv )$, for any probability measure $\mu$ in $[\lambda]$. In particular, $\mathbb{G}$ is closed in $(\Aut(X,\lambda),\tau_u)$.
	\end{prop}

We now prove that there exists a \textit{partial measure-preserving isomorphism} which sends any fixed Borel subset of positive measure to another fixed Borel subset of the same measure. The probability version of this result has seen many applications, see for instance \cite{CarderiLM2016}, \cite{Fathi1978} or \cite{Kechris2010}. We first give the definitions of partial isomorphisms and pseudo-full groups, in order to give a very general version of the main proposition of this section.

	\begin{defi}
Let $\mathbb{G}$ be a subgroup of $\Aut(X,\lambda)$.
\begin{enumerate}[(1)]
\item Let $A$ and $B$ be two Borel subsets of $X$. We say that $\phi:A \rightarrow B$ is a \textbf{partial isomorphism} of $\mathbb{G}$ if there exists a partition $(A_n)_{n \in \mathbb{N}}$ of A, a partition $(B_n)_{n \in \mathbb{N}}$ of $B$, and a sequence $(T_n)_{n \in \mathbb{N}}$ of elements of $\mathbb{G}$ such that $T_n(A_n) = B_n$ and $T{_n}_{\restriction A_n} = \phi_{\restriction A_n}$, for every $n$ in $\mathbb{N}$. The domain of $\phi$ is $\mathrm{dom}(\phi) = A$ and its range is $\mathrm{rng}(\phi)= B$, up to null sets. In particular $\lambda(A) = \lambda(B)$.
\item The set of all partial isomorphisms of $\mathbb{G}$ is called the \textbf{pseudo-full group} of $\mathbb{G}$, and is denoted by $[[\mathbb{G}]]$.
\item The \textbf{uniform topology} defined on $\mathbb{G}$ generalizes to $[[\mathbb{G}]]$: it is the topology generated by the following metric:
\begin{align*}
d(\phi_1,\phi_2)  \coloneqq & \mu(\left\{  x \in \mathrm{dom}(\phi_1) \cap \mathrm{dom}(\phi_2) \mid \phi_1(x) \neq \phi_2(x)  \right\}) \\
% + & \lambda(\left\{  x \in \mathrm{rng}(\phi_1) \cap \mathrm{rng}(\phi_2) \mid \phi_1(x) \neq \phi_2(x)  \right\}) \\
 + & \mu(\mathrm{dom}(\phi_1) \Delta \mathrm{dom}(\phi_2)),
% + & \lambda(\mathrm{rng}(\phi_1) \Delta \mathrm{rng}(\phi_2)),
\end{align*}
where $\mu \in [\lambda]$ is a probability measure. We refer to \cite{Danilenko1995} for more on this topology.
\end{enumerate}
	\end{defi}

	\begin{rem}
	{\label{rem:elements of pseudo full groups are more than just restrictions}}
It is tempting to see the partial isomorphisms in $[[\mathbb{G}]]$ as restrictions of elements of $\Aut(X,\lambda)$ (and it is indeed the case in the probability context, see \cite[Cor.~1.15]{LMthesis}), but it is not the case here. For example, consider $X = [0, + \infty[$ with the Lebesgue measure and $T$ defined by $T : x \mapsto x+1$, which can easily be realised as a partial isomorphism. It has $X$ as its domain, but its range is not conull. 
	\end{rem}

The following is a crucial tool for full groups, and is well-known, but we have found no complete reference for the $\sigma$-finite version of the statement. As such, we provide a complete and detailed proof. Do note that there are no conditions on the measure of the complements of $A$ and $B$ in $X$, as highlighted by \Cref{rem:elements of pseudo full groups are more than just restrictions}.

	\begin{prop}{\label{prop:pseudo full group exchange subsets}}
Let $\mathbb{G}$ be an ergodic subgroup of $\Aut(X,\lambda)$. There exists a map 
\[
\begin{array}{rcccl}
\Phi & : &  \left\{ (A,B) \in (\MAlg(X,\lambda))^2 \mid \lambda(A) = \lambda(B) \right\}  & \longrightarrow & [[\mathbb{G}]] \\
& & (A,B) &  \longmapsto & \phi_{A,B}
\end{array}
\]
satisfying the following conditions:
\begin{itemize}
\item[•] For any $\phi_{A,B}$, we have $\mathrm{dom}(\phi_{A,B})= A $ and $\mathrm{rng}(\phi_{A,B})= B$.
\item[•] The restriction of $\Phi$ to $\left\{ (A,B) \in (\MAlgf(X,\lambda))^2 \mid \lambda(A) = \lambda(B)  \right\}$ is continuous for the uniform topology.
\end{itemize} 
	\end{prop}

	\begin{proof}
The proof is roughly the same as \cite[Lem.~2.4]{Danilenko1995}, which is stated in the case of a probability measure and a countable acting group, but it easily adapts to our case by considering a countable dense subgroup of $\mathbb{G}$.

Since $(\mathbb{G},\tau_w)$ is a subgroup of the separable and metrizable (thanks to Proposition \ref{prop: Aut is Polish}) group $(\Aut(X,\lambda),\tau_w)$, it is itself separable. As such, we can consider $(T_n)$ a countable dense subset of $\mathbb{G}$ and $\Gamma = (\gamma_n)$ the countable subgroup generated by $(T_n)$. 

Let $A$ and $B$ first be two Borel subsets of $X$, such that $0 < \lambda(A) = \lambda(B) < + \infty$. We recursively define a countable family of pairwise disjoint subsets $A_n$ of $A$ as follows:
\begin{align*}
\begin{cases}
\displaystyle{\widetilde{A}_0 = (\gamma_{0}\inv B) \cap A }\\
\displaystyle{\widetilde{A}_{n+1} = \left(\gamma_{n+1}\inv \left( B \setminus \bigsqcup_{m \leqslant n} \gamma_m \widetilde{A}_m \right)\right) \bigcap \left( A \setminus \bigsqcup_{m \leqslant n}  \widetilde{A}_m \right). }
\end{cases}
\end{align*}
The set $\widetilde{A}_n$ represents the elements of $A$ sent by $\gamma_n$ to $B$, after removing the elements previously sent.
Now set $\widetilde{A} = \bigsqcup_{n \in \N} \widetilde{A}_n$ and let $\phi: \widetilde{A} \rightarrow \bigsqcup_{n \in \N} \gamma_n \widetilde{A}_n$ be the Borel application that sends $x \in \widetilde{A}_n$ to $\gamma_n (x) \in \gamma_n \widetilde{A}_n$. 
By definition, $\phi$ is a partial isomorphism between $\widetilde{A}$ and $\bigsqcup_{n \in \N} \gamma_n \widetilde{A}_n$. In particular, $\lambda(\mathrm{dom}(\phi)) = \lambda(\mathrm{rng}(\phi))$.

Let us now suppose that either $\lambda(A \setminus \mathrm{dom}(\phi)) > 0$ or $\lambda(B \setminus \mathrm{rng}(\phi)) > 0$. As $A$ and $B$ have finite measure, we have $\lambda(A \setminus \mathrm{dom}(\phi)) = \lambda(B \setminus \mathrm{rng}(\phi)) > 0$.
Define ${\overline{B}= \bigcup_{n \in \mathbb{N}}\gamma_n (A \setminus \mathrm{dom}(\phi))}$, and notice that $\displaystyle{\overline{B}}$ is non null and invariant under the action of $\Gamma$. Ergodicity and the fact that $\Gamma$ is dense in $\mathbb{G}$ ensure that $\displaystyle{\overline{B}}$ is conull, by Proposition \ref{prop:weaker invariance for ergodic subgroups}. This coupled with the fact that $\lambda(B \setminus \mathrm{rng}(\phi)) > 0$ implies that there exists an integer $n$ such that 
\[
\lambda \left(   (B \setminus \mathrm{rng}(\phi)) \, \bigcap \,  \gamma_n (A \setminus \mathrm{dom}(\phi))    \right) > 0.
\]
We define $n_0$ as the smallest such integer. 
As $\gamma_{n_0}$ is measure-preserving, we then have
\[
\lambda \left(   \gamma_{n_0}\inv  (B \setminus \mathrm{rng}(\phi)) \, \bigcap \,  (A \setminus \mathrm{dom}(\phi))    \right) > 0.
\]
Notice now that $(B \setminus \mathrm{rng}(\phi)) \subseteq  ( B \setminus \bigsqcup_{m < {n_0}} \gamma_m \widetilde{A}_m )  $ and $(A \setminus \mathrm{dom}(\phi)) \subseteq  ( A \setminus \bigsqcup_{m < {n_0}}  \widetilde{A}_m )  $, which means that $  \left(   \gamma_{n_0}\inv  (B \setminus \mathrm{rng}(\phi)) \, \bigcap \,  (A \setminus \mathrm{dom}(\phi))    \right)  $ is contained in $\widetilde{A}_{n_0}$ by construction, and thus it is contained in $\mathrm{dom} (\phi)$. This is the contradiction we sought, as this set has positive measure and is contained both in $\mathrm{dom} (\phi)$ and in $A \setminus \mathrm{dom} (\phi)$.\\

Now if $\lambda(A) = \lambda(B) = + \infty$, observe that $(A,\lambda_{\restriction A})$ and $(B,\lambda_{\restriction B})$ are both standard $\sigma$-finite spaces. It is then possible to write $A = \bigsqcup A'_i$ and $B = \bigsqcup B'_i$, with $\lambda(A'_i) = \lambda(B'_i) < + \infty$ for every $i$ in $\mathbb{N}$. The previous argument gives us a sequence $(\phi_i)$ of partial isomorphisms with domains $(A'_i)$ and ranges $(B'_i)$. The partial isomorphism defined on $A$ by $\phi_{\restriction A'_i} = \phi_i$ is in $[[\mathbb{G}]]$ by construction, and is suitable so we take it to be $\phi_{A,B}$.\\

For the last part of the statement, we notice that by construction, if $(A_n)$ and $(B_n)$ are as in the statement, then
	\[	
	\lambda\left( \left\{ x \in (A_n \cap A) \cup (B_n \cap B) \mid \phi_{A_n,B_n}(x) \neq \phi_{A,B}(x) \right\} \right) \longrightarrow 0.
	\]
which concludes the proof.
	\end{proof}

The following two corollaries are very useful and used extensively throughout the rest of this work. Recall that for any Borel subset $C \subseteq X$, $\mathbb{G}_C = \left\{ T \in \mathbb{G} \mid \supp T \subseteq C \right\}$. The proofs are straightforward if we just cut and paste the partial isomorphisms given by \Cref{prop:pseudo full group exchange subsets}.

	\begin{cor}
	{\label{cor:ergodic full group moves elements of same (co)-measure around}}
Let $\mathbb{G}$ be an ergodic full group and let $C \subseteq X$ be any Borel subset. Let $A$ and $B$ be two Borel subsets of $C$ such that $\lambda(A) = \lambda(B)$ and $\lambda(C \setminus A) = \lambda(C \setminus B)$. Then there exists an element $T$ of $\mathbb{G}_C$ such that $T(A) = B$ and $T(C\setminus A) = C \setminus B$.
	\end{cor}

\begin{cor}
	{\label{cor:exchanging involutions in ergodic full group}}
	Let $\mathbb{G}$ be an ergodic full group and let $A$ and $B$ be two Borel subsets of $X$ such that $\lambda(A \setminus B) = \lambda(B \setminus A)$. Then there exists an involution $U$ in $\mathbb{G}_{A\Delta B} $ such that $U(A) = B$. Moreover, $(A,B) \in (\MAlgf(X,\lambda))^2 \mapsto U$ is $\tau_u$-continuous.
\end{cor}

\begin{rem}
	There is a clear distinction here between the probability measure-preserving case and our context. Indeed, if $(X,\mu)$ is a standard probability space and $\mu(A) = \mu(B)$, then we automatically have $\mu(A \setminus B) = \mu(B \setminus A)$. Thus \Cref{cor:exchanging involutions in ergodic full group} can be seen as a strenghtening of \Cref{cor:ergodic full group moves elements of same (co)-measure around}, as we directly get that $A$ can be sent to $B$ by an involution of $\mathbb{G}$.
	
	If $(X,\lambda)$ is a standard $\sigma$-finite space however, one has to be more cautious when manipulating Borel subsets, even those of equal measure with complements of equal measure. For instance set $(X,\lambda)=(\R,\Leb)$ and let us consider:
	\begin{center}
	\begin{tabular}{ccc}
	$\displaystyle{A \coloneqq [0,+\infty[}$ & and & $\displaystyle{B \coloneqq \bigsqcup_{n \in \N} [2n,2n+1[.}$
	\end{tabular}
	\end{center}
	We have $\lambda(A) = \lambda(B) = \lambda(X\setminus A) = \lambda(X \setminus B)= \lambda(A \setminus B) = +\infty$ but $\lambda(B \setminus A) = 0$. If $\mathbb{G}$ is an ergodic full group, by \Cref{cor:ergodic full group moves elements of same (co)-measure around} there exists $T \in \mathbb{G}$ such that $T(A) = B$. However, if there existed a measure-preserving bijection $U$ such that $U(A) = B$ and $U(B) = A$ (in particular the involutions given by \Cref{cor:exchanging involutions in ergodic full group} satisfy this property), we would have $U(A\setminus B) = U(A) \setminus U(B) = B \setminus A$, which is not possible because $\lambda(A \setminus B) \neq \lambda(B \setminus A)$.
\end{rem}

The involutions given by \Cref{cor:exchanging involutions in ergodic full group} play a crucial role in understanding the structure of ergodic full groups (see \Cref{unic} and \Cref{sec: section7}), and since they are trivial on $A \cap B$, the following definition is very natural. 

\begin{defi}
	{\label{defi: exchanging involution}}
	Let $A$ and $B$ be two Borel subsets of $X$ satisfying $A \cap B = \emptyset$, $\lambda(A) = \lambda(B)$ and $\lambda(X \setminus A) = \lambda(X \setminus B)$. An involution $U \in \Aut(X,\lambda)$ such that $U(A) = B$ and $U_{\restriction X \setminus (A \sqcup B)} = \id_{X \setminus (A \sqcup B)}$ is called a \textbf{$(A,B)$-exchanging involution}, \textit{i.e.} an involution sending $A$ to $B$.
\end{defi}

We in fact have that all involutions in $\Aut(X,\lambda)$ are exchanging involutions. The proof is immediate from \Cref{lem: better separator}.

	\begin{prop}
	[{\cite[Cor.~382F]{FremlinVol3}}]
	{\label{prop: Fre382Fa}}
Let $U$ be an involution in $\Aut(X,\lambda)$. Then there exists two disjoint Borel subsets $A$ and $B$ verifying $\lambda(A) = \lambda(B)$ and $\lambda(X \setminus A) = \lambda(X \setminus B)$, and such that $U$ is an $(A,B)$-exchanging involution. 
	\end{prop}

Putting together \Cref{prop: Fre382Fa} and the following theorem by Ryzhikov, we are able to reconstruct full groups from their exchanging-involutions.

	\begin{thm}
	[{\cite{Ryzhikov1985}, see also \cite[Cor.~5.2]{Millerthesis}}]
	{\label{3invo}}
Let $\mathbb{G}$ be a full group. For any Borel subset $C$ of $X$, any element of $\mathbb{G}_C$ can be expressed as a product of at most three involutions in $\mathbb{G}_C$.
	\end{thm}
	
We now use these involutions to get the following generalization of 	\cite[Prop.~3.12]{CarderiLM2016}.

\begin{prop}
	{\label{prop: ergodic iff wweakly dense}}
	Let $\mathbb{G} \leqslant \Aut(X,\lambda)$ be a full group. The following are equivalent:
	\begin{enumerate}[(1)]\setlength\itemsep{0em}
	\item $\mathbb{G}$ is ergodic;
	\item $\mathbb{G}_f$ is ergodic;
	\item $\mathbb{G}$ is dense in $(\Aut(X,\lambda),\tau_w)$;
	\item $\mathbb{G}_f$ is dense in $(\Aut(X,\lambda),\tau_w)$.
	\end{enumerate}
\end{prop}

\begin{proof}
	\Cref{ex: G_f ergodic iff G ergodic} establishes the equivalence of \textit{(1)} and \textit{(2)}. The implication \textit{(4)}$\implies$\textit{(3)} is immediate, and the converse comes from the already mentioned fact that $\mathbb{G}_f$ is $\tau_w$-dense in $\mathbb{G}$. It remains to show that \textit{(1)} is equivalent to \textit{(3)}, which is done similarly to \cite[Prop.~3.1]{Kechris2010}, by adapting the argument thanks to \Cref{cor:exchanging involutions in ergodic full group}. 
\end{proof}

From \Cref{prop: ergodic iff wweakly dense} and \cite[Prop.~1.2.1]{BeckerKechris1996} we can generalize \cite[Cor.~3.13]{CarderiLM2016}.

\begin{cor}
	The only ergodic full group that is Polish for the weak topology is $\Aut(X,\lambda)$.
\end{cor}

We go back to the link between Polish topologies on ergodic full groups and the weak topology in \Cref{5.1}, and on Polish topologies on $\mathbb{G}_f$ in \Cref{sec: polishability of Gf}.
We end this section with the following lemma describing the behaviour of involutions under conjugation.	
	
	\begin{lem}
	{\label{lem: involutions are conjugated}}
	Let $(X,\lambda)$ be a standard $\sigma$-finite space and $\mathbb{G} \leqslant \Aut(X,\lambda)$ be an ergodic full group. For any $C\subseteq X$, if $U$ and $V$ are two involutions in $\mathbb{G}_C$ with $\lambda(\supp U) = \lambda(\supp V)$ and $\lambda(C \setminus \supp U) = \lambda(C \setminus \supp V)$, then $U$ and $V$ are conjugated in $\mathbb{G}_C$.
\end{lem}

\begin{proof}
	By \Cref{prop: Fre382Fa} there exists two Borel subsets $A$ and $B$ such that $\supp U = A \sqcup U(A) \subseteq C$ and $\supp V = B \sqcup V(B) \subseteq C$. By \Cref{prop:pseudo full group exchange subsets}, we can find $\phi_1 \in [[\mathbb{G}]]$ such that $T(A) = B$ and $\phi_2 \in [[\mathbb{G}]]$ such that $\phi_2(C \setminus \supp U) = C \setminus \supp V$. We define $T$ as follows:
	\[
	T \coloneqq \left\{ \begin{array}{ll}
	\phi_1 & \mbox{on } A \\ 
	V \phi_1 U  &  \mbox{on } U(A) \\
	\phi_2 & \mbox{on } C \setminus \supp U \\
	\id_X & \mbox{on } X \setminus C.
	\end{array}
	\right.
	\]
	We have $TU =V \phi_1 = VT$ on $A$, $TU = \phi_1 U =VT$ on $U(A)$ and $TU = \phi_2 = VT$ on $C \setminus \supp U$. Moreover $T$ is in $\mathbb{G}_C$ by construction, so $T$ is the desired conjugation between $U$ and $V$.
\end{proof}

\section{Orbit full groups on locally finite spaces}{\label{section: Polish Orbit full groups}}

\subsection{Polish structures on orbit full groups}{\label{section: Polish structures on orbit full groups}}

Let us recall the definition of an orbit full group, which is a special case of a full group of an equivalence relation. They arise from the action of a group on $(X,\lambda)$ and represent our main examples of full groups, and as such they will be the focus of our work in this section. The whole section follows what was done in the finite measure context in \cite{CarderiLM2016}. In particular we give a infinite measure result analogous to their Theorem 3.17. We start off with a few definitions.

	\begin{defi}
	{\label{defi: eqrel fg}}
Let $\mathcal{R}$ be an equivalence relation on a standard $\sigma$-finite space $(X, \lambda)$. The set of all $T \in \Aut(X,\lambda)$ such that for almost all $x \in X , (Tx,x) \in \mathcal{R}$ is called the \textbf{full group of} $\mathcal{R}$, and is denoted by $[\mathcal{R}]$. 

Let now $G$ be a group acting on a standard $\sigma$-finite space $(X, \lambda)$ in a Borel manner. This action defines a equivalence relation $\mathcal{R}_G$, where $(x,x') \in \mathcal{R}_G$ if there exists $g \in G$ such that $x' = g \cdot x$. $[\mathcal{R}_G]$ is naturally the full group of the equivalence relation $\mathcal{R}_G$, and in other words, $T \in [\mathcal{R}_G]$ if and only if $Tx \in G \cdot x$, for almost all $x \in X$. The full group $[\mathcal{R}_G]$ is called the \textbf{orbit full group} of the $G$-action.
	\end{defi}

	\begin{rem}
	{\label{rem: ergodic orbit full group}}
A Polish group $G$ acting in an ergodic manner on $(X,\lambda)$ (in the sense that the image of $G$ in $\Aut(X,\lambda)$ through the action is ergodic) yields an ergodic orbit full group $[\mathcal{R}_G]$, but the converse is not necessarily true (see \textit{e.g.} \cite[Ex.~3.14]{CarderiLM2016}).
	\end{rem}

Let us now consider a Polish group $G$ acting in a Borel manner on a $\sigma$-finite space $(X, \lambda)$. Let us also consider the space $\Lzero(X,\lambda,(G,\tau_G))$, which is Polish when equipped with the topology of convergence in measure by \Cref{prop: measurable functions to a Polish group is Polish}. We define $\Phi: \Lzero(X,\lambda,(G,\tau_G)) \rightarrow \Lzero(X,\lambda,(X,\tau_X))$ as follows:
\[
\Phi(f)(x) \coloneqq f(x)\cdot x.
\]
We view $\Aut(X,\lambda)$ as a subspace of $\Lzero(X,\lambda,(X,\tau_X))$, and we define $\widetilde{[\mathcal{R}_G]} \coloneqq \Phi\inv (\Aut(X,\lambda))$.

The following lemma works exactly the same as in the finite measure case so we omit the proof. 

	\begin{lem}
	[{\cite[Lem.~3.15]{CarderiLM2016}}]
We have the equality $\Phi \left(  \widetilde{[\mathcal{R}_G]}  \right) = [\mathcal{R}_G]$.
	\end{lem}

The space $\widetilde{[\mathcal{R}_G]}$ is a subspace of $\Lzero(X,\lambda,(G,\tau_G))$, and as such can be equipped with the topology of convergence in measure.

\begin{defi}
The \textbf{topology of orbital convergence in measure} $\tau_m^o$ on the orbit full group $[\mathcal{R}_G]$ is the quotient topology of the topology on convergence in measure on $\widetilde{[\mathcal{R}_G]}$ by $\ker(\Phi)$.
\end{defi}

The main theorem of this section states that $([\mathcal{R}_G],\tau_m^o)$ is a Polish group,
when $G$ acts in a continuous manner on a Polish space endowed with a locally finite measure.

	\begin{thm}
	{\label{Thm: orbit full groups are Polish groups}}
Let $G$ be a Polish group acting in a continuous manner on a Polish space $(X,\tau_X)$ equipped with an atomless $\sigma$-finite measure $\lambda$ which is locally finite on $\tau_X$. The orbit full group $[\mathcal{R}_G]$, equipped with the topology of orbital convergence in measure, is a Polish group.
	\end{thm}

	\begin{proof}
Let $(f_n)$ be a sequence of elements of $\Lzero(X,\lambda,(G,\tau_G))$ converging to $f$ in measure, and by Proposition \ref{prop: cv in measure equiv}, let $(f_{n_k})$ be a subsequence such that $f_{n_k}(x) \rightarrow f(x)$ $\lambda$-almost everywhere. 
Then by continuity of the action
\[
\Phi(f_{n_k})(x) = f_{n_k}(x)\cdot x \longrightarrow f(x) \cdot x = \Phi(f)(x)
\]
for $\lambda$-almost all $x \in X$. This ensures us that $\Phi$ is continuous.

We also know that $\Aut(X,\lambda)$ is $G_\delta$ in $\Lzero(X,\lambda,(X,\tau))$ thanks to Proposition \ref{prop: topology on Aut seen as measureable functions}, and so we write $\Aut(X,\lambda) = \bigcap_{n \in \mathbb{N}} \mathcal{O}_n$. We have
$
\widetilde{[\mathcal{R}_G]}  = \Phi\inv (\Aut(X,\lambda))
% = \Phi\inv \left(\bigcap_{n \in \mathbb{N}} \mathcal{O}_n\right)
 = \bigcap_{n \in \mathbb{N}} \Phi\inv  (\mathcal{O}_n),
$
which ensures us that $\widetilde{[\mathcal{R}_G]}$ is $G_\delta$ in $\Lzero(X,\lambda,(G,\tau_G))$, by the previously proven continuity. Thanks to Proposition \ref{prop: G delta Polish}, it is Polish.

Let us now define the group structure on $\widetilde{[\mathcal{R}_G]}$. Let $f$ and $g$ be two elements of $\widetilde{[\mathcal{R}_G]}$. We define the group operation $*$, the inverse and the neutral element as follows:\\
\begin{align*}
& (f * g)(x) \coloneqq f(\Phi(g)(x))g(x) \\
& f\inv (x)  \coloneqq f(\Phi(f)\inv (x))\inv \\
& e_{\widetilde{[\mathcal{R}_G]}}: x \in X \mapsto e_G.
\end{align*}

The definition of $\Phi$ ensures us that these group operations are well defined, and Propositions \ref{prop: topology on Aut seen as measureable functions} and \ref{prop: measurable functions to a Polish group is Polish} along with the continuity of $\Phi$ ensure us that they are continuous.
For $x \in X$, we then have
\begin{align*}
\Phi(f * g)(x) & = (f * g)(x) \cdot x\\
& = (f(\Phi(g)(x))g(x))\cdot x\\
& = f(\Phi(g)(x)) \cdot (g(x) \cdot x)\\
& = f(\Phi(g)(x)) \cdot (\Phi(g)(x))\\
& = \Phi(f) ( \Phi(g)(x))
\end{align*}	
which proves that $\Phi_{\restriction \widetilde{[\mathcal{R}_G]}}$ is a group homomorphism between $\widetilde{[\mathcal{R}_G]}$ and $\Aut(X,\lambda)$. We can then write 
\[
[\mathcal{R}_G] = \Phi\left(\widetilde{[\mathcal{R}_G]}\right) \cong \faktor{\widetilde{[\mathcal{R}_G]}}{\text{Ker} ( \Phi )}.
\]

Finally, $\text{Ker}(\Phi)$ is normal and closed by continuity of $\Phi$, so from Proposition \ref{prop: quotient Polish group} we can deduce that $[\mathcal{R}_G] $ is a Polish group.
	\end{proof}
	
\begin{rem}
	Note that the multiplication in the previous proof is given by the cocyle relation verified by the lift of an automorphism $T$ by the application $\Phi$. Indeed if $c_T: X \rightarrow G$ is such that $\Phi (c_T) = T$, we have $T(x) = \Phi(c_T)(x) = c_T(x) \cdot x$. Therefore, if $T$ and $U$ are elements of $\Aut(X,\lambda)$, we have 
\[
c_{TU}(x) \cdot x = c_T(c_U(x) \cdot x) \cdot (c_U(x) \cdot x).
\]
\end{rem}

In the case of a standard probability space rather than a standard $\sigma$-finite space, it is possible to weaken the conditions of the previous theorem, by simply asking for a Borel group action, rather than a continuous one. It is indeed possible to do so by using \cite[Thm.~5.2.1]{BeckerKechris1996}, as there are no topological properties that need to be preserved. This case is detailed in  \cite[Thm.~3.17]{CarderiLM2016}. For our infinite measure setup, we need to use a different model: by using a continuous Radon model for the action, we are able to obtain the following for locally compact Polish groups.

	\begin{thm}
	{\label{Thm: orbit full groups are Polish groups LC}}
Let $G$ be a locally compact Polish group acting in a Borel measure-preserving manner on a standard $\sigma$-finite space $(X,\lambda)$. The orbit full group $([\mathcal{R}_G],\tau_m^o)$ is a Polish group.
	\end{thm}

	\begin{proof}
By \cite[Thm.~4.4]{HoareauLM2024}, the Borel measure-preserving $G$-action on $(X,\lambda)$ is isomorphic to a continuous measure-preserving $G$-action on $(Y,\eta)$, where $Y$ is locally compact Polish, and $\eta$ is Radon. In particular $\eta$ is locally finite on $\tau_Y$ by \cite[Prop.~2.21]{HoareauLM2024}. Theorem \ref{Thm: orbit full groups are Polish groups} concludes, as the $G$-action on $(Y,\eta)$ induces the same full group as the $G$-action on $(X,\lambda)$.
	\end{proof}

Let us end this section by recalling the definition of an essentially free action, and by stating how it pertains to our result.

	\begin{defi}
We say that an action of $G$ on $(X,\lambda)$ is \textbf{essentially free} if there exists a conull $G$-invariant subset $A \subseteq X$, such that for all $g \in G \setminus \{ e_G\}$ and every $x \in A$ we have $g \cdot x \neq x$. That is to say that, in restriction to $A$, the $G$-action is free.
	\end{defi}
	
	\begin{rem}
If $G$ acts in an essentially free manner on $(X,\lambda)$, the application $\Phi$ is injective, and is therefore a bijection between $\widetilde{[\mathcal{R}_G]}  $ and $[\mathcal{R}_G]$. As $\Phi$ is a continuous group homomorphism between $\widetilde{[\mathcal{R}_G]}  $ and $[\mathcal{R}_G]$, we can assert that those two groups are topologically isomorphic. By identifying $G$ with the set of constant maps $x \mapsto g$ in $\Lzero(X,\lambda,(G,\tau_G))$, we have the following.
	\end{rem}
	
	\begin{cor}
	{\label{essfree}}
Let $G$ be a Polish group and $(X,\lambda)$ be a Polish space equipped with an atomless $\sigma$-finite locally finite measure $\lambda$. Consider a continuous measure-preserving $G$-action on $(X,\lambda)$. If the action is essentially free, then $G$ embeds into $([\mathcal{R}_G],\tau_m^o)$.
	\end{cor}

\subsection{Necessity of the local finiteness}{\label{section: necessity of the local finiteness}}

In this section we specify the behaviour of the weak topology and the topology of (orbital) convergence in measure when the hypothesis of local finiteness is not satisfied. We then provide two examples.

First, notice that having up to countably many points only admitting neighbourhoods of infinite measure does not constitute an issue, as it is still possible to remove them, since the measure is atomless. The real obstacle arises when there exists a subset of positive measure comprised solely of such points. 
As we will see in \Cref{5.1}, any Polish topology on an ergodic full group has to refine the weak topology. The following shows that there is no hope of getting a Polish topology with the technique of \Cref{section: Polish structures on orbit full groups} in this situation.

\begin{prop}
	{\label{prop: non locally finite tau_m does not refine tau_w}}
	Let $X$ be a Polish space endowed with an atomless $\sigma$-finite measure $\lambda$, such that the set of elements of $X$ which only admits neighbourhoods of infinite measure has positive measure. Then the topology of convergence in measure on any ergodic full group $\mathbb{G} \leqslant \Aut(X,\lambda)$ does not refine the weak topology.
\end{prop}

	\begin{proof}
Denote by $A$ the set of elements of $X$ which only admits neighbourhoods of infinite measure:
\[
A \coloneqq \left\{ x \in X \mid \mbox{for any neighbourhood } \mathcal{N}_x \mbox{ of } x, \lambda(\mathcal{N}_x) = + \infty \right\}.
\]
First note that we can assume that $A$ has finite measure (if that is not the case, simply consider a Borel subset of $A$ of finite measure for this argument). We fix a compatible metric $d$ on $X$ and $\varepsilon > 0$. We will construct a measure-preserving bijection $T$ of $X$ that sends every element of $A$ $\varepsilon$-close to itself, but such that $d_{X,\lambda}(A,T(A))>c$, for a $c$ that does not depend on $\varepsilon$.

As $X$ is Polish, $A$ is in particular second countable so we can apply Lindelöf's lemma. Thus we have $A \subseteq \bigcup_{n \in \mathbb{N}} B(x_n , \varepsilon)$, where $B(x_n , \varepsilon)$ is the open ball of center $x_n$ and radius $\varepsilon$, and $x_n$ is in $A$ for every $n$ in $\mathbb{N}$. We set $B_0 \coloneqq B(x_0, \varepsilon) \cap A$. We have $\lambda(B(x_0 , \varepsilon)) = + \infty$, and so there exists $C_0 \subseteq B(x_0 , \varepsilon) \setminus A$, such that $\lambda(C_0) = \lambda(B_0)< + \infty$. By Corollary \ref{cor:ergodic full group moves elements of same (co)-measure around}, there exists $\varphi_0$ in $\mathbb{G}_{B(x_0,\varepsilon)}$ such that $\varphi_0(B_0) = C_0$.
For $n > 0$, we then define
\[
%\begin{array}{l}
B_{n} \coloneqq \left(B(x_n,\varepsilon) \setminus \displaystyle{\bigcup_{k = 0}^{n-1} B_k } \right) \cap A ,\\
%\skipl
%C_{n} \subseteq B(x_n,\varepsilon) \setminus \left(A \cup \displaystyle{ \bigcup_{k = 0}^{n-1}} C_k \right).\\
%\skipl
%\end{array}
\] 
and notice that $\sqcup B_n$ covers $A$. We then choose
\[
C_{n} \subseteq B(x_n,\varepsilon) \setminus  \left( A \cup \bigsqcup_{k=0}^{n-1} C_k \right)
\]
such that $\lambda(C_n) = \lambda(B_n) < + \infty$.
We then choose, still thanks to Corollary \ref{cor:ergodic full group moves elements of same (co)-measure around}, a element $\varphi_{n}$ in $\mathbb{G}_{B(x_n,\varepsilon)}$ such that $\varphi_{n}(B_{n}) = C_{n}$. We finally define $T$ as follows:\\
\[
T = \left\{ \begin{array}{ll}
\varphi_{n} & \mbox{on } B_n  \ , \ n \in \mathbb{N} \\ 
{\varphi}\inv_n  &  \mbox{on } C_n \ , \ n \in \mathbb{N} \\
\id_X & \mbox{elsewhere.}\\
\end{array}
\right.
\]
For any $x$ in $A$, $x$ and $T(x)$ are in the same ball of radius $\varepsilon$, so $d(x,T(x)) < 2 \varepsilon$, but $\lambda(A \Delta T(A)) = 2 \lambda(A)$. As $\lambda(A) > 0$, a sequence $(T_n)$ constructed by considering $\varepsilon_n \rightarrow 0$ cannot converge weakly to $\id_X$ , but will do so in measure by the characterization given in Remark \ref{rem: pseudometrics}.
	\end{proof}

\begin{ex}
Our first concrete example is in $X=\R^2$, endowed with its usual topology. We define the measure $\lambda$ as follows:
\[
\lambda = c_{\Q} \otimes \Leb_\R.
\]
We explicitely build a suitable sequence of bijections preserving $\lambda$.
 
The measure space $(X,\lambda)$ is $\sigma$-finite and the measure is atomless by construction, but every point of $X$ only admits open neighbourhoods of infinite measure. In particular $\lambda$ is not locally finite for the usual $\R^2$ topology. Let $(q_n)$ be a sequence of strictly positive rational numbers converging to $0$. We consider the vertical segments $A_n \coloneqq \left\{ (q_n,y) \mid y \in [0,1] \right \} $ and $A_{\infty} \coloneqq \left\{ (0,y) \mid y \in [0,1] \right \} $ which are of length $1$, and the measure-preserving bijections $(T_n)$ defined by
\[
T_n(x,y) = \left\{ \begin{array}{ll}
(0,y) & \mbox{if } x = q_n \\
(q_n,y) &  \mbox{if } x =0 \\
(x,y)  & \mbox{otherwise.}
\end{array}
\right.
\]
In other words, $T_n$ exchanges $A_n$ and $A_{\infty}$, and fixes the rest of the space. The sequence $(T_n)$ converges pointwise, so by Proposition \ref{prop: cv in measure equiv} $(T_n)$ converges to $\id_{X}$ in measure, however $\lambda(T_n(A_{\infty}) \Delta A_{\infty}) = \lambda(A_{\infty}) + \lambda(A_n) = 2$ and thus $(T_n)$ does not converge weakly to $\id_X$. 
\end{ex}

We now turn to our second example, which provides information on $\tau_m^o$ this time. Start by recalling that the group $G = \mathfrak{S}_\infty$ equipped with the induced topology as a subset of $\mathbb{N}^{\mathbb{N}}$ is Polish, but is not locally compact (see \textit{e.g.} \cite[I.9]{Kechris1995}).

\begin{ex}{\label{ex: action of Sinf}}
The construction is detailed in \cite[Sec.~5.4]{HoareauLM2024}. The action is that of $\mathfrak{S}_\infty$ on $X = \{ 0,1 \}^{\mathbb{N}}$ endowed with the measure
\[
\lambda \coloneqq \sum_{n \in \mathbb{N}} \mu_n \coloneqq \sum_{n \in \mathbb{N}} (p_n \delta_1 + (1-p_n) \delta_0)^{\otimes \mathbb{N}},
\]
where $p_n \in \ \left] 0 , 1 \right[$ for all $n$ in $\mathbb{N}$. We also ask that $p_n \neq p_m$ for $n \neq m$, and that $p_n \rightarrow \frac{1}{2}$. The action of $\Sinf$ on $X$ is essentially transitive (\textit{i.e.} transitive on a conull set), and thus we have $\Aut(X,\lambda) = [\mathcal{R}_{\Sinf}]$. As the group $\Aut(X,\lambda)$ has a unique Polish group topology (see \cite{Kallman1985}, \cite[Cor~2.14]{LM2022}, or section \ref{unic}), $\Aut(X,\lambda)$ equipped with any topology other than the weak topology cannot be a Polish group. It can be shown that the orbital topology of convergence in measure does not coincide with the weak topology. Therefore $\tau_m^o$ is not a Polish group toplogy on $[\mathcal{R}_{\Sinf}]$.
\end{ex}

\subsection{Induced actions on spaces of infinite measure}

In this section we provide a family of examples of infinite measure-preserving actions of Polish groups, including non locally compact ones. In particular when the space is endowed with a topology on which the measure is locally finite, \Cref{Thm: orbit full groups are Polish groups} applies. We start with a countable index subgroup acting on a (locally) finite measure space and induce the action on a countable product of the space. We quickly recall the construction of an induced (non-singular) action.\\

We let $G$ be a topological group and $H$ an open subgroup of countable index of $G$, acting in a non-singular manner on any measured space $(X,\mu)$. Let $T$ be a Borel fundamental domain for the action of $H$ on $G$ by right multiplication by the inverse. We will freely use the following identification: $T \simeq \faktor{G}{H}$.
The group $G$ naturally acts on $T$, and the action is given by 
\[
(g,t) \mapsto g \ast t \coloneqq  \mbox{ the only element of } T \cap gtH.
\]
We can then define the cocycle of the $G$-action into $H$ by setting:
\[
c : (g, t) \in G \times T \longmapsto \mbox{the only } h \in H \mbox{ such that } (gt)h\inv  \in T.
\]
Finally, we define the induced action $\alpha$ as follows: 
\[
g \cdot_{\alpha} (x, t) = (c(g,t) \cdot x , g \ast t),
\]
and the cocycle relation ensures us that it is a $G$-action on $X \times \faktor{G}{H}$. We endow $Y = X \times \faktor{G}{H}$ with the product measure of $\mu$ and the counting measure on $\faktor{G}{H}$. 

It is of note that induction preserves ergodicity, we refer to \cite[Sec.~3.2]{Tsankov2024}, which can be consulted for more on induced actions. It is straightforward to see that the following holds.

	\begin{prop}
	{\label{prop: induced action general}}
Let $G$ be a Polish group, with $H$ a closed subgroup of countable index of $G$ (in particular $H$ is open), such that $H$ acts continuously and in a measure preserving manner on a Polish space $X$ equipped with a locally finite $\sigma$-finite measure $\mu$. Then $G$ acts continuously in a measure-preserving manner on the space $Y = X \times \faktor{G}{H}$ endowed with the product of $\mu$ and the counting measure, which is still locally finite $\sigma$-finite. Moreover, if the $H$-action on $X$ is essentially free, then so is the $G$-action on $Y$.
	\end{prop}

Notice that in the previous proposition $Y$ is made of a countable number of copies of $X$, and as such $Y$ with the product measure is locally finite whenever $(X,\mu)$ is locally finite. It is in particular true for the case of an action on a standard probability space. Therefore, one can construct such an example from any continuous $G$-action on a standard probability space. All continuous non-singular actions of non-archimedean Roelcke precompact Polish groups fit our conditions. Tsankov actually proves that all boolean non-singular actions of groups of this family are isomorphic to countable disjoint copies of induced actions \cite[Thm.~3.4]{Tsankov2024}. 

As a concrete example of the situation described in \Cref{prop: induced action general}, one can consider $G = \Sinf$, and the subgroup 
\(
H_0 \coloneqq \left\{\sigma\in\Sinf\colon  \sigma(0)=0\right\}.
\)
We see $H_0$ as acting naturally on the space of sequences $X = [0,1]^{\mathbb{N}^\ast}$ by permutation of the coordinates. It is clear that $H_0$ is a subgroup of $G$ of countable index. 
The space $Y = X \times \faktor{G}{H_0}$, where $\faktor{G}{H_0}$ denotes the set of cosets for the action by right multiplication by the inverse, is endowed with the product of $\mu = (\mbox{Leb}_{\restriction [0,1]})^{\otimes \mathbb{N}}$ with the counting measure, making it a locally finite $\sigma$-finite space with infinite measure. It can also be checked that this action is essentially free, and in particular $G$ embeds into $[\mathcal{R}_{G}]$.

\subsection{Orbit full groups are complete invariants for orbit equivalence}{\label{sec: complete invariants}}

Let us define orbit equivalence for two equivalence relations $\mathcal{R}$ and $\mathcal{R}'$ on a standard $\sigma$-finite space $(X, \lambda)$. One can consult \cite{KechrisMiller2004} for more details on this topic, although it is presented in the case of countable group actions and probability spaces. In this whole section $(X,\lambda)$ will be a standard $\sigma$-finite space, $\mathcal{R}$ and $\mathcal{R}'$ will be equivalence relations on $(X,\lambda)$, and $G$ and $H$ will be two Polish groups, unless stated otherwise.

	\begin{defi}{\label{defOE}}
We say that $\mathcal{R}$ and $\mathcal{R}'$ are \textbf{orbit equivalent} if there exists a conull subset $X_0 \subseteq X$ and a Borel bijection $S: X_0 \rightarrow X_0$ which verify the following:
\begin{enumerate}\setlength\itemsep{0em}
\item $S$ preserves $\lambda$ up to multiplication by an element $k \in \R_+^\ast$ \textit{i.e.} $S_\ast \lambda = k \times \lambda$;
\item for all $(x,y) \in X_0 \times X_0$, we have $(x,y) \in \mathcal{R} \Longleftrightarrow (S(x),S(y)) \in \mathcal{R}'$.
\end{enumerate}
We also say that $S$ is an orbit equivalence between $\mathcal{R}$ and $\mathcal{R}'$.
	\end{defi}
 
	\begin{rem}
Orbit equivalence is usually defined with the previous Borel bijection $S$ being measure-preserving. This definition is a bit more general and is the correct one to consider for $\sigma$-finite spaces.
	\end{rem}

An orbit equivalence $S$ between $\mathcal{R}_G$ and $\mathcal{R}_H$ conjugates the corresponding full groups which are in particular isomorphic. Note that when that is the case, they are in fact homeomorphic by \Cref{prop: analytic morphism}. We thus have the following.

\begin{lem}
	[{\cite[Lem.~3.25]{CarderiLM2016}}]
	{\label{lem: conjugation is homeo}}
	Let $G$ and $H$ be two Polish groups acting in a Borel manner on $(X,\lambda)$. Let $S \in \Aut(X,\lambda)$ be an orbit equivalence between the orbit full groups $[\mathcal{R}_G]$ and $[\mathcal{R}_H]$. Then the conjugation by $S$ is a group homeomorphism between $([\mathcal{R}_G],\tau^o_m)$ and $([\mathcal{R}_H],\tau^o_m)$.
\end{lem}

In \cite[Thm.~3.26]{CarderiLM2016}, the authors prove that in the case of two equivalence relations on a standard probability space coming from Borel measure-preserving ergodic actions of Polish locally compact groups, the orbit full groups are not only invariants, but \textit{complete invariants} of orbit equivalence.
Their result relies heavily on Dye's reconstruction theorem \cite[Thm.~2]{Dye1963}, and to extend it to our $\sigma$-finite context, we will be needing a more general result from \cite{FremlinVol3}, which applies to groups with \emph{many involutions}. We start by noticing that full groups of Polish groups acting on standard $\sigma$-finite spaces have many involutions, which will be immediate thanks to Corollary \ref{cor:exchanging involutions in ergodic full group}.

	\begin{defi}
A subgroup $\mathbb{G}$ of $\Aut(X,\lambda)$ has \textbf{many involutions} if for every Borel subset $A \subseteq X$ of positive measure, there exists a non-trivial involution $U$ in $\mathbb{G}$ such that the support of $U$ is contained in $A$.
	\end{defi}

    \begin{prop}
    {\label{prop: ergodic action gives many involutions}}
   Any ergodic full group $\mathbb{G}  \leqslant \Aut(X,\lambda)$ has many involutions. In particular, if a Polish group $G$ is acting in a Borel measure-preserving ergodic manner on $(X,\lambda)$, then the full group $[\mathcal{R}_G]$ of the associated equivalence relation $\mathcal{R}_G$ has many involutions.
    \end{prop}

    	\begin{proof}
Take any Borel subset $A \subseteq X$ of positive measure, and consider $A_1$ and $A_2$ two disjoint Borel subsets of $A$ such that $0< \lambda(A_1) = \lambda(A_2) < + \infty$. Corollary \ref{cor:exchanging involutions in ergodic full group} applies (see Remark \ref{rem: ergodic orbit full group}), and gives us an involution supported in $A$ and mapping $A_1$ to $A_2$, which concludes the proof.	
    	\end{proof}

Let us now state Fremlin's theorem. The theorem is very general, and the following formulation, which applies to our context is from \cite[Thm.~3.18]{LM2018}. Recall that $T$ is non-singular if $T_{\ast}\lambda \in [\lambda]$, and we denote by $\Aut(X,[\lambda])$ the group of non-singular bijections of $(X,\lambda)$.

	\begin{thm}
	[{\cite[Thm.~384D]{FremlinVol3}}]
	{\label{thmFRE}}
Let $\mathbb{G}$ and $\mathbb{H}$ be two subgroups of $\Aut(X,\lambda)$ with many involutions. Then any isomorphism between $\mathbb{G}$ and $\mathbb{H}$ is the conjugation by some non-singular bijection. In other words, for any group isomorphism $\psi: \mathbb{G} \rightarrow \mathbb{H}$ , there exists  $S$ in $\Aut(X,[\lambda])$ such that for all $T \in \mathbb{G}$, we have
\[
\psi(T) = STS\inv .
\]
	\end{thm}

We will also need the following proposition, which is from \cite[Prop.~3.28]{CarderiLM2016}. Although stated in the case of a standard probability space, the proof adapts verbatim to our case, since it in fact works for the more general non-singular setup.

	\begin{prop}{\label{propOE}}
Let $G$ and $H$ be two Polish locally compact groups acting in a Borel measure-preserving manner on $(X,\lambda)$, and suppose that $[\mathcal{R}_G] \subseteq [\mathcal{R}_H]$. Then there exists a conull subset $X_0 \subseteq X$ such that 
\[
\mathcal{R}_G \cap (X_0 \times X_0) \subseteq \mathcal{R}_H.
\]
	\end{prop}

We will finally be needing the following proposition, in which the main argument can be found in \cite[Sec.~I.4]{Kechris2010}, before combining Theorem \ref{thmFRE} and Proposition \ref{propOE}.

	\begin{prop}
	{\label{propOE2}}
If $\mathbb{G}$ is an ergodic subgroup of $\Aut(X,\lambda)$, any non-singular bijection $S$ of $(X,\lambda)$ that verifies $S\mathbb{G}S\inv  \leqslant \Aut(X,\lambda)$ preserves $\lambda$ up to multiplication by an element of $\R_+^\ast$.
	\end{prop}

	\begin{proof}
As $S$ is non-singular, by the Radon-Nikodym theorem there exists a unique non-zero Borel function $f: X \to \mathbb{R}_+$ such that for any Borel subset $A$ we have
\[
S_\ast\lambda (A) = \int_A fd\lambda.
\]

We will show that $f$ is actually essentially constant. 
Let $U = STS\inv $ be an element of $S\mathbb{G}S\inv $. Start by noticing that $U$ preserves $S_{*}\lambda$. Indeed, as $T$ preserves $\lambda$, we have
\[
U_\ast S_\ast \lambda = (STS\inv )_{*}S_{*}\lambda   = S_{*}T_{*}{S\inv }_{*}S_{*}\lambda 
 = S_{*}T_{*}\lambda 
 = S_{*} \lambda.
\]
Now for any Borel subset $A$ of $X$ we have
\[
U_{*}(S_{*}\lambda)(A)  = \int_{U\inv A}f(x)d\lambda(x)
 = \int_{A}f(U\inv x)d\lambda(U\inv x) 
 = \int_{A}f(U \inv x) d\lambda(x),
\]
as $U = ST S\inv $ preserves $\lambda$, because $U  \in S\mathbb{G}S\inv  \leqslant \Aut(X,\lambda)$. This means that the Radon-Nikodym derivative of $U_{*}(S_{*}\lambda)$ is $f \circ U\inv$. We previously proved that this pushforward measure $U_{*}(S_{*}\lambda)$ is equal to $S_{*}\lambda$, so by uniqueness $f = f \circ U\inv$, which means that $f$ is constant, up to a null set, by \Cref{rem: invariant f is constant}.
	\end{proof}

We finally obtain the following.

	\begin{thm}
	{\label{Thm: orbit full groups are complete invariants of OE}}
Let $G$ and $H$ be two Polish locally compact groups acting in a Borel measure-preserving ergodic manner on a standard $\sigma$-finite space $(X,\lambda)$. Let $\Psi: [\mathcal{R}_G] \rightarrow [\mathcal{R}_H]$ be a group isomorphism. Then there exists an orbit equivalence $S$ between $\mathcal{R}_G$ and $\mathcal{R}_H$, such that $\Psi$ is the conjugation by $S$, \emph{i.e.} for all $T \in [\mathcal{R}_G]$ we have
\[
\Psi(T) = STS\inv  \in [\mathcal{R}_H].
\]
In particular, $([\mathcal{R}_G],\tau^o_m)$ and $([\mathcal{R}_H],\tau^o_m)$ are homeomorphic.
	\end{thm}

	\begin{proof}
Proposition \ref{prop: ergodic action gives many involutions} ensures us that $[\mathcal{R}_G]$ and $[\mathcal{R}_H]$ have many involutions. We apply Fremlin's theorem (Theorem \ref{thmFRE}) between $[\mathcal{R}_G]$ and $[\mathcal{R}_H]$, then we apply Proposition \ref{propOE} both ways to construct the conull subset $A$ of definition \ref{defOE}, and Proposition \ref{propOE2} to verify that the non-singular bijection $S$ preserves $\lambda$ up to a multiplication by an element of $\R_+^\ast$. Finally, $\Psi$ is an homeomorphism by \Cref{lem: conjugation is homeo}.
	\end{proof}

\begin{rem}
	Unlike in the probability measure-preserving case, we cannot conclude that orbit full groups of measure-preserving ergodic actions of Polish locally compact groups have no outer automorphisms in $\Aut(X,[\lambda])$ (see \cite{Eigen1982} and \cite[Thm.~3.26]{CarderiLM2016}), but we still get a description of these outer automorphisms: they are conjugations by bijections preserving $\lambda$ up to multiplication by an element of $\R_+^\ast$.
\end{rem}

\section{Polishability and uniqueness of the Polish topology on ergodic full groups}{\label{unic}}

\subsection{Refining the weak topology}{\label{5.1}}

The arguments of this section have been used by Kallman in \cite{Kallman1985} to prove that $\Aut(X,\lambda)$ (and $\Aut(X,\mu)$) carry a unique Polish group topology, and later by Le Maître in \cite{LM2022} to prove that the normal subgroup $\Autf(X,\lambda)$ of finitely supported elements of $\Aut(X,\lambda)$ cannot carry a Polish group topology. The main result of this section, Theorem \ref{UniquePoltopoonfg}, has been proved by Carderi and Le Maître in \cite[Sect.~4.2]{CarderiLM2016} for the case of a standard probability space, and we use the exact same arguments. The important observation was that Kallman's technique was generalizable to ergodic full groups thanks to \cite[Prop.~3.10]{CarderiLM2016}. It generalizes to our context thanks to \Cref{cor:exchanging involutions in ergodic full group}.

We fix an ergodic full group $\mathbb{G}$, subgroup of $\Aut(X,\lambda)$, where as usual $(X,\lambda)$ is a standard $\sigma$-finite space. We also fix a Hausdorff topology $\tau$ on $\mathbb{G}$.

Let us start by introducing the following notations: for any $\varepsilon > 0$ and any two Borel subsets $A, B \subseteq X$, we set
\begin{align*}
& F_{\varepsilon}^A \coloneqq \left\{  T \in  \mathbb{G} \mid \lambda(T(A) \setminus A) \leqslant \varepsilon  \right\},\\
& \mathbb{G}_{(A,B)} \coloneqq \left\{  T  \in \mathbb{G} \mid T(A) \subseteq B  \right\}.
\end{align*}
We will be needing the following lemmas, proved in \cite[Sect.~2.2]{LM2022} for $\mathbb{G} = \Aut(X,\lambda)$. The proofs adapt very well by using \Cref{cor:exchanging involutions in ergodic full group} to exhibit measure-preserving elements with conditions on their supports.

	\begin{lem}
	{\label{lem1}}
For any Borel subset $A$ of $X$, the set $\mathbb{G}_{X\setminus A}$ is $\tau$-closed.
	\end{lem}

	\begin{lem}
	{\label{lem2}}
For any Borel subsets $A$ and $B$ of $X$, the set $\mathbb{G}_{(A,B)}$ is $\tau$-closed.
	\end{lem}

	\begin{lem}
	{\label{lem3}}
Let $\varepsilon > 0$, and consider two Borel subsets $A \subseteq B \subseteq X$ such that $\lambda(B \setminus A) = \varepsilon$. We have 	$\mathbb{G}_{X \setminus A} \cdot \mathbb{G}_{(A,B)} = F_{\varepsilon}^A$.
	\end{lem}

With these lemmas we can give the following generalization of \cite[Thm.~4.7(2-3)]{CarderiLM2016}.

	\begin{thm}
	{\label{UniquePoltopoonfg}}
Let $(X,\lambda)$ be a $\sigma$-finite space, and let $\mathbb{G}$ be an ergodic full group on $(X,\lambda)$.
\begin{enumerate}\setlength\itemsep{0em}
	\item[(1)] Any Polish group topology on $\mathbb{G}$ refines the weak topology. 
	\item[(2)] The group $\mathbb{G}$ carries at most one Polish group topology.
\end{enumerate}
	\end{thm}

	\begin{proof}
\emph{(1)} Let $\tau$ be a Polish group topology on $\mathbb{G}$. Fix $\varepsilon > 0$, and $A\subseteq B$ two Borel subsets of $Y$ such that $\lambda(B \setminus A) = \varepsilon$. The previous lemmas ensures us that $F_{\varepsilon}^A$ is analytic, as the pointwise product of two closed sets. Therefore the identity map $(\mathbb{G}, \tau) \rightarrow (\Aut(X,\lambda), \tau_w)$, where $\tau_w$ is the weak topology, is Baire-measurable, hence continuous thanks to Proposition \ref{prop: analytic morphism}. So the Polish topology $\tau$ refines $\tau_w$ on $\mathbb{G}$.\\

\noindent
\emph{(2)} Let now $\tau$ and $\tau'$ be two Polish group topologies on $\mathbb{G}$. By \emph{(1)} both topologies refine the weak topology $\tau_w$, so both of the following inclusion maps are Borel:
\begin{align*}
(\mathbb{G},\tau) & \longrightarrow   (\Aut(X,\lambda), \tau_w) \\
(\mathbb{G},\tau') & \longrightarrow  (\Aut(X,\lambda), \tau_w).
\end{align*}
Applying Theorem \ref{Thm: Lusin Suslin}, we observe that any $\tau$-Borel set is a $\tau_w$-Borel set, and the same thing is true for $\tau'$. The Borel $\sigma$-algebras induced by $\tau$, $\tau'$ and $\tau_w$ are therefore the same. Proposition \ref{prop: analytic morphism} applied both ways between $(\mathbb{G},\tau)$ and $(\mathbb{G},\tau')$  concludes the proof, as it ensures us that $\tau$ and $\tau'$ both refine each other. 
	\end{proof}

\subsection{Hopf decomposition and further factorisation}

In this section we recall the notions of \textit{conservative and dissipative parts} of an infinite measure-preserving bijection, as well as the associated \textit{Hopf decomposition}. We then use it to factorise any measure-preserving bijection $T$ which has a support of infinite measure, writing it as a product of three bijections, two of which have disjoint supports with infinite measure, and the third having a support of finite measure. It is also possible to add a condition on the measure of the intersection of the supports with any fixed Borel set of positive measure. We use this factorisation in \Cref{section: coarsening the uniform topology} to decrease the known Steinhaus exponent of an ergodic full group $\mathbb{G}$ with the uniform topology.

We briefly go over the necessary definitions. We refer to \cite[\textsection~1.1]{Aaronson1997} or \cite[\textsection~1.3]{Krengel1985} for more details on this.\\

Let $T$ be a measure-preserving bijection of a standard $\sigma$-finite space $(X,\lambda)$.
By Halmos' Recurrence Theorem (\cite{Halmos1956}, see also \cite[Thm.~1.1.1]{Aaronson1997}), we can properly define the \textbf{conservative part} (sometimes also called the \textbf{recurrent part}) $\mathfrak{C}(T)$ of the space as the part where $T$ revisits any Borel set of positive measure at least once (equivalently, infinitely many times). The \textbf{dissipative part} $\mathfrak{D}(T)$ is the complement of the conservative part. In other words we have
\[
X = \mathfrak{D}(T) \sqcup \mathfrak{C}(T),
\]
which we call the \textbf{Hopf decomposition} of the space (with regards to $T$).

We furthermore say that a conservative measure-preserving bijection $T$ is \textbf{periodic} if it only admits finite orbits, and \textbf{aperiodic} if it only admits infinite orbits (in other words it has almost no fixed points). We denote by $\mathfrak{C}_f(T)$ and $\mathfrak{C}_{\infty}(T)$ the periodic and aperiodic parts, comprised of the $T$-invariant Borel sets which consist of the finite and infinite $T$-orbits, respectively.

For any measure-preserving bijection $T$, we can then refine the Hopf decomposition: 
\[
X = \mathfrak{D}(T) \sqcup \mathfrak{C}_f(T) \sqcup \mathfrak{C}_\infty(T).
\]
Moreover, as it is a partition of $X$ into $T$-invariant parts, we have $T = T_{\mathfrak{D}} T_{\mathfrak{C}_f} T_{\mathfrak{C}_{\infty}} $, where $T_{\mathfrak{D}}$ is defined by $T$ on the $T$-orbits in $\mathfrak{D}(T)$, and by $\id_X$ elsewhere, and the other constituents $T_{\mathfrak{C}_{f}}$ and $T_{\mathfrak{C}_{\infty}}$ are defined similarly. In particular, each factor is in the full group $[T]$ (in particular the support of each factor is included in $\supp T$), and those three bijections commute, as their supports consist of different $T$-orbits.

	\begin{rem}
	{\label{rem:Borel fundamental domain for dissipative or periodic}}
The $T$-action on $\supp T$ admits a Borel fundamental domain when $T$ is either conservative periodic or dissipative.
Indeed, if $\supp T \subseteq \mathfrak{D}(T)$ this is a direct consequence of the definition via wandering sets (see \cite[\textsection~1.1]{Aaronson1997}). If $\supp T \subseteq \mathfrak{C}_f(T)$, identify $(\supp T, \lambda_{\restriction \supp T})$ with $(\mathbb{R}, \mbox{Leb})$ if it has infinite measure, and with an interval with the restricition of the Lebesgue measure if it has finite measure. Taking the smallest element for the natural order on $\R$ in each orbit seen through this identifiation provides us with a suitable set.
	\end{rem}

	\begin{lem}
	{\label{lem:factor of D and Cf}}
Let $(X,\lambda)$ be a standard $\sigma$-finite space, and $T \in \Aut(X,\lambda)$ be such that the $T$-action on $X$ has a Borel fundamental domain. Let $D \subseteq X$ be of positive measure (possibly infinite). We can write $T= T_1T_2$, with $T_1$ and $T_2$ in $[T]$ satisfying the following:
\begin{enumerate}\setlength\itemsep{0em}
    \item[•] $\supp T_1$ and $\supp T_2$ are disjoint and have equal measure,
	\item[•] $\lambda(\supp T_1 \cap D) = \lambda(\supp T_2 \cap D)$.
\end{enumerate}
	\end{lem}

\begin{proof} 
	Denote by $A$ a Borel fundamental domain for the $T$-action on its support (see Remark \ref{rem:Borel fundamental domain for dissipative or periodic}). We denote by $T_{\mathrm{sat}}(C)$ the $T$-saturation of any subset $C$, and we have $T_{\mathrm{sat}}(A) = \supp T$. We consider (up to measure zero) the following set:
	\[
	A_D \coloneqq \left\{ x \in A : \abs{\mathrm{orb_T}(x) \cap D} \neq 0 \right\},
	\]
	and we will also denote by $B$ the complement of $A_D$ in $A$. There are two cases to consider, depending on the measure of $A_D$:\\
	
	\emph{(1)}. If $\lambda( A_D) = + \infty$, then it suffices to cut $A_D$ and $B$ in half: write $A_D = A_D^1 \sqcup A_D^2$ with $\lambda(A_D^1) = \lambda(A_D^2) = +\infty$ and write $B = B^1 \sqcup B^2$ with $\lambda(B^1) = \lambda(B^2)$. We define (for $i \in \left\{ 1,2 \right\}$) $T_i$ by $T$ on $T_{\mathrm{sat}}(A_D^i \sqcup B^i)$, and by $\id_X$ elsewhere. 

	We have
	\[
	\lambda(\supp T_i \cap D) = \int_X \mathds{1}_{A_D^i}(x) \times \abs{\mathrm{orb_T}(x) \cap D} d\lambda(x)  \geqslant \lambda(A_D^i) = + \infty,
	\]
	and $\supp T_1$ is disjoint from $\supp T_2$ since they are comprised of different $T$-orbits.\\
	
	\emph{(2)}. If $\lambda(A_D)<+\infty$, cutting $A_D$ in half and defining $T_1$ and $T_2$ respectively by $T$ or $\id_X$ on each part is not enough, as there 	is no guarantee that the supports will be of equal measure, let alone their intersections with $D$. 
	
	We then cut both $A_D$ and $B$ in slices that depend on the cardinal of the orbits of their points. That is to say that for any $n$ in $\N^\ast$ we set
	\[
	A_{n,D} = \left\{  x \in A_D : |\mbox{orb}_T(x)| = n  \right\} \mbox{ and } A_{\infty,D}=\left\{  x \in A_D : |\mbox{orb}_T(x) | = +\infty  \right\}
	\]
	as well as 
	\[
	B_{n} = \left\{  x \in B : |\mbox{orb}_T(x)| = n  \right\} \mbox{ and } B_{\infty}=\left\{  x \in B : |\mbox{orb}_T(x) | = +\infty  \right\}.
	\]
	We have 
	\[
	\left\lbrace
	\begin{array}{ccl}
	A_D  & = & \displaystyle{\bigsqcup_{n \in \mathbb{N}^{*}\cup\left\{ \infty \right\}} A_{n,D}}\\
	B & = & \displaystyle{\bigsqcup_{n \in \mathbb{N}^{*}\cup\left\{ \infty \right\}} B_{n}}.
	\end{array}
	\right.
	\]
	We then furthermore cut the slices $A_{n,D}$ depending on how many times the orbit of their points intersect $D$: we define $A_{k,n,D} \coloneqq \left\{ x \in A_{n,D} : \abs{\mathrm{orb}_T(x) \cap D} = k \right\}$, for any $k \leqslant n$. We then have
	\[
	A_D = \bigsqcup_{n \in \mathbb{N}^{*}\cup\left\{ \infty \right\}} \bigsqcup_{k \leqslant n} A_{k,n,D}.
	\]
	For any $n$ we can then cut $B_{n}$ in half: define $B_{n}^1$ and $B_{n}^2$ such that $B_{n} = B_{n}^1 \sqcup B_{n}^2 $ and $\lambda(B_{n}^1) = \lambda(B_{n}^2)$, and do the same for $A_{k,n,D}$, for any $k \leqslant n$. We put the two families of pieces together: for $i \in \left\{ 1,2 \right\}$ let $A_D^i$ and $B^i$ be defined by 
	\[
	\left\lbrace
	\begin{array}{l}
	\displaystyle{A_D^i = \bigsqcup_{n \in \mathbb{N}^{*}\cup\left\{ \infty \right\}}  \left( \bigsqcup_{k \leqslant n} A_{k,n,D}^i \right)}\\
	\displaystyle{B^i = \bigsqcup_{n \in \mathbb{N}^{*}\cup\left\{ \infty \right\}} B_{n}^i}
	\end{array}
	\right.
	\]
	and let (for $i \in \left\{ 1,2 \right\}$) $T_i$ be defined by $T$ on $T_{\mathrm{sat}}(A_D^i \sqcup B^i)$, and by $\id_X$ elsewhere. By construction we have
	\begin{align*}
	\lambda(\supp T_1 \cap D) & = \sum_{n \in \mathbb{N}^{*}\cup\left\{ \infty \right\}} \sum_{k \leqslant n} k \lambda(A_{k,n,D}^1) \\
	 & = \sum_{n \in \mathbb{N}^{*}\cup\left\{ \infty \right\}} \sum_{k \leqslant n} k \lambda(A_{k,n,D}^2)
	 = \lambda(\supp T_2 \cap D).
	\end{align*}
	We also have 
	\begin{align*}
	\lambda(\supp T_1 \cap T_{\mathrm{sat}}(A_D) ) & = \sum_{n \in \mathbb{N}^{*}\cup\left\{ \infty \right\}} \sum_{k \leqslant n} n \lambda(A_{k,n,D}^1) \\
	& = \sum_{n \in \mathbb{N}^{*}\cup\left\{ \infty \right\}} \sum_{k \leqslant n} n \lambda(A_{k,n,D}^2)
	 = \lambda(\supp T_2 \cap T_{\mathrm{sat}}(A_D) ).
	\end{align*}
	And finally, we have 
	\begin{align*}
	\lambda(\supp T_1 \cap T_{\mathrm{sat}}(B)) & = \sum_{n \in \mathbb{N}^{*}\cup\left\{ \infty \right\}} n \lambda(B_{n}^1)\\
	& = \sum_{n \in \mathbb{N}^{*}\cup\left\{ \infty \right\}} n \lambda(B_{n}^2) 
	 = \lambda(\supp T_2 \cap T_{\mathrm{sat}}(B)).
	\end{align*}
	
	Those equalities hold whether the quantities considered are finite or not, and the supports of $T_1$ and $T_2$ are disjoint since they are comprised of different $T$-orbits. Finally $\lambda(\supp T_1) = \lambda(\supp T_2)$ thanks to the fact that $\supp T = T_{\mathrm{sat}}(A) = T_{\mathrm{sat}}(A_D) \sqcup T_{\mathrm{sat}}(B)$. This concludes the proof, as $T_1$ and $T_2$ are in $[T]$ by construction.
\end{proof}

For the next part, we will need to recall the definition of an induced bijection.

	\begin{defi}
	{\label{def: induced bijection}}
Consider $(X,\lambda)$ a standard $\sigma$-finite space. Let $T \in \Aut(X,\lambda)$, and $A \subseteq \mathfrak{C}(T)$ be a Borel subset with $\lambda(A) > 0$. Halmos' Recurrence Theorem ensures us that for $\lambda$-almost every $x \in A$ there exists a (finite) \textbf{smallest return time} in $A$ denoted by $n_A(x)$, \emph{i.e.} 
\[
n_A(x) \coloneqq \min \{ n \in \mathbb{N}^{*} \mid T^{n}(x) \in A  \} .
\]
We can then define the \textbf{induced bijection} $T_A$ by $T_A(x) = T^{n_A(x)}(x)$ for all $x \in A$ and $T_A(x) = x$ for all $x \notin A$. Partitioning $A$ in sets of the form $\left\{ x \in A \mid n_A(x) = n \right\}$ easily yields that $T_A \in [T]$.
	\end{defi}

As stated before, the case of a $T$ which is conservative aperiodic is the most complicated, as the $T$-action does not necessarily admit a Borel fundamental domain. We instead use Rokhlin's lemma to construct a Borel subset which intersects each orbit, and has controllable measure. The following is classical.

	\begin{lem}
	{\label{rokhlin}}
Let $(X,\lambda)$ be a standard $\sigma$-finite space, and $T \in \Aut(X,\lambda)$ be conservative aperiodic on its support. For any $\varepsilon > 0$ there exists a Borel subset $C$ of $\supp T$ such that $C$ intersects $\lambda$-almost all the $T$-orbits, and $0 < \lambda( C ) < \varepsilon$. 
	\end{lem}

	\begin{proof}
The case of a finitely supported $T$ is easier and taken care of by a straightforward application of Rokhlin's lemma, so let's assume that $\supp T$ has infinite measure. Fix $\varepsilon >0$ and let $ \supp T = \bigsqcup_{n \in \mathbb{N}} X_n$, with $\lambda(X_n) = \varepsilon$ for all $n$ in $\mathbb{N}$. Do note at this point that for $\lambda$-almost all $x \in X_n$ we have $\mathrm{orb}_{T_{X_n}}(x) = \mathrm{orb}_T(x) \cap X_n$, therefore we can work with the $T_{X_n}$ orbits, which are infinite by recurrence. In each $X_n$, we apply Rokhlin's (aperiodic) lemma (see \textit{e.g.}~\cite[Thm.~172]{KalikowMccutcheon2010}) to find a subset $A_n \subseteq X_n$ such that $A_n$, $T_{X_n}(A_n)$, \ldots , $T_{X_n}^{2^{n+2}}(A_n)$ are pairwise disjoint, and such that $B_n$ the complement of the tower in $X_n$ satisfies
\[
\lambda(B_n) = \lambda \left(     X_n \setminus \bigsqcup_{k= 0}^{2^{n + 2}} T_{X_n}^k(A_n)       \right) < 2^{-(n+2)}\varepsilon.
\]
Define now $C_n = A_n \sqcup B_n$ and notice that $C_n$ intersects all the $T_{X_n}$-orbits on $X_n$, so $C = \bigsqcup_n C_n$ intersects all the $T$-orbits. By construction $\lambda(A_n)< 2^{-(n+2)}\varepsilon$ and we then have
\[
0 < \lambda(C) = \sum_{n \in \mathbb{N}} \lambda(C_n) = \sum_{n \in \mathbb{N}} (\lambda(A_n) + \lambda(B_n)) < \varepsilon,
\]
which concludes the proof.
\end{proof}

	\begin{lem}
	{\label{conservinf}}
Let $(X,\lambda)$ be a standard $\sigma$-finite space, and $T \in \Aut(X,\lambda)$ be conservative aperiodic on its support. Let $D \subseteq X$ be of positive measure (possibly infinite).
For any $\varepsilon > 0$ we can write $T = T_1T_2T_\varepsilon$, with $T_1$, $T_2$ and $T_\varepsilon$ in $[T]$ satisfying the following:
\begin{enumerate}\setlength\itemsep{0em}
	\item[•] $\supp T_1$ and $\supp T_2$ are disjoint and have equal measure,
	\item[•] $\lambda(\supp T_1 \cap D) = \lambda(\supp T_2 \cap D)$,
	\item[•] $T_\varepsilon$ is aperiodic and $\lambda(\supp T_\varepsilon)< \varepsilon$.
\end{enumerate}
	\end{lem}

	\begin{proof}
Fix $\varepsilon >0$. We start by using Lemma \ref{rokhlin} to define a Borel subset $C$ of $X$ intersecting $\lambda$-almost all the $T$-orbits, and such that $0 < \lambda(C) < \varepsilon$.

We then write $\displaystyle{T = T T_C\inv T_C}$. The support of $T_C$ has measure less than $\varepsilon$, and $\supp TT_C\inv $ has infinite measure whenever $\supp T$ does.

Let us then prove that $TT_C\inv $ has finite orbits. 

We consider one $T$-orbit, and identify it to $\mathbb{Z}$, equipped with its natural order. Let $x_1$ and $x_2$ be two consecutive elements of that orbit that are in $C$ (Halmos' Recurrence ensures us that they exist). The integer interval $\{ x_1+1, \ldots ,x_2 \}$ is $TT_C\inv $-invariant. Indeed, $TT_C\inv (x_2) = T(x_1) = x_1 +1$, and applying $TT_C\inv $ to an element between $x_1$ and $x_2$ merely moves it along the $T$-orbit, as $T_C$ is trivial outside of $C$. This means that the $TT_C\inv $-orbit containing $x_2$ is finite, and contains $x_2 - x_1 = n_C(x_1)$ elements.

Thanks to Lemma \ref{lem:factor of D and Cf}, we can then write $TT_C\inv $ as the product of two measure-preserving bijections $T_1,T_2$ with disjoint supports of equal measure and satisfying the first two parts of the statement. This concludes the proof,as $T_C = T_\varepsilon$ is suitable.
	\end{proof}

\begin{rem}
	{\label{rem: periodic dense in conservative}}
	The previous proof also yields the following, which is most likely well-known, but for which we have found no reference: for any $\varepsilon >0$, for any conservative $T$ in $\Aut(X,\lambda)$, there exists a periodic (hence conservative) $T'$ in $[T]$ such that $T'$ and $T$ are $\varepsilon$-uniformly close for $\lambda$, \textit{i.e.} $\lambda(\left\{ x \in X \mid T(x) \neq T'(x) \right\})< \varepsilon$.
\end{rem}

We can now combine the previous lemmas to obtain the following factorisation.

	\begin{prop}
	{\label{factoriseT}}
Let $(X,\lambda)$ be a standard $\sigma$-finite space, and consider $T$ in $\Aut(X,\lambda)$. Let $D \subseteq X$ be of positive measure (possibly infinite). For any $\varepsilon >0$ we can write $T = T_1 T_2 T_f$, with $T_1$, $T_2$ and $T_\varepsilon$ in $[T]$ satisfying the following:
\begin{enumerate}\setlength\itemsep{0em}
	\item[•] $\supp T_1$ and $\supp T_2$ are disjoint and of equal measure, 
	\item[•] $\lambda(\supp T_1 \cap D) = \lambda(\supp T_2 \cap D)$,
	\item[•] $T_\varepsilon$ is aperiodic and $\lambda(\supp  T_\varepsilon)< \varepsilon$.
\end{enumerate}
	\end{prop}

	\begin{proof}
We write $T = T_{\mathfrak{D}} T_{\mathfrak{C}_f} T_{\mathfrak{C}_{\infty}} $ with $T_{\mathfrak{D}}$, $T_{\mathfrak{C}_f}$ and $T_{\mathfrak{C}_{\infty}}$ commuting with each other.

We have $T_{\mathfrak{D}} = T_{\mathfrak{D}}^1 T_{\mathfrak{D}}^2$ and $T_{\mathfrak{C}_f} = T_{\mathfrak{C}_f}^1 T_{\mathfrak{C}_f}^2$ thanks to Lemma \ref{lem:factor of D and Cf}, where the supports of the two factors (for both factorisations) are disjoint, of equal measure, and meet $D$ on sets of equal measure.
Lemma \ref{conservinf} ensures us that $T_{\mathfrak{C}_{\infty}} = T_{\mathfrak{C}_{\infty}}^1 T_{\mathfrak{C}_{\infty}}^2 T_{\mathfrak{C}_{\infty}}^\varepsilon$, where the supports of $T_{\mathfrak{C}_{\infty}}^1$ and $T_{\mathfrak{C}_{\infty}}^2$ are disjoint and of equal measure, and meet $D$ on sets of equal measure, and the support of $T_{\mathfrak{C}_{\infty}}^\varepsilon$ has measure less than $\varepsilon$.

We can finally define the following measure-preserving bijections:
\[
\left\lbrace
\begin{array}{l}
T_1 = T_{\mathfrak{D}}^1 T_{\mathfrak{C}_f}^1 T_{\mathfrak{C}_{\infty}}^1\\
T_2 = T_{\mathfrak{D}}^2 T_{\mathfrak{C}_f}^2 T_{\mathfrak{C}_{\infty}}^2\\
T_{\varepsilon} = T_{\mathfrak{C}_{\infty}}^\varepsilon
\end{array}
\right.
\]
which concludes the proof in this case. Indeed by construction the supports of the factors of $T_{\mathfrak{D}}$ (respectively $T_{\mathfrak{C}_f}$, $T_{\mathfrak{C}_{\infty}}$) are included in the support of $T_{\mathfrak{D}}$ (respectively $T_{\mathfrak{C}_f}$, $T_{\mathfrak{C}_{\infty}}$) so there is no commutativity issue arising during the factorisation.
	\end{proof}

\subsection{Coarsening the uniform topology}{\label{section: coarsening the uniform topology}}

The following question is very natural to ask: When is the uniform topology on a full group a Polish group topology? 
Carderi and Le Ma\^itre gave a complete answer in the probability case (\cite[Prop.~3.8, Thm.~4.7]{CarderiLM2016}), by using a technique of Kittrell and Tsankov on automatic continuity for ergodic full groups (\cite[Thm.~3.1]{KittrellTsankov2010}). It works the same for infinite $\sigma$-finite measures.
Let us start off by recalling the following.

	\begin{prop}
	[{\cite[Lem.~5.4]{Dye1959}}]
	{\label{Dye59lem54}}
The restriction of the uniform metric to a full group of $\Aut(X,\mu)$ is complete, where $(X,\mu)$ is a standard probability space. 
	\end{prop}
	
This useful result has been used in \cite{CarderiLM2016},and in order to generalize, we will use \Cref{closedforunif} to answer the question of completeness, however separability proves to be the real obstacle. Following \cite{CarderiLM2016}, we prove that a full group can only be $\tau_u$-separable if and only if it is the full group of a countable equivalence relation. The following proof can be found in \cite[Prop.~1.25]{LMthesis}, we provide it as it is in french.

	\begin{prop}
	{\label{sepequivdnbeqrel}}
Let $\mathbb{G} \leqslant \Aut(X,\lambda)$ be a full group. It is separable for the uniform topology if and only if it is the full group of a countable equivalence relation.
	\end{prop}

	\begin{proof}
We start by fixing a probability measure $\mu$ in $[\lambda]$.

$(\Rightarrow)$ First assume that $\Gamma$ is a dense countable subgroup of $\mathbb{G}$. We can consider $[\Gamma]$ the full group generated by $\Gamma$, that is to say the saturation of $\Gamma$ with regard to cutting and pasting. $[\Gamma]$ is closed, thanks to Proposition \ref{closedforunif}, and thus is equal to $\mathbb{G}$. Therefore $\mathbb{G}$ is the full group of the countable equivalence relation given by $\Gamma$.

$(\Leftarrow)$ Now assume that $\mathbb{G}$ is equal to $[\mathcal{R}]$, where is $\mathcal{R}$ is countable. For a Borel subset $A \subseteq \mathcal{R}$, we define
\[
M_l(A) = \int_X |A_x|  d\mu(x)
\]
where $A_x = \left\{ y \in X \mid (x,y) \in A \right\} \subseteq \mbox{orb}_{\mathcal{R}}(x)$.
Notice now that for any partial isomorphism $\phi$ in $[[\mathbb{G}]]$ we have $M_l(\mbox{graph}(\phi)) = \mu(\mathrm{dom}(\phi))$. Thus, it follows from the theorem of Lusin-Novikov (see \textit{e.g.}~\cite[Thm.~18.10]{Kechris1995}) that $M_l$ is $\sigma$-finite, and it is atomless because $\mu$ is atomless and $\mathcal{R}$ is countable. Therefore, if $S$, $T$ are two elements of $\mathbb{G}$,
\[
  d_{\mu}(S,T) = \mu\left(\left\{  x \in X \mid S(x) \neq T(x)  \right\}\right)  = \dfrac{1}{2} M_l(\mbox{graph}(S) \Delta \mbox{graph}(T)).
\]
As $\hat{d}(A,B) \coloneqq M_l(A \Delta B)$ is a separable distance on $\MAlgf(\mathcal{R},M_l)$ by \Cref{prop: MAlgf is Polish}, we just proved
 that $\mathbb{G}$ is separable for the uniform topology, as a metric subspace of a separable metric space.
	\end{proof}

We now prove, following Kittrell and Tsankov in \cite{KittrellTsankov2010}, that should a full group $\mathbb{G}$ have a Polish group topology, it would necessarily be weaker that the uniform topology. We need a few preliminary results, including a key proposition from \cite{RosendalSolecki2007}, which gives us a condition called Steinhaus implying automatic continuity. It uses the notion of countable syndeticity, which we recall.

	\begin{defi}
A subset $V$ of a group $G$ is \textbf{countably syndetic} if countably many translates of $V$ cover $G$, \emph{i.e.} if there exists $(g_n)_{n \in \mathbb{N}}$ a sequence of elements of $G$ such that $\displaystyle{\bigcup_{n} g_n V = G     }$.
	\end{defi}

	\begin{prop}
	[{\cite[Prop.~2]{RosendalSolecki2007}}]
	{\label{AutCont}}
Let $G$ be a topological group. If there exists $n \in \mathbb{N}$ such that for any symmetric countably syndetic subset $V$ of $G$, $V^n$ contains a open neighbourhood of $e_G$ (we say that $G$ is \textbf{$n$-Steinhaus}), then any morphism $G \rightarrow H$, where $H$ is a separable group, is continuous.

If there exists a integer $n \in \mathbb{N}^{*}$ such that $G$ is $n$-Steinhaus, we say that $G$ is \textbf{Steinhaus}. Therefore we have the following reformulation, more concise:
any Steinhaus topological group has automatic continuity.
	\end{prop}

We also have the following result.

	\begin{lem}
	[{\cite[Lem.~494M]{FremlinVol4}}]
	{\label{Puiss3}}
Let $\mathbb{G}$ be an ergodic full group, and $V \subseteq \mathbb{G}$ a symmetric subset. Let $C$ be a Borel subset of $X$, and let $U$ and $U'$ respectively be an $(A,B)$-exchanging involution and an $(A',B')$-exchanging involution, with $A, B, A'$ and $B'$ Borel subsets of $C$ satisfying the measure conditions of Definition \ref{defi: exchanging involution}. Suppose the following:
\begin{enumerate}\setlength\itemsep{0em}
\item $\lambda(A) = \lambda(A')$ and $\lambda(C \setminus A) = \lambda(C \setminus A')$ ;
\item $U \in V$ ;
\item for all $T$ in $\mathbb{G}$ such that $\supp  T \subseteq C$, there exists a $S$ in $V$, agreeing with $T$ on $C$.
\end{enumerate}
Then $U$ and $U'$ are conjugate in $\mathbb{G}_C $ and $U'$ is in $V^3$.
	\end{lem}

We can now prove the following. The details are from Fremlin (\cite[494Y(i)]{FremlinVol4}, the proof is unpublished), based on the arguments of Kittrell and Tsankov. Proposition \ref{factoriseT} is used in a later part of the proof to reduce the required exponent of the symmetric countably syndetic subset used.

	\begin{thm}
	{\label{494Yi}}
Let $\mathbb{G}$ be an ergodic full group. Let $V$ be a symmetric countably syndetic subset of $\mathbb{G}$. Then, $V^{114}$ contains an open neighbourhood of $\id_X$ for the uniform topology. In particular, $(\mathbb{G}, \tau_u)$ is $114$-Steinhaus.
	\end{thm}

	\begin{proof}
By syndeticity, we fix a sequence $(\phi_n)$ of elements of $\mathbb{G}$ such that ${\bigcup_{n} \phi_n V = \mathbb{G}  } $. Do note at this point that $\id_X \in V^2$. Indeed, there exists $n$ in $\mathbb{N}$ such that $\id_X \in \phi_n V$, so $\phi_n\inv  \in V$. As $V$ is symmetric, this ensures us that $\id_X$ is in $V^2$.\\

\begin{claim}{\label{Claim1}}  There exists a Borel subset $C' \subseteq X$ of infinite measure, such that for all $T \in \mathbb{G}_{C'}$, there exists $S \in V^2$ such that $S_{\restriction C'} = T_{\restriction C'}$.
\end{claim}
\noindent
\textbf{Proof of Claim 1}: Let $X = \bigsqcup_{n \in \mathbb{N}} X_n$, with $\lambda(X_n) = + \infty$ for any $n$. Suppose by contradiction that for all $n \in \mathbb{N}$, there exists $T_n$ in $\mathbb{G}_{X_n}$, such that there exists no $S$ in $V^2$ satisfying $S_{\restriction X_n} = T_{n \restriction X_n}$ (we say that $T_n$ disagrees with all elements of $V^2$ on $X_n$). For any $n$ in $\mathbb{N}$, do note that $V^2 = (\phi_n V)\inv (\phi_n V)$, and $T_n = \id_X\inv  T_n$. Therefore, in order for $T_n$ to disagree with all elements of $V^2$, either $\id_X$ or $T_n$ has to disagree with every element of $\phi_n V$ on $X_n$. We can then define $T'_n \in \mathbb{G}_{X_n}$ as follows:
\[
T'_{n \restriction X_n} =  \left\{ \begin{array}{ll}
\id_{X_n} & \mbox{ if $\id_X$ disagrees on $X_n$ with every $S$ in $\phi_n V$},  \\ 
T_{n \restriction X_n} & \mbox{ if $T_n$ disagrees on $X_n$ with every $S$ in $\phi_n V$}. 
\end{array}
\right.
\]

If both $\id_X$ and $T_n$ disagree with all $S$ in $\phi_n V$, we can choose one or the other arbitrarily. Define now $T \in \mathbb{G}$ by 
\[
T_{\restriction X_n} = T'_{n \restriction X_n}.
\]

By syndeticity, there exists $m \in \mathbb{N}$ such that $T \in \phi_m V$. This is impossible, as $T$ agrees with $T'_m$ on $X_m$ by construction of $T$, but this contradicts the definition of $T'_m$, which cannot agree with an element of $\phi_m V$. Thus $C'$ can be chosen to be one of the $X_n$.\hfill $\diamondsuit$\\

\begin{claim}{\label{Claim2}} There exists an involution $U^{*} \in V^2$, with $C \coloneqq \supp  U^{*} \subseteq C'$, and $\lambda(C) = \lambda(C' \setminus C) = + \infty$.
\end{claim}
\noindent
\textbf{Proof of Claim 2}:
First consider $U  \in \mathbb{G}$, an $(A,B)$-exchanging involution between two disjoint infinite measure Borel subsets $A$ and $B$ of $C'$. Its existence is ensured by Corollary \ref{cor:exchanging involutions in ergodic full group}.

Next we construct a family $(A_t)_{t \in 
% \mathbb{Q} \, \cap \, 
\mathopen[0,  1 \mathclose]}$ of Borel subsets of $A$, with $A_0 = \emptyset$, $A_1 = A$, $A_t \subseteq A_{t'}$ (up to null sets), and $\lambda(A_{t'} \setminus A_t) = + \infty$ whenever $t<t'$. To see how to construct such a sequence, we can construct it in $(\mathbb{R}, \mbox{Leb}) \sim (A, \lambda_{\mid A})$. $A_0$ and $A_1$ are fixed, and for any $t \in 
%\mathbb{Q} \, \cap \, 
\mathopen]0, 1 \mathclose[$, we can consider 
\[
A_t = \bigcup_{n \in \mathbb{Z}} \mathopen[n, n+  t \mathclose[ \,.
\]

This gives us an uncountable family of Borel subsets verifying the aforementioned conditions. As for all $t \in \mathopen[0,  1 \mathclose]$,  $U(A_t)$ is in $B$ , which is disjoint from $A$, we can then define  
$
U_t  \in \mathbb{G},
$
an $(A_t,U(A_t))$-exchanging involution supported in $A_t \sqcup U(A_t)$ and exchanging those two subsets. Corollary \ref{cor:exchanging involutions in ergodic full group} once again ensures that $U_t$ exists in $\mathbb{G}$. By syndeticity and cardinality, there exist $s<t$ in $\mathopen[0,  1 \mathclose]$ and $n \in \mathbb{N}$ such that $U_s$ and $U_t$ are in $\phi_n V$. That means that both $\phi_{n}\inv U_s$ and $\phi_{n}\inv  U_t$ are in $V$. We finally define 
\[
U^{*} = U_{s}\inv  U_t = (\phi_n\inv  U_s)\inv  \phi_n\inv  U_t \in V^2,
\]
so that we have $\supp  U^{*} = (A_t \setminus A_s) \sqcup U(A_t \setminus A_s)$, with $\lambda(\supp  U^{*}) = \lambda(C' \setminus \supp  U^{*}) = + \infty$. \hfill $\diamondsuit$\\

\begin{claim}{\label{Claim3}} If $U \in \mathbb{G}_{C}$ is an involution, then $U \in V^{12}$. In particular, $\mathbb{G}_C \subseteq V^{36}$. 
\end{claim}
\noindent
\textbf{Proof of Claim 3}: \emph{(i)} First assume that $U$ is an involution in $\mathbb{G}_{C'}$, with $\lambda(\supp  U) = \lambda(C' \setminus \supp  U) = + \infty$. Proposition \ref{prop: Fre382Fa} ensures that there exists two disjoint Borel subsets $A$ and $B$ of $C'$, both of infinite measure, such that $U$ is an $(A,B)$-exchanging involution. The involution $U^{*}$ is also an exchanging involution, so we define $A^{*},B^{*}$ such that $U^{*}$ is an $(A^{*},B^{*})$-exchanging involution, with $\lambda(A) = \lambda(A^{*}) = \lambda(C' \setminus A ) = \lambda(C' \setminus A^{*}) = + \infty$. Claim \ref{Claim2} states that $U^{*}$ is in $V^2$ and Claim \ref{Claim1} ensures us that the third condition of Lemma \ref{Puiss3} is verified, with $V^2$ playing the role of the symmetric set. Thus $U \in (V^2)^3 = V^6$.\\

\emph{(ii)} Let now $A$ and $B$ be two Borel subsets of $C$, and let $U$ be an $(A,B)$-exchanging involution in $\mathbb{G}_C$.
If $\lambda(A) = \lambda(B)= + \infty$, fix $A_1 \subseteq A$ and $A_2 = A \setminus A_1$, with $\lambda(A_1) = \lambda(A_2) = + \infty$. For $i \in \left\{ 1,2 \right\}$, define also $B_i = U(A_i)$  and  $U_i$ an $(A_i,B_i)$-exchanging involution, and do note that $U_i \in \mathbb{G}_{C'}$, as $\supp  U_i = A_i \cup B_i \subseteq C \subseteq C'$. Then $U= U_1 U_2$ can be written as the product of two involutions verifying the conditions of \emph{(i)}, and as such is in $V^{12}$.

If $\lambda(A) = \lambda(B) < + \infty$, just consider $A'$ and $B'$ disjoint subsets of $C' \setminus C$, with $\lambda(A') = \lambda(B') = + \infty$. Fix $A_1 \subseteq A$ and $A_2 = A \setminus A_1$ such that $0< \lambda(A_1) = \lambda(A_2) < +\infty$. For $i \in \{ 1,2 \}$ define $B_i = U(A_i)$, and
\[
U_i =  \left\{ \begin{array}{l}
\mbox{an $(A_i,B_i)$-exchanging involution on } A \cup B, \\
\mbox{an $(A',B')$-exchanging involution on } A'\cup B',\\
\id_X   \mbox{ elsewhere.}
\end{array}
\right.
\]

The involutions $U_i$ are supported in $C'$ by construction, and verify the conditions of \emph{(i)}, as $\supp  U_i = A_i \sqcup B_i \sqcup A' \sqcup B' \subseteq C'$ satisfies $\lambda(\supp  U_i) = \lambda(C' \setminus \supp  U_i) = + \infty$, and $U = U_1U_2$, which proves that $U \in V^{12}$.

Finally, Theorem \ref{3invo} ensures us that $\mathbb{G}_C \subseteq V^{36}$. \hfill $\diamondsuit$\\

To conclude this proof, we now suppose that $V^{114}$ does not contain an open neighbourhood of $\id_X$. We will show that it is possible to construct a sequence of measure-preserving bijections contradicting Claim \ref{Claim3}.\\

\begin{claim}{\label{Claim4}} For any two Borel subsets $B$ and $D$ of $X$ such that $\lambda(B)< + \infty$ and $\lambda(D) = +\infty$ and any $\varepsilon > 0$, there exists a $T$ in $\mathbb{G} \setminus V^{38}$ satisfying $\lambda(B \cap \supp T) \leqslant \varepsilon$ and $\lambda(D \setminus \supp T) = + \infty$. 
\end{claim}
\noindent
\textbf{Proof of Claim 4}:
Fix $B$, $D$ and $\varepsilon$ such as defined in the statement. Recalling the definition of neighbourhoods of the identity in $(\Aut(X,\lambda),\tau_u)$ from Proposition \ref{prop: uniform topology}, negating the fact that $V^{114}$ contains an open neighbourhood of $e_{\mathbb{G}}$ exactly provides us with the following: there exists $T \in \mathbb{G} \setminus V^{114}$, such that $\lambda(B \cap \supp  T) \leqslant \varepsilon$.

If $\lambda( D \cap \supp  T) < + \infty$ then $T$ is suitable, as $V^{38} \subseteq V^{114}$. If $\lambda( D \cap \supp  T) = + \infty$, we use Proposition \ref{factoriseT} to write $T$ as a product of three measure-preserving bijections:
\[
T = T_\infty^1 T_\infty^2 T_f.
\]
$T_\infty^1$ and $T_\infty^2$ have disjoint supports of infinite measure, and $T_f$ has a support of finite measure. Moreover, these three bijections have their supports included in $\supp T$, so they intersect $B$ on sets of measure less than $\varepsilon$. For the second condition, $T_f$ obviously intersects $D$ on a set of finite measure, 
and $T_\infty^1$ and $T_\infty^2$ can be chosen such that the measures of the intersections of their supports with $D$ are infinite. Thus, all three factors of $T$ verify the conditions described in the statement of the claim, and at least one of the three is in $\mathbb{G} \setminus V^{38}$.
\hfill $\diamondsuit$\\

We now define two sequences of Borel subsets of $X$, and two sequences of measure-preserving bijections. Fix $B_0 = \emptyset$ and $D_0= C$, $T_0$ given by Claim \ref{Claim4} with $\phi_0(B_0), \phi_0(D_0)$ and $\varepsilon_0 = 2^0$, as well as $T'_0 = \phi_0 T_0 \phi_0\inv$. By induction, we define $D_{n} = D_{n-1} \setminus \supp  T'_{n-1}$ ($D_{n}$ has infinite measure) and $B_{n}$ to be a subset of finite measure such that $B_{n-1} \subseteq B_{n}$ and $\lambda(B_{n} \cap D_{n}) \geqslant {n-1}$. We have $\lambda(\phi_{n}\inv (B_{n})) = \lambda(B_{n}) < + \infty$ and $\lambda(\phi_{n}\inv (D_{n})) = \lambda(D_{n}) = + \infty$, and we can consider $T_{n}$ again provided by Claim \ref{Claim4} with $\phi_{n}(B_{n}), \phi_{n}(D_{n})$ and $\varepsilon_{n} = 2^{-n}$. We finally define $T'_{n} = \phi_{n} T_{n} \phi_{n}\inv $.

In particular for any $n \in \N$ we have the following
\begin{equation}\label{eqn:cdt}\tag{$*$}
\left\lbrace
\begin{array}{l}
\supp  T'_{n} = \phi_{n}(\supp  T_{n}) \\
\lambda(B_{n} \cap \supp  T'_{n}) \leqslant 2^{-n}\\
\lambda(D_n \setminus \supp  T'_n) = + \infty.
\end{array}
\right.
\end{equation}

\begin{claim}{\label{Claim5}} Set $\displaystyle{E = \bigcup_{n \in \mathbb{N}} \supp  T'_n }$. Then $\lambda(X \setminus E) = + \infty$. 
\end{claim}
\noindent
\textbf{Proof of Claim 5}: We denote by $E_n$ the support of $T'_{n}$, such that $E = \bigcup E_n$. Noticing that $(B_n)$ is an increasing sequence and $(D_n)$ is a decreasing sequence and using \eqref{eqn:cdt}, we have
\[
\left\lbrace
\begin{array}{lc}
\lambda(B_{n+1} \cap E_m) \leqslant \lambda(B_{m} \cap E_m) \leqslant 2^{-m} & (m \geqslant n+1)\\
\displaystyle{\lambda\left(B_{n+1} { \setminus} \bigcup_{m \leqslant n} E_m \right) \geqslant \lambda(B_{n+1} \cap D_n)} \geqslant \lambda(B_{n+1} \cap D_{n+1}) \geqslant n & (\forall \, n).
\end{array}
\right.
\]
For any $m$ we then have
\[
\lambda(B_{n+1} \setminus E) = \lambda\left(B_{n+1} \setminus \bigcup_{m \leqslant n} E_m\right) - \sum_{m \geqslant n+1} \lambda(B_{n+1} \cap E_m) \geqslant n-1
\]
which concludes, as it is true for any $n$. \hfill $\diamondsuit$\\

The complement of $E$ in $X$ has infinite measure by Claim \ref{Claim4}, so $E$ can be sent to a subset of $C$: that is to say that there exists $T \in \mathbb{G}$ such that $T(E) \subseteq C$. By syndeticity again, there exists $n \in \mathbb{N}$ such that $T\inv  \in \phi_n V$, \emph{i.e.} $T\phi_n \in V$ by symmetry. Finally, $T \phi_n T_n \phi_n\inv  T\inv  = T'T_{n}T^{-1}$ has support $T (\supp  T'_n) \subseteq T(E) \subseteq C$, which means that it is in $V^{36}$ by Claim \ref{Claim3}. This in turn means that $T_n \in V^{38}$, which contradicts the choice of $T_n$.	
	\end{proof}

Combining Theorem \ref{494Yi} and Proposition \ref{AutCont}, we obtain the following.

	\begin{thm}
	{\label{Thm: any Polish topology is coarser than the uniform}}
Let $\mathbb{G}$ be an ergodic full group. Any separable group topology on $\mathbb{G}$ is weaker than the uniform topology.
	\end{thm}

	\begin{proof}
The topological group $(\mathbb{G},\tau_u)$, where $\tau_u$ is the uniform topology, is Steinhaus thanks to Theorem \ref{494Yi}, hence it has automatic continuity thanks to Proposition \ref{AutCont}. Therefore the identity map $(\mathbb{G},\tau_u) \rightarrow (\mathbb{G},\tau)$ is continuous whenever $\tau$ is separable, which implies that $\tau_u$ refines $\tau$.
	\end{proof}

\subsection{Polishability of the finitely supported elements of an ergodic full group}{\label{sec: polishability of Gf}}

Recall that $\Autf(X,\lambda) = \left\{ T \in \Aut(X,\lambda) \mid \lambda(\supp T) < \infty \right\}$, and that $\mathbb{G}_f = \mathbb{G} \cap \Autf(X,\lambda)$, for any $\mathbb{G} \leqslant \Aut(X,\lambda)$.
	
	\begin{defi}
	{\label{def: uniform finite topology}}
	We endow $\Autf(X,\lambda)$ with $d_{u,f}$, which is defined by
	\[
	d_{u,f} : S,T \mapsto \lambda\left( \{ x \in X \mid S(x) \neq T(x) \} \right).
	\]
	It corresponds to $d_{u,X}$ if we allow subsets of infinite measure in \Cref{def: uniform topo}. It is a metric on $\Autf(X,\lambda)$ (and not on $\Aut(X,\lambda)$), and we call the \textbf{uniform finite topology} the topology induced by $d_{u,f}$, which we denote by $\tau_{u,f}$. It is immediate that $(\Autf(X,\lambda),\tau_{u,f})$ is a topological group and that $\tau_{u,f}$ refines $\tau_{u}$.
\end{defi}

In this section we consider the subgroup $\mathbb{G}_f $ of a fixed ergodic full group $\mathbb{G}$. Our goal is to prove that $\mathbb{G}_f$ cannot carry a Polish group topology, hence generalizing \cite[Thm.~2.6]{LM2022}, where the result is obtained for $\mathbb{G} = \Aut(X,\lambda)$.

Proving that $\Autf(X,\lambda)$ is not Polishable is done by displaying uncountably many measure-preserving bijections of the circle that are at a fixed distance of $3n$ of one another, contradicting the $n$-density of a countable subset of $\Autf(X,\lambda)$, which is easily established if we suppose that it is Polishable. As there is no reason for those bijections to be in $\mathbb{G}_f$, we instead prove that for any finitely supported $T$ in $\mathbb{G}_f$, there exist in $\mathbb{G}_f$ an element whose support has measure $k\lambda(\supp  T)$, for any $k$ in $\mathbb{N}$. We then combine this property with the characterization of separability for the uniform topology obtained in Proposition \ref{sepequivdnbeqrel} in order to contradict the $n$-density.

We start by defining what it means for a subgroup of a Polish group to be Polishable.

	\begin{defi}
A subgroup $H$ of a Polish group $G$ is called \textbf{Polishable} if it admits a Polish group topology which refines the topology of $G$.
	\end{defi}

	\begin{lem}
	{\label{MultSupp}}
Let $\mathbb{G}$ be an ergodic full group, and $\mathbb{G}_f = \mathbb{G} \cap \Autf(X,\lambda)$. For any $k \geqslant 1$ we have an injective group homomorphism 
$
\pi_k: T \mapsto T_k
$
from $\mathbb{G}_f$ to $\mathbb{G}_f$, such that for any $T$ in $\mathbb{G}_f$:
\[
\lambda(\supp  T_k) = k \lambda(\supp  T).
\]
Moreover, for any two measure-preserving bijections $S$ and $T$ in $\mathbb{G}_f$ we have
$
%\lambda \left( \left\{  x \in X \mid \pi_k(S)(x) \neq \pi_k(T)(x) \right\} \right) = k\lambda \left( \left\{  x \in X \mid S(x) \neq T(x) \right\} \right).
d_{u,f}(\pi_k(S) , \pi_k(T)) = k  d_{u,f}(S,T).
$
	\end{lem}

	\begin{proof}
For $k=1$ it is immediate, so let us fix $k \geqslant 2$, and $T \in \mathbb{G}_f$.

The proof is based on the the phenomenon described in Remark \ref{rem:elements of pseudo full groups are more than just restrictions}, it is possible to decompose $X$ into $k$ parts of infinite measure, such that $T$ acts on each part in the same way it acts on $X$. In this whole proof we identify $(X,\lambda) \sim (\mathbb{R}, \mbox{Leb})$ and fix $T$ an element of $\mathbb{G}_f$.

We start by dividing $\mathbb{R}$ into $k$ parts which are copies of itself, and then we map $\mathbb{R}$ to one of its copies, $k$ times. Consider for $i \in \{ 0 , \ldots , k-1 \}$ the following partial isomorphism:
\[
\begin{array}{rcccl}
\varphi_k^i &: &  \displaystyle{\bigsqcup_{n \in \mathbb{Z}} \, [n,n+1[ } & \longrightarrow & \displaystyle{\bigsqcup_{n \in \mathbb{Z}} \, [kn + i,kn+1 + i[ }\\ 
 &  & x \in [n,n+1[ & \longmapsto & x+(k-1)n + i \in [kn+i,kn+1+i[. \\
\end{array}
\]
Each segment $[n,n+1[$ is sent by $\varphi_k^i$ to the $(i+1)$th copy of itself. Each $\varphi_k^i$ has $\mathbb{R}$ as its source, and has a range of infinite measure, which is not conull. Moreover, the ranges of the $\varphi_k^i$ are disjoint, and 
\[
\bigsqcup_{i=0}^{k-1} \left( \mathrm{rng}(\varphi_k^i) \right) = \bigsqcup_{i=0}^{k-1} \left( \bigsqcup_{n \in \mathbb{Z}} \, [kn + i,kn+1 + i[ \right) \, = \mathbb{R} .
\] 
Proposition \ref{prop:pseudo full group exchange subsets} gives us for each $i \in \{ 0 , \ldots , k-1 \}$ an element $\phi_k^i$ in $[[\mathbb{G}]]$ such that $\phi_k^i(\mathbb{R}) = \mathrm{rng}(\varphi_k^i)$. We then define
\[
T_k \coloneqq (\phi_k^0)T(\phi_k^0)\inv   (\phi_k^1)T(\phi_k^1)\inv   \ldots  (\phi_k^{k-1})T(\phi_k^{k-1})\inv ,
\]
which acts as $T$ on each one of the $k$ copies of $\mathbb{R}$. We have 
\[
\supp (\phi_k^i)T(\phi_k^i)\inv = \bigsqcup_{n \in \mathbb{Z}} \, [kn + i,kn+1 + i[ 
\]
and therefore $\lambda(\supp  T_k) = k \lambda(\supp  T)$.

By construction of $\pi_k$, if $S$ is another bijection in $\mathbb{G}_f$, we have $(S(x) \neq T(x)) \Leftrightarrow (S_k(x) \neq T_k(x))$, as we're just copying the way $S$ and $T$ act, $k$ times. In particular, $\pi_k$ is injective, and we have
\[
d_{u,f}(S_k , T_k)  = \lambda\left(\left\{ x \in X \mid S_k(x) \neq T_k(x)  \right\} \right)
= k \lambda\left(\left\{ x \in X \mid S(x) \neq T(x)  \right\} \right) = k d_{u,f}(S,T).
\]
Moreover, elements acting on different copies of $\mathbb{R}$ commute, and for $\lambda$-almost all $x \in \mathrm{rng}(\varphi_k^i)$ we have 
\[
\pi_k(TS)(x) = (\phi_k^i)TS(\phi_k^i)\inv(x) = (\phi_k^i)T(\phi_k^i)\inv(\phi_k^i)S(\phi_k^i)\inv(x) = \pi_k(T)\pi_k(S)(x),
\]
so $\pi_k(TS) = \pi_k(T)\pi_k(S)$.

We conclude the proof by proving that the image $T_k$ of $T$ is indeed in $\mathbb{G}$, which is enough since its support has finite measure. It suffices to prove that each $(\phi_k^i)T(\phi_k^i)\inv $ is in $\mathbb{G}$. By definition $\phi_k^i \in [[\mathbb{G}]]$ means that it is obtained by cutting and pasting a sequence $(S_n)_{n \in \mathbb{N}}$ of elements of $\mathbb{G}$ over two partitions $(A_n)$ and $(B_n)$ of $\mathbb{R}$ and $\mathrm{rng}(\varphi_k^i)$, respectively. Therefore $(\phi_k^i)T(\phi_k^i)\inv$ is obtained by cutting and pasting the sequence $(S_n T S_n \inv)_{n \in \N}$ of elements of $\mathbb{G}$ over $(B_n)$. As this holds for any $i \in \{ 0 , \ldots , k-1 \}$, $T_k$ is in $\mathbb{G}$, and thus it is in $\mathbb{G}_f$.
	\end{proof}
	
By using this Lemma and some of the arguments used in \cite{LM2022}, we then show that $\mathbb{G}_f$ cannot carry a Polish group topology. The proof relies on Proposition \ref{sepequivdnbeqrel}, in order to work with a uncountable set of measure-preserving bijections distant enough from each other. 

	\begin{lem}
	[{\cite[Lem.~2.4]{LM2022}}]
	{\label{FepsilonClosed}}
For all $R > 0$, the set of all $T$ in $\Aut(X,\lambda)$ such that $\lambda(\supp  T) \leqslant R$ is closed in $(\Aut(X,\lambda),\tau_w)$.
	\end{lem}
	
	\begin{prop}
	{\label{NotPolishable}}
Let $\mathbb{G}$ be an ergodic full group, such that $\mathbb{G}$ is not the full group of a countable equivalence relation. The subgroup $\mathbb{G}_f \leqslant \Aut(X,\lambda)$ is not Polishable.
	\end{prop}	
	
	\begin{proof}
Suppose that $\mathbb{G}_f$ is Polishable, that is to say that there exists a topology $\tau$ on $\mathbb{G}_f$ refining $\tau_w$. For each $n$ in $\mathbb{N}$, we consider the set
\[
F_n \coloneqq \left\{ T \in \mathbb{G}_f \mid \lambda(\supp  T) \leqslant n \right\} = \left\{ T \in \Autf(X,\lambda) \mid \lambda(\supp  T) \leqslant n \right\} \cap \mathbb{G}.
\]
By Lemma \ref{FepsilonClosed} the set $\left\{ T \in \Autf(X,\lambda) \mid \lambda(\supp  T) \leqslant n \right\}$ is closed, and Lemma \ref{lem2} yields that $\mathbb{G} = \mathbb{G}_{(X,X)}$ is closed in $(\Aut(X,\lambda),\tau_w)$, so each $F_n$ is closed. We have $\mathbb{G}_f = \bigcup_{n \in \mathbb{N}}F_n$, so the Baire Category Theorem ensures us that there exists $n_0$ such that $F_{n_0}$ has nonempty interior. Since $(\mathbb{G}_f,\tau)$ is Polish, it is Lindelöf, and therefore $\mathbb{G}_f$ is covered by countably many $F_{n_0}$-translates. In other words, $\mathbb{G}_f$ contains a countable set which is $n_0$-dense for the uniform finite metric $d_{u,f}$.
Indeed $d_{u,f}$ is invariant under group operations. We have then established that any element of $\mathbb{G}_f$ is at $d_{u,f}$-distance at most $n_0$ from an element of a countable set. We will show that this is not possible.

Recall that, in restriction to $\Autf(X,\lambda)$, the uniform finite metric $d_{u,f}$ induces a finer topology than the uniform topology. 

We write $X = \bigcup_{n \in \mathbb{N}} X_n$, such that $X_n \subseteq X_{n+1}$ and $\lambda(X_n) < + \infty$ for all $n$ in $\mathbb{N}$, and we consider $\mathbb{G}_{X_n} = \left\{  T \in \mathbb{G} \mid \supp  T \subseteq X_n \right\}$. We will now show that $\mathbb{G}_f =  \bigcup_{n \in \mathbb{N}} \mathbb{G}_{X_n}$ is $\tau_u$-dense in $\mathbb{G}$. Let $U \in \mathbb{G}$ be an involution. We define $U_n$ by
\[
U_n(x) =  \left\{ \begin{array}{ll}
U(x) & \mbox{ if } x \in X_n \mbox{ and } U(x) \in X_n,  \\
x & \mbox{ else.} 
\end{array}
\right.
\]
For any $n$ in $\mathbb{N}$, the involution $U_n$ is in $\mathbb{G}_{X_n}$. Consider a probability measure $\mu$ in $[\lambda]$, and recall that $d_{\mu}$ induces $\tau_u$. It is easy to see that $U_n \rightarrow_{\tau_u} U$, as 
\begin{align*}
d_{\mu}(U_n,U) & = \mu\left(\left\{ x \in X \mid U_n(x) \neq U(x) \right\} \right)\\
& \leqslant \mu\left(\left\{ x \in X \mid x \notin X_n \mbox{ or } U(x) \notin X_n \right\} \right)\\
& \leqslant \mu(X \setminus X_n) + \mu(X \setminus U(X_n)),
\end{align*}
which tends to zero as $n$ tends to infinity. This is sufficient, as Theorem \ref{3invo} ensures us that any element $T$ in $\mathbb{G}$ can be written as a product of three involutions, so we have proved the $\tau_u$-density of $\mathbb{G}_f$.

Since $\mathbb{G}$ does not come from a countable equivalence relation, thanks to Proposition \ref{sepequivdnbeqrel} we know that it is not separable for the uniform topology. Therefore, at least one of the $\mathbb{G}_{X_n}$ is not $\tau_u$-separable. We then fix $\varepsilon' > 0$ and $n \in \mathbb{N}$ such that there are uncountably many elements $(T_t)_{t \in [0,1]}$ of $\mathbb{G}_{X_n}$ verifying $d_{\mu}(T_s,T_t) > \varepsilon$, for $s \neq t \in [0,1]$. This implies that $d_{u,f}(T_s,T_t) > \varepsilon$, for $s \neq t \in [0,1]$. We fix $s \neq t$ in $[0,1]$, and we choose $k$ in $\mathbb{N}$ such that $k \varepsilon > n_0$. Lemma \ref{MultSupp} gives us an injective application
\[
\pi_k: T \in \mathbb{G}_f \longmapsto \pi_k(T) \in \mathbb{G}_f
\]
such that $\lambda\left(\left\{ x \in X \mid \pi_k(T_s)(x) \neq \pi_k(T_t)(x)  \right\} \right)
= k \lambda\left(\left\{ x \in X \mid T_s(x) \neq T_t(x)  \right\} \right)$. We then have
\[
d_{u,f}(\pi_k(T_s),\pi_k(T_t)) = k d_{u,f}(T_s, T_t) > k \varepsilon > n_0.
\]
Thus we have found uncountably many elements of $\mathbb{G}_f$ that are strictly more than $n_0$-distant from each other. By the pigeonhole principle, this contradicts the $n_0$-density of the countable subset of $\mathbb{G}_f$ that we constructed, which concludes the proof.
	\end{proof}
	
\begin{rem}
	{\label{rem: Autf uniformly dense in Aut}}
	In the previous proof we in particular proved that for any full group $\mathbb{G}$, $\mathbb{G}_f$ is $\tau_u$-dense in $\mathbb{G}$.
\end{rem}

Combining this last result with the arguments of section \ref{5.1}, we obtain the following Theorem. Indeed, by replacing $\mathbb{G}$ by $\mathbb{G}_f$ in section \ref{5.1}, all the results hold, as it is still possible to send any Borel set of finite measure to another Borel set of the same measure with an element of $\mathbb{G}_f$, which is the only property that was used. Similarly to what we did in the proof of Theorem \ref{UniquePoltopoonfg}, we then consider $\tau$ a Polish topology on $\mathbb{G}_f$, and see that it refines the weak topology. But this is impossible according to Proposition \ref{NotPolishable}. We have then proved the following.
	
	\begin{prop}
	{\label{prop: non Polishable Gf}}
Let $\mathbb{G}$ be an ergodic full group, such that $\mathbb{G}$ is not the full group of a countable equivalence relation. The group $\mathbb{G}_f$ cannot carry a Polish group topology.
	\end{prop}

The rest of this section is dedicated to proving that an ergodic full group $\mathbb{G}$ arising from a countable equivalence relation provides a subgroup $\mathbb{G}_f$ which can be endowed with a Polish topology. In particular we obtain a characterisation of such ergodic full groups.

\begin{prop}
	{\label{prop: Gf from countable eqrel is Polish}}
	Let $\mathbb{G}$ be a full group. Then $(\mathbb{G}_f, d_{u,f})$ is complete. Moreover, if $\mathbb{G}$ is the full group of a countable equivalence relation, then $(\mathbb{G}_f, \tau_{u,f})$ is separable.
\end{prop}

\begin{proof}
\textit{(1)} We take $(T_n)$ which is $d_{u,f}$-Cauchy in $\mathbb{G}_f$, and construct the limit $T$ by Borel-Cantelli using the measure $\lambda$ instead of $\mu$. The important point is that imposing $d_{u,f}(T_n,T_{n+1}) < 2^{-n}$ (by passing to a subsequence if necessary) allows us to work in the space 
\[
\supp T_0 \cup \bigcup_{n \geqslant 0} \left\{ x \in X \mid T_n(x) \neq T_{n+1}(x)  \right\},
\]
which has finite measure. Then by construction $T$ is finitely supported and $(T_n)$ converges to $T$ for $d_{u,f}$.

\textit{(2)} For separability we mimic the proof of \Cref{sepequivdnbeqrel}$(\Leftarrow)$. The proof is very similar, but we keep $\lambda$ instead of replacing it with an equivalent probability measure. The same arguments yield that $(\MAlgf(\mathcal{R},M_l),\widehat{d})$ is a separable metric space, where $\mathcal{R}$ is the countable equivalence relation on $X$ associated with $\mathbb{G}$, $M_l$ is defined by
$
M_l(A) = \int_X \abs{\left\{ y \in X \mid (x,y) \in A \right\}} d\lambda(x)
$
for any Borel $A \subseteq \mathcal{R}$, and $\widehat{d}(A,B) = M_l(A \Delta B)$. Finally, for $S,T$ in $\mathbb{G}_f
$ we have
\[
d_{u,f}(S,T) = \dfrac{1}{2} M_l(\mbox{graph}(S) \Delta \mbox{graph}(T)),
\]
ensuring that $\mathbb{G}_f$ is separable, as a metric subspace of a separable metric space.
\end{proof}

Putting together \Cref{prop: non Polishable Gf} and \Cref{prop: Gf from countable eqrel is Polish}, we get the following characterization.

\begin{thm}
	{\label{Thm: characterisation of fg of countable eqrel}}
	Let $\mathbb{G}$ be an ergodic full group on a standard $\sigma$-finite space. The following are equivalent:
	\begin{enumerate}[(i)]\setlength\itemsep{0em}
	\item $\mathbb{G}$ is the full group of a countable equivalence relation;
	\item $\mathbb{G}_f$ can be endowed with a Polish group topology.
	\end{enumerate}
	Moreover, when it exists, the Polish group topology on $\mathbb{G}_f$ refines the uniform topology, which itself refines the weak topology.
\end{thm}

\section{Algebraic and topological properties of ergodic full groups}{\label{sec: section7}}

\subsection{Normal subgroups of an ergodic full group}

In this section we use \Cref{3invo} to recover all the information of a full group from its involutions. In particular, by using \Cref{cor:exchanging involutions in ergodic full group} and the following well-known key lemma, we are able to retrieve many elements of an ergodic full group from the involutions contained in a normal subgroup.

\begin{lem}
	{\label{lem: commutator is an involution}}
	Let $(X,\lambda)$ be a standard $\sigma$-finite space, $\mathbb{G} \leqslant \Aut(X,\lambda)$ be an ergodic full group, and $N$ be a non-trivial normal subgroup of $\mathbb{G}$. Let $T$ be in $N$, and $A$ be such that $A$ is disjoint from $T(A)$. Then for any Borel subset $B \subseteq A$ and any involution $V$ supported in $B$, the involution $U = [T,V]$ is in $N$ and satisfies $\supp U = B\sqcup T(B)$.
\end{lem}

\begin{proof}
	Fix a Borel subset $B \subseteq A$.
	By \Cref{cor:exchanging involutions in ergodic full group}, there exists an involution $V \in \mathbb{G}$ such that $\supp V = B$. Let us now prove that $U = [T,V] = TV T\inv V\inv$ is an involution in $N$ with support $B \sqcup T(B)$. We have
	\[
	(TVT\inv ) V = T(V T \inv V\inv).
	\]
	From the left-hand side we get that $T V T\inv$ is an involution supported on $T(B)$ and $V$ is an involution supported on $B$, so their product is an involution supported on $B \sqcup T(B)$. From the right-hand side, $T \in N$ and $V T\inv V\inv \in N$ because $N$ is normal, so their product is in $N$.
\end{proof}

The previous Lemma implies in particular that non-trivial normal subgroups of $\Aut(X,\lambda)$ contain finitely supported elements. It is important to keep in mind also that conjugation by elements of $\Aut(X,\lambda)$ preserve the finiteness of the support. We then have the following.

\begin{prop}
	{\label{prop: normal subgroup contains G_f}}
	Let $(X,\lambda)$ be a standard $\sigma$-finite space and $\mathbb{G} \leqslant \Aut(X,\lambda)$ be an ergodic full group. Then $\mathbb{G}_f$ is simple. 
\end{prop}

\begin{proof}
	In this whole proof we fix a non-trivial $N \trianglelefteq \mathbb{G}$.
	
	\textit{(1)} We start by proving that $N$ contains an involution with a support of measure $R$, for any $R >0$. Using \Cref{lem: commutator is an involution} and \Cref{prop: Fre382Fa}, we consider an involution $U \in N$ and a Borel subset $A\subseteq X$ of positive and finite measure such that $\supp U = A \sqcup U(A)$. Fix a positive real number $0< t \leqslant 2\lambda(\supp U)$. Let $B\subseteq A$ be such that $\lambda(B) = \frac{t}{4}$, and set $C = B \sqcup U(B)$. Let now $D$ be disjoint from $\supp U$, and such that $\lambda(D) = \lambda(C) = \frac{t}{2}$. Finally let $V \in \mathbb{G}_{C \sqcup D}$ be a $(C,D)$-exchanging involution, which exists by \Cref{cor:exchanging involutions in ergodic full group}. We have
\[
[U,V] = U V U\inv V\inv  =  \left\{ \begin{array}{ll}
\id_X & \mbox{ on }  X \setminus (C \sqcup D),  \\
U  & \mbox{ on } C, \\
VUV & \mbox{ on } D,
\end{array}
\right.
\]
	so $\supp [U,V] = C \sqcup D$, which has measure $\lambda(\supp [U,V]) = t$, and $[U,V]$ is in $N$ by  \Cref{lem: commutator is an involution}. For any $R >0$, starting with the newly obtained involution each time, iterating this construction enough times  (which is a finite number of times as we can always double the measure the support of the previous involution) yields an involution $W \in N$ with $\lambda(\supp W) = R$.
	
	\textit{(2)} We now prove that any involution in $\mathbb{G}_f$ is in $N$.
	Let $U \in \mathbb{G}_f$ be an involution, and let $V \in N$ be an involution such that $\lambda(\supp V) = \lambda(\supp U)$, which exists by \textit{(1)}. By \Cref{lem: involutions are conjugated} there exists $T \in \mathbb{G}_f$ such that $TVT\inv = U$, which proves that $U$ is in $N$.

	\textit{(3)} Let $T$ be in $\mathbb{G}_f$. By \Cref{3invo}, we can write $T$ as the product of three involutions in $\mathbb{G}_f$, which are all in $N$ by \textit{(2)}. Therefore $T$ is in $N$.
\end{proof}

We now only have take care of infinitely supported elements of $N$ in order to obtain the following generalization of \cite{Eigen1981} to any type $\mathrm{II}_{\infty}$ ergodic full group.

\begin{thm}
	{\label{Thm: normal subgroup of an ergodic fg}}
	Let $(X,\lambda)$ be a standard $\sigma$-finite space, and $\mathbb{G} \leqslant \Aut(X,\lambda)$ an ergodic full group. Let $N \trianglelefteq \mathbb{G}$ be normal and non-trivial. Then either $N = \mathbb{G}_f$ or $N = \mathbb{G}$.
\end{thm}

\begin{proof}
	By \Cref{prop: normal subgroup contains G_f} we have $\mathbb{G}_f \leqslant N$, so if $N \leqslant \mathbb{G}_f$, there is nothing to do. Assume then that there exists $T \in N$ such that $\lambda(\supp T) = + \infty$. We need to prove that $N$ contains all the involutions in $\mathbb{G}$ with supports of infinite measure, then \Cref{3invo} will conclude the proof. In order to use \Cref{lem: involutions are conjugated}, we must separately treat the cases of whether the measure of the complement of the support of an involution is infinite or finite. Let then $U_0$ and $V_0$ in $\mathbb{G}$ be such that 
	\[
	\left\lbrace
	\begin{array}{l}
	\lambda(\supp U_0) = \lambda(X \setminus \supp U_0) = +\infty \\
	\lambda(\supp V_0) = + \infty \mbox{ and } \lambda(X \setminus \supp V_0) < +\infty.
	\end{array}
	\right.
	\]
	Our goal is to prove that $U_0$ and $V_0$ are in $N$, which will imply that their respective conjugacy classes are in $N$.
	
	\textit{(1)} Using \Cref{lem: better separator} we consider a separator $A$ for $T$ (in particular $\lambda(A) = \lambda(X \setminus A) = + \infty$). Let now $U \in \mathbb{G}$ be an involution with $\supp U \subseteq A$ and $\lambda(\supp U) = \lambda(A \setminus \supp U) = + \infty$, which exists by \Cref{cor:exchanging involutions in ergodic full group}. By \Cref{lem: commutator is an involution}, $[T,U]$ is an involution in $N$, and $\lambda(\supp [T,U]) = \lambda(X \setminus \supp [T,U]) = + \infty$. By \Cref{lem: involutions are conjugated}, $U_0$ and $[T,U]$ are conjugated in $\mathbb{G}$, so $U_0$ is in $N$.

	\textit{(2)} As $\lambda(\supp V_0) = + \infty$, it is possibe to write $\supp V_0 = B \sqcup C$, with $\lambda(B) = \lambda(C) = +\infty$. By \textit{(1)} and \Cref{lem: involutions are conjugated}, we can find in $N$ two involutions $V_1$ and $V_2$ such that $\supp V_1 = B$ and $\supp V_2 = C$. The product $V_1 V_2$ is in $N$ and is an involution, its support satisfies $\lambda(\supp V_1 V_2) = \lambda(\supp V_0) = +\infty$ and $\lambda(X \setminus \supp V_1 V_2) = \lambda(X \setminus \supp V_0) < +\infty$, so by \Cref{lem: involutions are conjugated} we obtain that $V_0$ and $V_1V_2$ are conjugated in $\mathbb{G}$. Therefore $V_0$ is in $N$.
\end{proof}

\subsection{Contractibility of orbit full groups}

Again we fix a standard $\sigma$-finite space $(X,\lambda)$.
The goal of this section is to prove the following proposition.

\begin{prop}
	{\label{prop: ergodic orbit fg is contractible}}
	Let  $\mathbb{G} \leqslant \Aut(X,\lambda)$ be an ergodic orbit full group, such that $\tau_m^o$ is Polish on $\mathbb{G}$. Then $(\mathbb{G},\tau_m^o)$ is contractible.
\end{prop}

We fix such an ergodic orbit full group $\mathbb{G} = [\mathcal{R}_G]$, along with its Polish topology of orbital convergence in measure $\tau_m^o$.

We denote by $\mathcal{P}$ the family of countable ordered partitions $P = (X_n)_{n \in \N}$ of $X$ with $\lambda(X_n) = 1$ for any $n \in \N$. We endow it with the metric $r$ defined by
\[
r ( P_1 ,P_2 ) = \sum_{n \in \N} \lambda(X^1_n \Delta X^2_n)
\]
for any $P_1 = (X^1_n)_{n \in \N}, P_2 = (X^2_n)_{n \in \N}$ in $\mathcal{P}$. 
Notice that the topology induced by $r$ is the toplogy induced by the product topology on $\MAlgf(X,\lambda)^\N$, so it is Polish by \Cref{prop: MAlgf is Polish}. Note also that this topology only depends on the measure class of $\lambda$, as on $\left\{ A \in \MAlgf(X,\lambda) \mid \lambda(A) \leqslant 1 \right\}$ the topology from $\MAlgf(X,\lambda)$ coincides with the one from $\MAlg(X,\lambda)$. 

For any $P = (X_n)_{n \in \N} \in \mathcal{P}$, define
	\[
	\mathbb{G}^P = \left\{ T \in \mathbb{G} \mid \forall\, n \in \N : T(X_n) = X_n \right\}.
	\]
Until the end of this section we fix $P = (X_n)_{n \in \N}$ in $\mathcal{P}$. For any Borel subset $A$ of positive measure we also recall that $\mathbb{G}_A = \left\{ T \in \mathbb{G} \mid \supp T \subseteq A \right\}$, and naturally identify it with $\mathbb{G}_{A\restriction A} = \left\{ T_{\restriction A} \in  \Aut(A,\lambda_{\restriction A})  \mid T \in \mathbb{G}_A \right\}$.

Our immediate goal is to prove that any full group $(\mathbb{G}_A, \tau_m^o)$ is contractible (recall that $\mathbb{G}_A$ is closed in $\mathbb{G}$ by \Cref{lem1}, hence Polish for $\tau_m^o$). There is no reason for $\mathbb{G}_A$ to be an orbit full group, however in \cite[Cor.~4.3]{CarderiLM2016} the authors actually define a metric which is $\tau_m^o$-compatible, and prove contractibility through it. We recall the important points.

\begin{prop}
	[{\cite[Prop.~3.19, Cor.~3.20]{CarderiLM2016}}]
	Let $[\mathcal{R}_G]$ be an orbit full group, such that $\tau_m^o$ is Polish on $[\mathcal{R}_G]$. Fix a probability measure $\mu \in [\lambda]$, and let $d_G$ be a right-invariant compatible bounded metric on $G$. Denote $G_x \coloneqq \mathrm{stab}_G(x)$. We identify the orbit $G \cdot x$ to the homogeneous space $G / G_x$, and endow it with the metric $d_x(gG_x,g'G_x) \coloneqq \inf_{h \in G_x} d_G(gh,g')$. Then the quotient metric $d_{[\mathcal{R}_G]}$ is defined by
	\[
	d_{[\mathcal{R}_G]}(T,S) = \int_X d_x(T(x),S(x)) d\mu(x).
	\]
	Moreover, for any sequence $(T_n)$ and any $T$ in $[\mathcal{R}_G]$, the following are equivalent:
	\begin{enumerate}[(i)]\setlength\itemsep{0em}
	\item $T_n \to T$ for $\tau_m^o$;
	\item for any $\varepsilon > 0$, $\mu(\left\{ x \in X \mid d_x(T_n(x),T(x)) > \varepsilon \right\} ) \to 0$;
	\item every subsequence of $(T_n)_{n \in \N}$ admits a subsequence $(T_{n_k})_{k \in \N}$ such that for $\mu$-almost all $x$ in $X$, $T_{n_k}(x) \to T(x)$ in $G \cdot x$ for $d_x$.
	\end{enumerate}	 
\end{prop}

\begin{proof}
We start by recalling that the quotient metric $d_x$ as we defined it induces the quotient topology on $G/G_x$.
Indeed, $G_x$ is closed by \cite[Thm.~9.17]{Kechris1995}, then \cite[Lem.~2.2.8]{Gao2009} applies. Notice now that the right-invariance of $d_G$ yields the right-invariance of $\tilde{d}_G$, that we recall is defined on $\Lzero(X,\lambda,G)$ by
 $
 \tilde{d}_G(f,f') = \int_X d_G(f(x),f'(x)) d \mu(x).
 $
 
 We define 
 \[
 K = \{ f \in \Lzero(X,\lambda,G) \mid \forall \, x \in X, f(x) \in G_x \} = \text{Ker}\left(\Phi_{\restriction \widetilde{[\mathcal{R}_G]}}\right).
 \]
 Recall the definition of the $\ast$-product from the proof of \Cref{Thm: orbit full groups are Polish groups}: for any $f \in \widetilde{[\mathcal{R}_G]}$ and $k \in K$ we have
 \[
 (f \ast k) (x) = f(\Phi(k)(x))k(x) = f(x)k(x),
 \] 
 so the right $K$-cosets are the same in $\widetilde{[\mathcal{R}_G]} / K \cong [\mathcal{R}_G]$ for $\ast$ and for the pointwise product.
 By \Cref{prop: measurable functions to a Polish group is Polish} and \cite[Lem.~2.2.8]{Gao2009} again, the quotient metric induced by $\tilde{d}_G$ induces the quotient topology on $\Lzero(X,\lambda,G)/K$, and by the previous argument, the quotient metric induced by $\tilde{d}_G$ on $\widetilde{[\mathcal{R}_G]} / K$ and the quotient metric induced by $\tilde{d}_G$ on $\Lzero(X,\lambda,G)/K$ agree on $[\mathcal{R}_G]$. By definition, the latter is given by $\rho(f,f') =  \inf_{k \in K} \tilde{d}_G(fk,f')$, therefore our goal is to show that 
 \[
 \rho(f,f') = \int_X d_x(f(x),f'(x)) d\mu(x).
 \]

	We fix $\varepsilon > 0$ and $f,f' \in \Lzero(X,\lambda,G)$. For the first inequality we apply \cite[thm.~18.1]{Kechris1995} to
	\[
	\left\{ (x,g) \in X \times G  \mid d_G(f(x)g,f'(x)) < d_x(f(x),f'(x)) + \varepsilon \mbox{ and } g \cdot x = x      \right\},
	\]
	yielding a function $k \in K$ satisfying $d_G(f(x)k(x),f'(x)) < d_x(f(x),f'(x)) + \varepsilon$ for any $x \in X$. In particular, we have $\rho(f,f') \leqslant \int_X d_x(f(x),f'(x)) d\mu(x)$. The converse inequality is immediate, as for any $k \in K$ and any $x \in X$ we have $d_G(f(x)k(x),f'(x)) \geqslant d_x(f(x),f'(x))$ by definition of $K$ and $d_x$.\\
	
	The equivalence between $(i)$, $(ii)$ and $(iii)$ is proved exactly as in \cite[Prop.~6]{Moore1976} (Moore actually works with $\sigma$-finite spaces as well): $(i) \Rightarrow (ii)$ follows from Markov's inequality applied to $\mu(\left\{ x \in X \mid d_x(T_n(x),T(x)) > \varepsilon \right\} )$, $(ii) \Rightarrow (iii)$ follows from the general fact that convergence in (finite) measure implies convergence $\mu$-a.e. along a subsequence. Finally $(iii) \Rightarrow (i)$ follows from the fact that $(iii)$ gives convergence in measure of the sequence $(d_x(T_n(x),T(x)))$ to $0$, then Lebesgue's dominated convergence theorem concludes the proof.
\end{proof}

Recall the definition of the induced bijection (see \Cref{def: induced bijection}). As we already mentioned, convergence in measure only depends on the measure class, so 
\cite[Prop.~4.2]{CarderiLM2016} generalizes to our context without trouble.

\begin{prop}
	[{\cite[Prop.~4.2]{CarderiLM2016}}]
	Let $\mathbb{G}$ be an ergodic orbit full group on $(X,\lambda)$, such that $\tau_m^o$ is Polish on $\mathbb{G}$. Let $A \subseteq X$ be of finite measure. The function which maps $(B,T) \in \MAlg(A,\lambda_{\restriction A}) \times \mathbb{G}_A$ to the induced bijection $T_B \in \mathbb{G}_A$ is $\tau_m^o$-continuous.
\end{prop}

\begin{proof}
	For any $A \subseteq X$ of finite measure, $T \in \mathbb{G}_A$, and $n \in \mathbb{N}^\ast$ we define
	\[
	C_n(A,T) \coloneqq \left\{   x \in A \mid T^n(x) \in A   \right\}.
	\]
	Set also $B_0(A,T) \coloneqq X \setminus A$ and $B_n(A,T) \coloneqq  C_n(A,T) \setminus \bigcup_{m<n}C_m(A,T)$, so that $(B_n(A,T))_{n \in \N}$ is a partition of $X$ and notice that for any $x \in B_n(A,T)$ we have $T_A(x) = T^n(x)$. The important point is that $A$ has finite measure so any element of $\mathbb{G}_A$ is conservative, thus the induced bijections are well-defined, and are contained in $\mathbb{G}_A$. Note also that 
	\[
	\Psi : (A,T) \in (\MAlgf(X,\lambda) \times \mathbb{G}_A) \mapsto B_n(A,T) \in \MAlg(X,\lambda)
	\]
	depends $(d_{X,\lambda},\tau_w)$-continuously on $(A,T)$. We now fix $\varepsilon > 0$ and $(A,T) \in \MAlgf(X,\lambda) \times \mathbb{G}_{A}$. By virtue of $(B_n(A,T))_{n \in \N}$ being a partition of $X$, there exists $N > 0$ such that $\lambda(X \setminus \bigcup_{n<N}B_n(A,T)) < \varepsilon$. We now fix a probability measure $\mu \in [\lambda]$ and we let $\mathcal{U}$ be the set of couples $(A',T') \in \MAlg(A,\lambda_{\restriction A}) \times \mathbb{G}_A$ satisfying both following conditions:
	\begin{enumerate}[(1)]
    \item $\sum_{n=0}^{N} d_{[\mathcal{R}_G]} (T^n,T'^{\,n}) < \varepsilon $,
    \item $\sum_{n=0}^{N} \mu(B_n(A,T) \Delta B_n(A',T')) < \varepsilon $.
	\end{enumerate}
	By $\tau_w$-continuity of $T \mapsto T^n$, continuity of $\Psi$ and \Cref{lem: probability measure equivalent to infinite measure} (since we want a condition on $\mu$ and not on $\lambda$, even though the Borel sets in play have finite $\lambda$-measure), $\mathcal{U}$ is open. We have
	\begin{align*}
	d_{[\mathcal{R}_G]} (T_A,T'_{A'}) & = \int_X d_x(T_A(x),T'_{A'}(x)) d \mu(x) \\
	& \leqslant \sum_{n = 0}^{N} \int_{B_n(A,T) \Delta B_n(A',T') } d_x(T_A(x),T'_{A'}(x)) d \mu(x) \\
	& + \sum_{n = 0}^{N} \int_{B_n(A,T) \bigcap B_n(A',T') } d_x(T_A(x),T'_{A'}(x)) d \mu(x) \\
	& + \int_{X \setminus \bigcup_{n<N}B_n(A,T)} d_x(T_A(x),T'_{A'}(x)) d \mu(x).
	\end{align*}
	Without loss of generality it is possible to assume that $d_x$ is bounded by $1$, so by construction and (2) the first and third term are both less that $\varepsilon$. We finally have from (1):
	\begin{align*}
	d_{[\mathcal{R}_G]} (T_A,T'_{A'}) & < \sum_{n = 0}^{N} \int_{B_n(A,T) \bigcap B_n(A',T') } d_x(T_A(x),T'_{A'}(x)) d \mu(x) + 2 \varepsilon \\
	& =  \sum_{n = 0}^{N} \int_{B_n(A,T) \bigcap B_n(A',T') } d_x(T^n(x),{T'}^{\,n}(x)) d \mu(x) + 2 \varepsilon \\
	& < 3 \varepsilon.
	\end{align*}	 
	which concludes the proof.
\end{proof}

The previous proposition yields that the contraction path from \cite[Cor.~4.3]{CarderiLM2016} remains in $\mathbb{G}_A$, so we have the following.

\begin{prop}
	[{\cite[Cor.~4.3]{CarderiLM2016}}]
	{\label{prop: induced full groups are orbit full groups}}
	For any Borel subset $A$ of $X$ of positive finite measure, the full group $(\mathbb{G}_A,\tau_m^o)$ is contractible.
\end{prop}

\begin{proof}
	As $A$ has finite measure, we may assume that it is $[0,1]$, endowed with the Lebesgue measure. The homotopy between the identity map and the constant map $T \mapsto \id_A$ is given by $(s,T) \in [0,1] \times \mathbb{G}_A \to  T_{[s,1]} \in \mathbb{G}_A$.
\end{proof}

By definition of $\mathbb{G}^P$, the following application is an homeomorphism:
\[
\begin{array}{ccc}
\mathbb{G}^P & \longrightarrow & \displaystyle{ \prod_{n \in \N} \mathbb{G}_{X_n}} \\
T & \longmapsto & \displaystyle{(T_{\restriction X_n})_{n \in \N}} .\\
\end{array}
\]
By \Cref{prop: induced full groups are orbit full groups} each $\mathbb{G}_{X_n}$ is contractible, and
therefore $\mathbb{G}^P$ is contractible as a product.
We then consider the following application:
\[
\begin{array}{rcccl}
\pi^P & : & \mathbb{G} & \longrightarrow & \mathcal{P} \\
& & T & \longmapsto & (T(X_n))_{n \in \N} .\\
\end{array}
\]
It is continuous for $\tau_w$, hence for $\tau_m^o$ by \Cref{UniquePoltopoonfg}, surjective by \Cref{cor:exchanging involutions in ergodic full group} and constant on the left $\mathbb{G}^P$-cosets. 
Danilenko proved the following.

\begin{lem}
	[{\cite[Lem.~2.3]{Danilenko1995}}]
	{\label{lem: space of partitions is contractible}}
	The Polish space $(\mathcal{P},r)$ is contractible.
\end{lem}

We can now finish the proof of \Cref{prop: ergodic orbit fg is contractible} just like in \cite{Danilenko1995}.
For every partition $P' = (Y_n)_{n \in \N} \in \mathcal{P}$, by using \Cref{cor:exchanging involutions in ergodic full group} we define for any $n \in \N$, for any $x \in Y_n$:
	\[
	\psi(P')(x) = U_n(x)
	\]
	where $U_n$ is the involution in $\mathbb{G}_{Y_n \Delta X_n}$ sending $Y_n$ to $X_n$. By continuity of $(X_n,Y_n) \mapsto U_n$ (see \Cref{cor:exchanging involutions in ergodic full group}), the injective map $\psi : P' \in \mathcal{P} \mapsto \psi(P') \in \mathbb{G}$ is $\tau_u$-continuous, hence it is $\tau_m^o$-continuous.  Moreover, $\pi^P \circ \psi = \id_{\mathcal{P}}$, so $\pi^P$ has a right inverse, hence it is a continuous projection map from $\mathbb{G}$ onto $\mathcal{P}$. From this we get two things. As $\pi^P$ is surjective and constant on the left $\mathbb{G}^P$-cosets, we have the homeomorphism $\mathbb{G}/\mathbb{G}^P \cong \mathcal{P}$ (see \textit{e.g.} \cite[Thm.~22.2]{Munkres2000}). Moreover, $\pi^P \circ \psi = \id_{\mathcal{P}}$ means that $\psi$ is a global section for the fibre bundle $(\mathbb{G},\pi^P,\mathbb{G}/\mathbb{G}^P,\mathbb{G}^P)$ (see \cite[\textsection~7.4]{Steenrod1951}), which implies (by \cite[\textsection~8.3]{Steenrod1951}) that we have the homeomorphism 
	\[
	\mathbb{G} \cong  \faktor{\mathbb{G}}{\mathbb{G}^P} \times \mathbb{G}^P \cong \mathcal{P} \times \mathbb{G}^P.
	\]
	Thus $\mathbb{G}$ is contractible as a product of contractible spaces, which ends the proof of \Cref{prop: ergodic orbit fg is contractible}.\\
	
We conclude this section by the following remark: although the presence of non-conservative elements makes the proof more complicated for the contractibility of $\mathbb{G}$, there is no issue for $\mathbb{G}_f$. In particular, as \cite{Keane1970} shows continuity of $(B,T) \mapsto T_B$ for the uniform topology $\tau_u$, which coincides with $\tau_{u,f}$ (see \Cref{def: uniform finite topology}) on $\mathbb{G}_f$, we get the following.

\begin{prop}
	Let $\mathbb{G} \leqslant \Aut(X,\lambda)$ be the full group of a countable equivalence relation. Then $(\mathbb{G}_f,\tau_{u,f})$ is contractible.
\end{prop}

\subsection{Density and genericity in orbit full groups}

The goal of this section is to obtain density and category (in the sense of Baire) theorems for certain classes of measure-preserving bijections in full groups of $\Aut(X,\lambda)$. 

In \cite[Thm.~8.1]{Krengel1967} the author proved that the conservative elements form a dense $G_\delta$ set for both the uniform and the weak topology. For the rest of this section, $\mathrm{APER}$ and $\mathrm{ERG}$ will denote the set of aperiodic and ergodic elements of $\Aut(X,\lambda)$, respectively.
We have the following results for the weak and uniform topologies on $\Aut(X,\lambda)$, exhibiting different behaviours for the different topologies on $\Aut(X,\lambda)$. 

\begin{thm}
	[{\cite[Thm.~2.2 and Thm.~3.1]{Sachdeva1971}}]
	$\mathrm{ERG}$ is a dense $G_\delta$ set in $(\Aut(X,\lambda),\tau_w)$, but is nowhere dense in $(\Aut(X,\lambda),\tau_u)$.
\end{thm}

We aim to localize results of this flavour to full groups endowed with their relevant topologies.
From \Cref{rem: periodic dense in conservative} and \Cref{rem: Autf uniformly dense in Aut}, we get the following classical result. One can also consult \cite[Prop.~494C(c)]{FremlinVol4}.

\begin{prop}
	{\label{prop: density of periodic}}
	Let $\mathbb{G} \leqslant \Aut(X,\lambda)$ be a full group. Then the set of periodic bijections in $\mathbb{G}_f$ is dense for the uniform topology in $\mathbb{G}$. 
\end{prop}

In particular the previous results yield $\tau$-density when the considered full group is ergodic and has a separable topology $\tau$, by \Cref{Thm: any Polish topology is coarser than the uniform}. 
Results on aperiodic elements demand more work, but a lot has already been done for $\Aut(X,\lambda)$. We need a few adaptations to obtain result for our general Polish topologies on ergodic full groups. We have the following.

\begin{thm}
	[{\cite[Thm.~6]{ChoksiKakutani1979}}]
	{\label{thm: choksi kakutani thm 6}}
	Let $\mathbb{G}\leqslant \Aut(X,\lambda)$ be an ergodic full group.
	Let $T$ be in $\mathrm{ERG}$, and $S$ be in $\mathrm{APER}$. For any Borel subset $E \subseteq X$ of finite measure, there exists $U \in \mathbb{G}_f$ such that $T = USU\inv$ a.e. on $E$.
\end{thm}

\begin{proof}
	Choksi and Kakutani prove that there exists two families of pairwise disjoint Borel sets $(E_{k,i})$ and $(F_{k,i})$ (with $k \in \N$ and $i\in \{ 1, \ldots , k+1 \}$) satisfying the following:
	\begin{itemize}\setlength\itemsep{0em}
	\item for any $k \in \N$, for all $i \in \{ 1, \ldots , k \}$, $T(E_{k,i}) = E_{k,i+1}$ and $S(F_{k,i}) = F_{k,i+1}$;
	\item for all $k \in \N, i \in \{ 1, \ldots , k+1 \}$, $\lambda(E_{k,i}) = \lambda(F_{k,i})$;
	\item $\bigcup_{k \in \N} \bigcup_{i = 1}^k E_{k,i} = E$.
	\end{itemize}
	Their conclusion is that the conjugate of $S$ by the measure-preserving bijection $U \in \Aut(X,\lambda)$ exchanging the $E_{k,i}$ and the $F_{k,i}$ is equal to $T$ a.e. on E, but thanks to \Cref{cor:exchanging involutions in ergodic full group}, the conclusion holds even if we ask $U$ to be in $\mathbb{G}_f$.
\end{proof}

Their main result \cite[Thm.~7]{ChoksiKakutani1979} can also be extended to our context of ergodic full groups:

\begin{thm}
	[{\cite[Thm.~7]{ChoksiKakutani1979}}]
	{\label{thm: conjugacy class of an aperiodic is uniformly dense in APER}}
	Let $\mathbb{G}\leqslant \Aut(X,\lambda)$ be an ergodic full group.
	If $S$ is a fixed element in $\mathrm{APER} \cap \mathbb{G}$, the set of $\mathbb{G}_f$-conjugates of $S$, \textit{i.e.} $\{ USU\inv \mid U \in \mathbb{G}_f \}$ is $\tau_u$-dense in $\mathrm{APER} \cap \mathbb{G}$.
\end{thm}

\begin{proof}
	Fix $\varepsilon >0$ and a metric compatible with $\tau_u$ on $\Aut(X,\lambda)$, as well as $S' \in \mathrm{APER} \cap \mathbb{G}$. By \cite[Prop.~4]{ChoksiKakutani1979}, there exists $T \in \mathrm{ERG}$ (not necessarily in $\mathbb{G}$) which is $\varepsilon$-close to $S$. By \Cref{thm: choksi kakutani thm 6}, $T$ can be $\varepsilon$-approximated by a $\mathbb{G}_f$-conjugate of $S$, which means that this conjugate is $2\varepsilon$-close to $S'$.
\end{proof}

We also have the following generalization of a classical result (see \textit{e.g.} \cite[Thm.~3.5]{Kechris2010}), which is of independant interest. It is most likely well-known, but we have found no reference, so we give a quick proof using a very powerful result from \cite{Tserunyan2022}.

\begin{thm}
	{\label{thm: existence of ergodic bijections in ergodic orbit full group}}
	Let $\mathbb{G}\leqslant \Aut(X,\lambda)$ be an ergodic full group. Then $\mathrm{ERG} \cap \mathbb{G}$ is not empty.
\end{thm}

\begin{proof}
	We let $\Gamma$ be a $\tau_w$-dense countable subgroup of $\mathbb{G}$. We view $\Gamma$ as acting on $(X,\lambda)$ in a non-singular manner, and use \cite[Thm.~1.3]{Tserunyan2022} to obtain an ergodic hyperfinite non-singular subgraph of the associated equivalence relation graph. This subgraph is generated by an ergodic non-singular $T$, which is then in $\mathbb{G}$ (in particular measure-preserving) by construction.
\end{proof}

Recall that $\mathbb{G}_f$ is $\tau_u$-dense (hence $\tau_m^o$-dense when $(\mathbb{G},\tau_m^o)$ is a Polish group by \Cref{Thm: any Polish topology is coarser than the uniform}) in $\mathbb{G}$ (see \Cref{rem: Autf uniformly dense in Aut}). We have the following infinite-measure analogue of \cite[Thm.~4.4]{CarderiLM2016}. Notice that we assume ergodicity.

\begin{thm}
	{\label{thm: category-density theorem}}
	Let $\mathbb{G} = [\mathcal{R}_G]$ be an ergodic orbit full group on $(X,\lambda)$ associated with a Borel action of a Polish group $(G,\tau_G)$, such that $\tau_m^o$ is Polish on $\mathbb{G}$. The following are equivalent:
	\begin{enumerate}[(1)]\setlength\itemsep{0em}
	\item the set $\mathrm{APER} \cap \mathbb{G}$ is $\tau_m^o$-dense in $\mathbb{G}$;
	\item the set $\mathrm{ERG} \cap \mathbb{G}$ is $\tau_m^o$-dense in $\mathbb{G}$;
	\item the $\mathbb{G}_f$-conjugacy class of any element of $\mathrm{APER} \cap \mathbb{G}$ is $\tau_m^o$-dense in $\mathbb{G}$;
	\item for all $A \in \MAlg(X,\lambda)$, there is a sequence $(T_n)$ of elements of $\mathbb{G}$ such that $T_n \to \id_X$ for $\tau_m^o$, and for all $n \in \N$, $\supp T_n = A$ and $T_n$ is ergodic on its support; 
	\item for all $A \in \MAlgf(X,\lambda)$, there is a sequence $(T_n)$ of elements of $\mathbb{G}_f$ such that $T_n \to \id_X$ for $\tau_m^o$ and such that $\supp T_n = A$ for all $n \in \N$;
	\item for any sequence $(\gamma_n)$ in $\Aut(X,\lambda)$ such that any $\gamma_n$ has almost no fixed points, there is a dense $G_\delta$ set (for $\tau_m^o$) of elements $S$ in $\mathbb{G}$ such that any product-word composed alternatingly of letters in $\left\{ \gamma_n \mid n \in \N \right\}$ and $\left\{ S^k \mid  k \in \Z \setminus\{0\} \right\}$ has almost no fixed points;
	\item for any any essentially free measure-preserving action $\Gamma \curvearrowright (X,\lambda)$ of a countable discrete group $\Gamma$, there is a dense $G_\delta$ set of elements of $\mathbb{G}$ inducing an essentially free action of $\Gamma \ast \Z$;
	\item for any $n \in \N$ there is a dense $G_\delta$ set (for the product of $\tau_m^o$) of elements $(S_1, \ldots ,S_n)$ in $\mathbb{G}^n$ inducing an essentially free action of $\mathbb{F}_n$.
	\end{enumerate}
	These equivalent conditions moreover imply the following one:
	\begin{enumerate}[(1)]\setlength\itemsep{0em}
	\item[(9)] for any $\tau_G$-neighbourhood $V$ of $e_G$, the set $\bigcup_{g \in V} \supp g$ is conull;
	\end{enumerate}
	and the converse implication \textit{(9)}$\implies$\textit{(1-8)} holds if the space $X$ is endowed with a compatible Polish topology, with regards to which the measure is locally finite and the $G$-action is continuous.
\end{thm}

\begin{proof}
	We first notice that \textit{(1)}, \textit{(2)} and \textit{(3)} are all equivalent. Indeed, by \Cref{thm: existence of ergodic bijections in ergodic orbit full group} there exists an element $T$ in $\mathrm{ERG} \cap \mathbb{G}$ (in particular it is conservative aperiodic, see \cite[Prop.~1.2.1]{Aaronson1997}). This yields \textit{(3)}$\implies$\textit{(2)}. By \Cref{thm: conjugacy class of an aperiodic is uniformly dense in APER}, the $\mathbb{G}_f$ conjugacy class of the aforementioned ergodic element is $\tau_u$-dense in $\mathrm{APER} \cap \mathbb{G}$. As $\tau_u$ refines $\tau_m^o$, this yields \textit{(1)}$\implies$\textit{(3)}. Finally, \textit{(2)}$\implies$\textit{(1)} is immediate.

	\textit{(2)}$\implies$\textit{(4)}: We choose a sequence $(T_n)$ or ergodic bijections which converges to $\id_X$. For any Borel subset $A$ of positive measure the corresponding induced bijections are well defined. We fix such a subset $A$, by \cite[Prop.~4.2]{CarderiLM2016}, the sequence $(T_{n,A})$ of induced bijections on $A$ is as wanted, since any bijection is ergodic if and only its induced bijection on $A$ is ergodic (see \textit{e.g.} \cite[Prop.~1.5.2]{Aaronson1997}).
	
	\textit{(4)}$\implies$\textit{(5)} is immediate,	\textit{(6)}$\implies$\textit{(7)} is straightforward as soon as we enumerate the elements of an acting countable group $\Gamma$ and rewrite \textit{(6)} as an essentially free action. Now \textit{(7)}$\implies$\textit{(8)} is a direct induction and \textit{(1)} is a weaker reformulation of \textit{(8)} for $n=1$.\\
	
	\textit{(5)}$\implies$\textit{(6)}: The proof closely follows the one from \cite[The Category Lemma]{Tornquist2006}. We fix a partition $(X_k)$ such that $X = \sqcup_{k \in \N}X_k$ and $\lambda(X_k) = 1$ for any $k \in \N$ for the rest of the proof. 
For any $(\varepsilon_i) \to 0$, $T \in \Aut(X,\lambda)$ has almost no fixed points if and only if
	\[
	d_{u,X_k}(T,\id_X) = \lambda(\left\{ x \in X_k \mid T(x) \neq x \right\}) > 1 - \varepsilon_i
	\]
for any $k,i \in \N$ (see \Cref{def: uniform topo} for the definition of pseudometrics inducing the uniform topology).
We will need the following two preliminary lemmas, we prove the first one (as our statement is a bit different) and refer to Törnquist's paper for the proof of the second one.

\begin{lem}
	[{\cite{Tornquist2006}},\cite{CarderiLM2016}]
	{\label{lem: observation Tornquist}}
	Let $\mathbb{G} \leqslant \Aut(X,\lambda)$ be an ergodic orbit full group satisfying condition \textit{(5)} of \textnormal{\Cref{thm: category-density theorem}}. If $(A_i)_{i \in \mathcal{I}}$ is a finite family of disjoint Borel subsets of $X$ of finite measure, then for any $\tau^o_m$-neighbourhood $\mathcal{N}$ of $\id_X$ in $\mathbb{G}$ there is an element $T \in \mathcal{N}$ such that $\supp T = \sqcup_{i \in \mathcal{I}} A_i$ and $T(A_i) = A_i$ for any $i \in \mathcal{I}$.
\end{lem}

\noindent \textit{Proof.} 
	From condition \textit{(5)}, for each $i \in \mathcal{I}$ we get a sequence $(T_n^i)$ in $\mathbb{G}$ converging to $\id_X$ for $\tau_m^o$, and with $\supp T_n^i = A_i$ (for any $n \in \N$). The sequence $(\prod_{i \in \mathcal{I}} T^i_n)$ still converges to $\id_X$, and has $\sqcup_{i \in \mathcal{I}} A_i$ as its support, so for $n$ large enough $T \coloneqq \prod_{i \in \mathcal{I}} T^i_n \in \mathcal{N}$, and is adequate. \hfill $\diamondsuit$\\

\begin{lem}
	[{\cite[Lem.~2]{Tornquist2006}}]
	{\label{lem: lemma Tornquist}}
	Let $\nu$ be a finite Borel measure on $X$. Let $n \in \N^\ast$, and $T_1, \ldots T_n$ be Borel maps from $X$ to itself such that 
	\[
	\nu \left( \left\{ x \in X \mid \forall \, i,j \in \left\{ 1,\ldots, n\right\} : i\neq j \Rightarrow T_i(x) \neq T_j(x)  \right\} \right) > K > 0.
	\] 
	Then there exist finitely many disjoint non-null Borel subsets $E_1,\ldots,E_m$ such that the two following conditions are satisfied:
	$\nu(\bigsqcup_{l\leqslant m} E_l) > K$ and $T_i(E_l) \cap T_j(E_l) = \emptyset$ for any $l \leqslant m$ whenever $i \neq j$.
\end{lem}

With these two lemmas, we can now prove the analogue of the Category Lemma from \cite{Tornquist2006}, using only condition \textit{(5)}. The proof adapts with a few modifications.\\

\noindent \textbf{Proof of \textit{(5)}$\implies$\textit{(6)}:} We consider $( \gamma_n )_{n  \in \N}$ as in \textit{(6)}, $a$ a generator of the group $\Z$, and $\mathscr{A}_1 \cup \mathscr{A}_2 \coloneqq \left\{ \gamma_n \mid n \in \N \right\} \cup \left\{ a^k \mid k \in \Z\setminus\{ 0\} \right\}$ an alphabet. Then $\mathbb{F}_{alt}$ is the set of words formed with letters in $\mathscr{A}_1$ and $\mathscr{A}_2$ alternatingly. 
We will however write $a^k$ as $a \ldots a$ ($k$ times), and say that only $a$ and $a\inv$ are letters of $\mathrm{w}$. For instance the word
\[
\mathrm{w} = aa \gamma_{i_8} a\inv a\inv \gamma_{i_5} a a a \gamma_{i_1}
\] 
has length $\mathrm{len}(\mathrm{w}) = 10$, and notice that we count the letters from right to left, as we see them as functions (here the last letter of $\mathrm{w}$ is $a$). 
For a word $\mathrm{w} \in \mathbb{F}_{alt}$, we consider the evaluation map $e_\mathrm{w} : \Aut(X,\lambda) \to \Aut(X,\lambda)$ sending $S$ to the product-word obtained by replacing $a$ by $S$, and sends any $S$ to $\id_X$ if $\mathrm{w}$ is the empty word. By virtue of $\Aut(X, \lambda)$ being a topological group for $\tau_w$, the map $e_\mathrm{w}$ is $\tau_w$-continuous. We fix a non-trivial word $\mathrm{w}$ of length $\mathrm{len}(\mathrm{w}) = n$, and our goal is to show that 
\[
\left\{ S \in \mathbb{G} \mid  \supp e_\mathrm{w}(S) = X \right\}
\]
is a dense $G_\delta$ set in $\mathbb{G}$ for $\tau_m^o$. This is done by induction on $\mathrm{len}(\mathrm{w})$. We assume that the property holds for any word $\eta$ such that $\mathrm{len}(\eta) < n$. We need to show that for any $\varepsilon > 0$, the set
\[
%G_{k,m} \coloneqq \left\{ S \in \mathbb{G} \mid d'_{u,X_k}(e_W(S),\id_X) > 1 -  \right\} ) < \frac{1}{m} \right\}
G_{k,\varepsilon} \coloneqq \left\{ S \in \mathbb{G} \mid d_{u,X_k}(e_\mathrm{w}(S),\id_X) > 1- \varepsilon \right\}
\]
is open and dense for $\tau_m^o$. The following arguments encompass the case $\mathrm{w} = a^{\pm 1}$, which is the only case needed to initialize the induction, we add details when necessary.\\

\textit{(i)} $G_{k,\varepsilon}$ is open:
this is a consequence of the following claim.

\begin{claim2}{\label{ClaimTornquist}}
	If $T \in \Aut(X,\lambda)$ satisfies $d_{u,X_k}(T,\id_X) > K > 0$, then there is a $\tau_w$-neighbourhood $\mathcal{N}$ of $T$ such that $d_{u,X_k}(S,\id_X) > K$ for any $S \in \mathcal{N}$.
\end{claim2}
\noindent
\textbf{Proof of Claim}: 
We fix $\delta > 0$ such that $d_{u,X_k}(T,\id_X) > K + \delta$. By \Cref{lem: lemma Tornquist} applied to $T$ and $\id_X$, there exists disjoint non-null Borel subsets $E_1, \ldots, E_m$ satisfying $\lambda(\bigsqcup_{l \leqslant m} E_l) > K + \delta$ (we can assume that the $E_l$ are subsets of $X_k$), and $T(E_l) \cap E_l = \emptyset$ for any $l \leqslant m$. Consider the set
	\[
	\mathcal{N} \coloneqq  \left\{ S \in \Aut(X,\lambda) \; \middle| \; \forall \, l \leqslant m : \lambda(S(E_l) \Delta T(E_l)) < \dfrac{\delta}{m} \right\},
	\]
which is a $\tau_w$-neighbourhood of $T$.
For any $S \in \mathcal{N}$ we have $\lambda(S(E_l)\cap E_l) < \frac{\delta}{m}$, so 
\begin{align*}
d_{u,X_k}(S,\id_X) & \geqslant  \sum_{l \leqslant m} \lambda(\left\{ x \in E_l \mid S(x) \neq x \right\}) 
 \geqslant \sum_{l \leqslant m}( \lambda(E_l) - \lambda(\left\{ x \in E_l \mid S(x) = x  \right\})) \\ & \geqslant \sum_{l \leqslant m}(\lambda(E_l) - \lambda(S(E_l) \cap E_l) \geqslant \sum_{l \leqslant m} \left(\lambda(E_l) - \frac{\delta}{m}\right) > K,
\end{align*}
which means that $\mathcal{N}$ is the desired neighbourhood. \hfill $\diamondsuit$\\

The fact that $G_{k,\varepsilon}$ is $\tau_w$-open is then immediate, since it is preimage by $e_\mathrm{w}$ of the open set $\left\{  T \in \Aut(X,\lambda) \mid d_{u,X_k}(T,\id_X) > 1 - \varepsilon  \right\}$. It is then $\tau_m^o$-open by \Cref{UniquePoltopoonfg}.\\

\textit{(ii)} $G_{k,\varepsilon}$ is dense:
We denote by $\mathrm{w} = \mathrm{w}_n, \ldots , \mathrm{w}_1$ the words obtained from $\mathrm{w}$ by removing the leftmost letter when passing to $\mathrm{w}_{i}$ from $\mathrm{w}_{i+1}$. They are uniquely determined, and satisfy $\mathrm{len}(\mathrm{w}_i) = i$ and $\mathrm{w}_{i+1} = \alpha \mathrm{w}_i$ for a unique $\alpha \in \mathscr{A}_1 \cup \{ a, a\inv \}$ (for all $i<n$).
We fix a $\tau_m^o$-neighbourhood $\mathcal{N}$ of $\id_X$. By the inductive hypothesis, it is enough to exhibit an element $S'$ in $ G_{k,\varepsilon} \cap S\mathcal{N}$, where 
	\[
	S \in \bigcap_{\left\{ \eta \mid \mathrm{len}(\eta)<n \right\} }   \bigcap_{k \in \N} \left\{ S' \in \mathbb{G} \mid d_{u,X_k}(e_\eta(S'),\id_X) = 1  \right\}
	\]
	is fixed (if $n = 1$ then we fix an arbitrary $S \in \mathbb{G}$).
	We denote by $i_0$ the largest $i < n$ such that $\mathrm{w}_{i+1} = a^{\pm 1} \mathrm{w}_i$, \textit{i.e.} the rank of the letter at the right of the leftmost $a^{\pm 1}$ in $\mathrm{w}$. We can moreover (up to a noncommittal swap) assume that it is $a$ and not $a\inv$. In the example above where $\mathrm{len}(\mathrm{w}) = 10$, $i_0 = 9$, and notice that $i_0$ is necessarily equal to $n-1$ (if the last letter of $\mathrm{w}$ is $a$) or $n-2$ (if the last letter of $\mathrm{w}$ is in $\mathscr{A}_1$), and in the latter case there is nothing to do. 
	
	For any $i < j < n$, by the induction hypothesis we have $e_{\mathrm{w}_i}(S) \neq e_{\mathrm{w}_j}(S)$ $\lambda$-a.e. (in particular $\lambda_{\restriction X_k}$-a.e.), so by \Cref{lem: lemma Tornquist} we can find finitely many disjoint Borel subsets $E_1, \ldots, E_m$ of $X_k$ such that 
	\[
	\left\lbrace
	\begin{array}{l}\label{eqn:conditions for the generalized category Lemma}
	\lambda\left(\bigsqcup_{l \leqslant m} E_l \right) > 1-\varepsilon\\
	\tag{$\star$} \displaystyle{e_{\mathrm{w}_i}(S)(E_l) \cap e_{\mathrm{w}_j}(S)(E_l) = \emptyset \mbox{ for any } i < j < n \mbox{ and any } l \leqslant m.}
	\end{array}
	\right.
	\]
	We define the finite family $(A_i)_{i \in \mathcal{I}_1}$ as the family consisting of all the Borel sets of the form $e_{\mathrm{w}_i}(S)(E_l)$ for $i<n$ and $l \leqslant m$. If $n = 1$ the second condition of \eqref{eqn:conditions for the generalized category Lemma} is empty, and so is the family $(A_i)$, and the rest of the argument is mostly unchanged. We also define 
	\[
	\mathrm{Fix}_{\mathrm{w},k}(S) \coloneqq X_k \setminus \supp e_\mathrm{w}(S)
	\]
	the set in $X_k$ of all points fixed by the $\mathrm{w}$-evaluation of $S$, up to measure $0$. 
	
	If $\lambda(\mathrm{Fix}_{\mathrm{w},k}(S))< \varepsilon$, that means that $d_{u,X_k}(e_\mathrm{w}(S),\id_X) > 1 - \varepsilon$ and there is nothing to do. By using the first line of \eqref{eqn:conditions for the generalized category Lemma} we can therefore assume that $\lambda(\mathrm{Fix}_{\mathrm{w},k}(S) \cap E_l)> 0$ for a certain $l \leqslant m$, and up to a renaming we can assume that $\lambda(\mathrm{Fix}_{\mathrm{w},k}(S) \cap E_1)> 0$.
	
	We then let $B = e_{\mathrm{w}_{i_0}}(S)(\mathrm{Fix}_{\mathrm{w},k}(S) \cap E_1)$. In particular when $n=1$, $B = \mathrm{Fix}_{\mathrm{w},k}(S) \cap E_1$, as $e_{\mathrm{w}_{i_0}}(S) = \id_X$. We apply \Cref{lem: observation Tornquist} to the elements of the (finite) partition generated by $ \left\{ B \cap A_i \mid i \in \mathcal{I}_1 \right\}$, which is just the set $B$ in the case $n = 1$.
	 This yields a $T$ in $\mathcal{N}$ such that $\supp T = B$, $T(A_i) = A_i$ for any $i \in \mathcal{I}_1$. We then define $S_1 = ST$, which satisfies
	\[
	\left\lbrace
	\begin{array}{l}
	S_1(A_i) = S(A_i) \mbox{ for all } i \in \mathcal{I}_1,\\
	S_1 \in S\mathcal{N},\\
	\mbox{a.e. on }  E_1 \setminus \mathrm{Fix}_{\mathrm{w},k}(S), \, e_{\mathrm{w}}(S_1) = e_\mathrm{w}(S) \neq \id_X,\\
	\mbox{a.e. on }  \mathrm{Fix}_{\mathrm{w},k}(S) \cap E_1, \, e_{\mathrm{w}}(S_1) \neq e_\mathrm{w}(S) = \id_X.
	\end{array}
	\right.
	\]
	Indeed, the first two line hold by construction of $T$, and the third and fourth lines hold because $\supp T = B$ and because of the second line of \eqref{eqn:conditions for the generalized category Lemma}: $e_{\mathrm{w}_{j}}(S)(E_1)$ and $e_{\mathrm{w}_{j}}(S_1)(E_1)$ do not meet $B$ for $j < i_0$ by \eqref{eqn:conditions for the generalized category Lemma} so they agree on $E_1$, then the last application of $T$ makes them disagree (for the fourth line). Note that this whole argument holds regardless of whether $i_0 = n-1$ or $n-2$, as everything is the same up to $i_0$, and then only the application of $a^{\pm 1}$ matters. We thus have $E_1 \subseteq \supp e_{\mathrm{w}}(S_1)$.
	
	We can then apply \Cref{lem: lemma Tornquist} again to $e_{\mathrm{w}}(S_1)$ and $\id_X$ to obtain disjoint non-null Borel subsets $E'_1,\ldots E'_p$ of $E_1 \subseteq X_k$ such that 
	\[
	\lambda\left(\bigsqcup_{q \leqslant p} E'_q  \ \sqcup \bigsqcup_{1 < l \leqslant m} E_l \right) > 1- \varepsilon
	\]
	and $e_{\mathrm{w}}(S_1)(E'_q) \cap E'_q = \emptyset$ for any $q \leqslant p$.\\
	
	The proof is now almost over. We define $(A_i)_{i \in \mathcal{I}_2}$ to be the collection $(A_i)_{i \in \mathcal{I}_1}$ to which we adjoined all the $e_{\mathrm{w}_i}(S_1)(E'_q)$, for $q\leqslant p$ and $i<n$.
	
	If $\lambda(\mathrm{Fix}_{\mathrm{w},k}(S_1))< \varepsilon$, we are done as $S_1 \in S\mathcal{N}$, and if not, there exists $l \in 2, \ldots , m$ such that $\lambda(\mathrm{Fix}_{\mathrm{w},k}(S) \cap E_l) > 0$, and once again we may assume that $\lambda(\mathrm{Fix}_{\mathrm{w},k}(S) \cap E_2) > 0$. We can apply the whole argument again to $E_2$ and $(A_i)_{i \in \mathcal{I}_2}$ replacing $E_1$ and $(A_i)_{i \in \mathcal{I}_1}$ respectively, and we obtain $S_2 \in S\mathcal{N}$ satisfying $E_2 \subseteq \supp e_\mathrm{w}(S_2)$. The main point is that the construction ensures that $S_1(A_i) = S_2(A_i)$ for all $i \in \mathcal{I}_2$, so we still have $e_\mathrm{w}(S_2)(E'_q) \cap E'_q = \emptyset$ for all $q \leqslant p$. therefore, in finitely many iterations of the argument we obtain $S'$ satisfying the requirements: $S' \in S\mathcal{N}$ and $d_{u,X_k}(e_\mathrm{w}(S'),\id_X) = \lambda(\left\{ x \in X_k \mid e_\mathrm{w}(S')(x) \neq x \right\}) = \lambda(X_k \setminus  \mathrm{Fix}_{\mathrm{w},k}(S')) > 1- \varepsilon$.\\

	The only thing left to do is to link \textit{(9)} to the other equivalent conditions. The proof is exactly the same as \cite[Thm.~4.4]{CarderiLM2016}, as such it is detailed only for the sake of completeness.\\

	\noindent \textbf{Proof of \textit{(1)}$\implies$\textit{(9):}} The proof is done by contraposition. We choose a probability measure $\mu \in [\lambda]$. Assuming \textit{(9)} is not satisfied, there exists a $\tau_G$-neighbourhood $V$ of $e_G$ such that $\mu(\bigcup_{g \in V} \supp g) < 1 - \delta$ for some $\delta > 0$. By definition of $\tau_m$ on $\widetilde{[\mathcal{R}_G]}$ (see \Cref{section: Polish structures on orbit full groups} for the relevant notations) the set
	\[
	\mathcal{U} \coloneqq \left\{ f \in \widetilde{[\mathcal{R}_G]}  \; \middle| \; \mu(\left\{ x \in X : f(x) \notin V \right\}) > \frac{\delta}{2} \right\} \subseteq \Lzero(X,\lambda,(G,\tau_G))
	\]
	is a $\tau_m$-neighbourhood of the identity in $\widetilde{[\mathcal{R}_G]}$. However for any $f \in \mathcal{U}$ we have $d_\mu(\Phi(f),\id_X) < 1$ by construction, so $\Phi(f)$ is not aperiodic, and therefore the projection of $\mathcal{U}$ on $[\mathcal{R}_G]$ is a $\tau_m^o$-neighbourhood of $\id_X$ comprised of non-aperiodic elements, so \textit{(1)} does not hold.\\

	\noindent \textbf{Proof of \textit{(9)}$\implies$\textit{(5):}} Recall that for this implication we have a Polish topology on $X$, that the $G$-action is continuous and that $\lambda$ is locally finite with regards to the topology. 
	For a $\tau_G$-neighbourhood $V$ of $e_G$ in $G$, we say that $T \in \mathbb{G}$ is $V$-\textit{uniformly small} if there exists $f \in \Lzero(X,\lambda,(G,\tau_G))$ such that 
	\[
	\left\lbrace
	\begin{array}{l}
	f(X) \subseteq V \\
	\lambda\mbox{-a.e. on } X : T(x) = f(x) \cdot x
	\end{array}
	\right.
	\]
	and the main point is that for any $\tau_m^o$-neighbourhood $\mathcal{N}$ of $\id_X$ in $\mathbb{G}$, there exists a $\tau_G$-neighbourhood $V$ of $e_G$ in $G$ such that being $V$-uniformly small implies being in $\mathcal{N}$. We fix such sets $\mathcal{N}$ and $V$ as before. 
	Our goal is to prove that for any $A \in \MAlgf(X,\lambda)$ there exists $T \neq \id_X$ which is $V$-uniformly small with $\supp T \subseteq A$. By a maximality argument (see \Cref{prop: maximal element in MAlgf}), this will yield the existence of a $T \in \mathcal{N}$ with $\supp T = A$, concluding the proof.
	
	By outer regularity (see \textit{e.g.} \cite[Prop.~2.21]{HoareauLM2024}, this is where the local finiteness is used), we fix an open set $B \subseteq X$ such that $A \subseteq B$ and $\lambda(B \setminus A) < \lambda(A)$, in particular $B$ has finite measure. We also fix an open neighbourhood $U$ of $e_G$ in $G$ such that $U = U\inv$ and $U^2 \subseteq V$.
	\begin{claim2}{\label{ClaimCarderiLM}}
	There exists a countable family $(g_i)_{i \in \N}$ of elements of $U$ and an a.e. partition $(A_i)_{i \in \N}$ of $A$ such that for all $i \in \N$, $g_i(A_i)$ is a subset of $B$ that is disjoint from $A_i$.
\end{claim2}
\noindent
\textbf{Proof of Claim}: Let $(U_n)_{n \in \N}$ be a countable basis of open neighbourhoods of $e_G$ in $G$, and let
	\[
	S \coloneqq \bigcap_{n \in \N} \bigcup_{g \in U_n} \supp g.
	\]
By hypothesis, $S$ has full measure. Moreover since the action is continuous, for all $x \in S \cap B$ there is a $g \in U$ such that $g \cdot x \in B$ and $g \cdot x \neq x$. Again by continuity we can find an open neighborhood $W_x \subseteq B$ of $x$ such that $g(W_x)$ and $W_x$ are disjoint. We can now define the partition $(A_i)_{i \in \N}$ from an countable open subcover $(W_i)$ of $(W_x)_{x \in S \cap B}$, which exists by Lindelöf’s lemma, by setting $A_0 = W_0$ and $A_i = W_i \setminus (\cup_{j < i} W_j)$ for $i>0$. \hfill $\diamondsuit$\\
	
	There are now two cases to consider.\\
	
	\textbullet \ If there exists some $i \in \N$ such that $g_i(A_i)$ is not disjoint from $A$, we set $C_i \coloneqq g_i\inv (g_i(A_i) \cap A)$, and define $T \in \mathbb{G}$ by
	\[
	T(x) = 
	\left\lbrace
	\begin{array}{ll}
	g_i \cdot x & \mbox{if } x \in C_i, \\
	g_i\inv \cdot x & \mbox{if } x \in g_i(C_i),\\
	x & \mbox{otherwise.}
	\end{array}
	\right.
	\]
	
	\textbullet \ If for all $i \in \N$ we have that $g_i(A_i)$ is disjoint from $A$, then since $\lambda(B \setminus A) < \lambda(A)$ there exists $i \neq j$ in $\N$ such that $g_i(A_i) \cap g_j(A_j)$ is non-null. We then set $C_{i,j} \coloneqq g_i\inv(g_i(A_i) \cap g_j(A_j))$ and define $T \in \mathbb{G}$ by
	\[
	T(x) = 
	\left\lbrace
	\begin{array}{ll}
	g_j\inv g_i \cdot x & \mbox{if } x \in C_{i,j}, \\
	g_i\inv g_j \cdot x & \mbox{if } x \in g_j\inv g_i(C_{i,j}),\\
	x & \mbox{otherwise.}
	\end{array}
	\right.
	\]
	
	In both cases, the element $T \in \mathbb{G}$ is $V$-uniformly small, so it is in $\mathcal{N}$, it is non trivial, and it satisfies $\supp T \subseteq A$, which concludes the proof.
\end{proof}

We finally obtain an analogue of \cite[Cor.~4.5]{CarderiLM2016} thanks to \cite[Thm.~4.4]{HoareauLM2024}, \Cref{propOE} and \Cref{sepequivdnbeqrel}. the proof works without any changes.

\begin{cor}
	{\label{cor: dichotomy dense conjugacy class}}
	Let $G$ be a locally compact Polish group acting in an measure-preserving and ergodic manner on a standard $\sigma$-finite space $(X,\lambda)$. Then exactly one of the following holds:
	\begin{enumerate}[(i)]\setlength\itemsep{0em}
	\item $\mathcal{R}_G$ is a countable measure-preserving equivalence relation;
	\item The set $\mathrm{APER} \cap [\mathcal{R}_G]$ is $\tau_m^o$-dense in $[\mathcal{R}_G]$.
	\end{enumerate}
\end{cor}

\begin{rem}
	Without using condition \textit{(9)}, we can still deduce that \textit{(i)} implies that \textit{(ii)} does not hold. Indeed, if $\mathcal{R}_G$ is a countable measure-preserving equivalence relation, then $\tau_u = \tau_m^o$ by uniqueness of the Polish topology on $[\mathcal{R}_G]$ (\Cref{UniquePoltopoonfg}) and \Cref{sepequivdnbeqrel}, so it is not possible for \textit{(4)} to hold, by definition of $\tau_u$.
\end{rem}

Given the ambiguous role that $\mathbb{G}_f$ plays, it is natural to ask whether we can find a statement on its elements that is equivalent to the statements in \Cref{thm: category-density theorem}. It is not hard to see that the answer is positive if we can approximate aperiodic elements of $\mathbb{G}$ by elements of $\mathbb{G}_f$ that are aperiodic on their supports. The following is most likely well-known.

\begin{prop}
	{\label{prop: aper of G_f unif dense in APER}}
	Let $T$ be a conservative aperiodic (resp.~ergodic) element of $\Aut(X,\lambda)$. For any $\tau_u$-neighbourhood $N$ of $T$, there exists $T_f \in [T]_f \cap N$ such that $T_{f \restriction \supp T_f}$ is aperiodic (resp. ergodic).
\end{prop}

\begin{proof}
	We start by describing $T$ as a skyscraper as in step\textit{(i)} of the proof of \cite[Prop~2]{ChoksiKakutani1979}: there exists $B \subseteq X$ with $\lambda(B)<1$ and a sequence $(B_n)_{n \in \N}$ satisfying
	\[
	\left\lbrace
	\begin{array}{l}
	B_0 = B,\\
	\lambda(B_{n+1}) \leqslant \lambda(B_n),\\
	\lambda(B_n) \to_n 0,\\
	X = \bigsqcup_{n \in \N} B_n,
	\end{array}
	\right.
	\]
	as well as an aperiodic (resp.~ergodic) measure-preserving bijection $T_B$ of $B$ and injective measure-preserving maps $\pi_n : B_{n+1} \to B_{n}$ satisfying for any $n \in \N$:
	\[
	\left\lbrace
	\begin{array}{ll}
	T = \pi_n \inv & \mbox{ on } \pi_n(B_{n+1}),\\
	T = T_B \pi_1 \pi_2 \ldots \pi_{n-1} & \mbox{ on } B_n \setminus \pi_n(B_{n+1}).
	\end{array}
	\right.
	\]
	Once this observation is made, the proof is over, since the sequence of induced bijections on $\bigsqcup_{0 \leqslant i \leqslant n} B_i$ uniformly approaches $T$ by construction, and each term of the sequence is in $[T]$, has a support of finite measure, and is aperiodic (resp. ergodic) on its support. 
\end{proof}

\begin{rem}
	Thanks to \Cref{prop: aper of G_f unif dense in APER} we can add the following equivalent conditions to \Cref{thm: category-density theorem}:
	\begin{itemize}\setlength\itemsep{0em}
	\item[\textit{(1')}] The set of elements of $\mathbb{G}_f$ aperiodic on their support is $\tau_m^o$-dense in $\mathbb{G}$;
	\item[\textit{(2')}] The set of elements of $\mathbb{G}_f$ ergodic on their support is $\tau_m^o$-dense in $\mathbb{G}$.
	\end{itemize}
\end{rem}

\bibliographystyle{alphaurl}
\bibliography{bibliGeneral}

\bigskip
{\footnotesize

	\noindent
	{F.~Hoareau, \textsc{Université Paris Cité, Sorbonne Université, CNRS, IMJ-PRG, F-75013 Paris, France}}
	\par\nopagebreak \texttt{fabien.hoareau@imj-prg.fr}
}

\end{document}